\documentclass[reqno]{amsart}

\usepackage{mathabx}
\usepackage{tikz-cd}
\usepackage{verbatim}
\usepackage{fullpage,dsfont}
\usepackage{hyperref} 
\usepackage{parskip}
\usepackage[nobysame,alphabetic,initials,msc-links]{amsrefs}
\usepackage[final]{pdfpages}
\usepackage{enumitem}
\usepackage{multirow}
\usepackage{stmaryrd}
\usepackage{array}
\usepackage[bbgreekl]{mathbbol}			
\usepackage{stackengine} 

\makeatletter 
\def\thm@space@setup{%
 \thm@preskip=\parskip \thm@postskip=0pt
}
\def\th@remark{%
  \thm@headfont{\itshape}%
  \normalfont 
  \thm@preskip\parskip \thm@postskip=0pt
}
\makeatother

\DefineSimpleKey{bib}{how}
\DefineSimpleKey{bib}{mrclass}
\DefineSimpleKey{bib}{mrnumber}
\DefineSimpleKey{bib}{fjournal}
\DefineSimpleKey{bib}{mrreviewer}

\renewcommand{\PrintDOI}[1]{%
  \href{http://dx.doi.org/#1}{{\tt DOI:#1}}%
}
\renewcommand{\eprint}[1]{#1}
\BibSpec{book}{%
    +{}  {\PrintPrimary}                {transition}
    +{.} { \PrintDate}                  {date}
    +{.} { \textit}                     {title}
    +{.} { }                            {part}
    +{:} { \textit}                     {subtitle}
    +{,} { \PrintEdition}               {edition}
    +{}  { \PrintEditorsB}              {editor}
    +{,} { \PrintTranslatorsC}          {translator}
    +{,} { \PrintContributions}         {contribution}
    +{,} { }                            {series}
    +{,} { \voltext}                    {volume}
    +{,} { }                            {publisher}
    +{,} { }                            {organization}
    +{,} { }                            {address}
    +{,} { }                            {status}
    +{,} { \PrintDOI}                   {doi}
    +{,} { \PrintISBNs}                 {isbn}
    +{}  { \parenthesize}               {language}
    +{}  { \PrintTranslation}           {translation}
    +{;} { \PrintReprint}               {reprint}
    +{.} { }                            {note}
    +{.} {}                             {transition}
    +{}  {\SentenceSpace \PrintReviews} {review}
}
\BibSpec{article}{%
    +{}  {\PrintAuthors}                {author}
    +{,} { \textit}                     {title}
    +{.} { }                            {part}
    +{:} { \textit}                     {subtitle}
    +{,} { \PrintContributions}         {contribution}
    +{.} { \PrintPartials}              {partial}
    +{,} { }                            {journal}
    +{}  { \textbf}                     {volume}
    +{}  { \PrintDatePV}                {date}
    +{,} { \issuetext}                  {number}
    +{,} { \eprintpages}                {pages}
    +{,} { }                            {status}
    +{,} { \PrintDOI}                   {doi}
    +{,} { \eprint}        {eprint}
    +{}  { \parenthesize}               {language}
    +{}  { \PrintTranslation}           {translation}
    +{;} { \PrintReprint}               {reprint}
    +{.} { }                            {note}
    +{.} {}                             {transition}
    +{}  {\SentenceSpace \PrintReviews} {review}
}
\BibSpec{collection.article}{%
    +{}  {\PrintAuthors}                {author}
    +{,} { \textit}                     {title}
    +{.} { }                            {part}
    +{:} { \textit}                     {subtitle}
    +{,} { \PrintContributions}         {contribution}
    +{,} { \PrintConference}            {conference}
    +{}  {\PrintBook}                   {book}
    +{,} { }                            {booktitle}
    +{,} { \PrintDateB}                 {date}
    +{,} { pp.~}                        {pages}
    +{,} { }                            {publisher}
    +{,} { }                            {organization}
    +{,} { }                            {address}
    +{,} { }                            {status}
    +{,} { \PrintDOI}                   {doi}
    +{,} { \eprint}        {eprint}
    +{}  { \parenthesize}               {language}
    +{}  { \PrintTranslation}           {translation}
    +{;} { \PrintReprint}               {reprint}
    +{.} { }                            {note}
    +{.} {}                             {transition}
    +{}  {\SentenceSpace \PrintReviews} {review}
}
\BibSpec{misc}{%
  +{}{\PrintAuthors}  {author}
  +{,}{ \textit}      {title}
  +{.}{ }             {how}
  +{}{ \parenthesize} {date}
  +{,} { available at \eprint}        {eprint}
  +{,}{ available at \url}{url}
  +{,}{ }             {note}
  +{.}{}              {transition}
}
\usepackage{amssymb, amsfonts, amsxtra, amsmath}
\usepackage{mathrsfs}
\usepackage{mathdots}
\usepackage{wasysym}
\usepackage[all]{xy}
\usepackage{bbm}
\usepackage{calc}
\usepackage{accents}

\numberwithin{equation}{section}

\DeclareSymbolFontAlphabet{\mathbb}{AMSb}	
\DeclareSymbolFontAlphabet{\mathbbl}{bbold}	

\newtheorem{Theorem}{Theorem}[section]
\newtheorem*{Theorem*}{Theorem}
\newtheorem{Def}[Theorem]{Definition}
\newtheorem*{Def*}{Def}
\newtheorem{Lem}[Theorem]{Lemma}
\newtheorem{Prop}[Theorem]{Proposition}
\newtheorem{Cor}[Theorem]{Corollary}

\newtheorem{Rem}[Theorem]{Remark}

\newtheorem{Exa}[Theorem]{Example}

\newcommand\bp{\begin{proof}}
\newcommand\ep{\end{proof}}

\mathchardef\mhyph="2D

\DeclareMathOperator{\ad}{\mathrm{ad}}

\DeclareMathOperator{\id}{\mathrm{id}}

\DeclareMathOperator{\Imm}{\mathrm{Im}}

\DeclareMathOperator{\Ker}{\mathrm{Ker}}

\DeclareMathOperator{\cb}{\mathrm{cb}}
\DeclareMathOperator{\CB}{\mathcal{CB}}
\DeclareMathOperator{\CP}{\mathcal{CP}}
\DeclareMathOperator{\M}{M}
\DeclareMathOperator{\vN}{vN}

\newcommand{\NMod}{\operatorname{NMod}}
\newcommand{\F}{\mathcal{F}}
\newcommand{\mcF}{\mathcal{F}}
\newcommand{\mcH}{\mathcal{H}}

\newcommand{\mcK}{\mathcal{K}}

\newcommand{\C}{\mathbb{C}}
\newcommand{\Z}{\mathbb{Z}}
\newcommand{\G}{\mathbb{G}}
\newcommand{\Hh}{\mathbb{H}}

\newcommand{\N}{\mathbb{N}}

\newcommand{\R}{\mathbb{R}}

\newcommand{\Mod}{\mathrm{Mod}}
\newcommand{\eps}{\varepsilon}

\newcommand{\Prod}{\prod}
\newcommand{\ovot}{\bar{\otimes}}
\newcommand{\oon}{\operatorname}

\newcommand{\I}{1}
\newcommand{\wh}{\widehat}
\newcommand{\vp}{\varphi}

\DeclareMathOperator{\B}{B}
\DeclareMathOperator{\A}{A}
\newcommand{\LL}{L}
\DeclareMathOperator{\lin}{span}
\DeclareMathOperator{\Sat}{Sat}
\DeclareMathOperator{\Fix}{Fix}

\newcommand{\uu}{\mathrm{U}}

\newcommand{\vv}{\mathrm{V}}
\newcommand{\Vv}{\mathds{V}}
\newcommand{\vV}{\text{\reflectbox{$\Vv$}}\:\!}

\newcommand{\ww}{\mathrm{W}}
\newcommand{\WW}{{\mathds{V}\!\!\text{\reflectbox{$\mathds{V}$}}}}
\newcommand{\Ww}{\mathds{W}}
\newcommand{\wW}{\text{\reflectbox{$\Ww$}}\:\!}
\newcommand{\Uu}{\mathds{U}}
\newcommand{\uU}{\text{\reflectbox{$\Uu$}}\:\!}
\renewcommand{\H}{\mathbb{H}}

\allowdisplaybreaks

\newcommand{\pten}{\hat{\otimes}}
\newcommand{\oten}{\bar{\otimes}}
\newcommand{\la}{\langle}
\newcommand{\ra}{\rangle}
\newcommand{\ten}{\otimes}
\newcommand{\om}{\omega}
\newcommand{\BH}{B(\mathcal{H})}
\newcommand{\BK}{B(\mathcal{K})}
\providecommand{\normof}[1]{\lVert#1\rVert}

\author{Jason Crann, Joeri De Ro, Jacek Krajczok}

\begin{document}
\title{Actions of quantum groups on dual operator spaces and their crossed products}

\address{School of Mathematics and Statistics, Carleton University, Ottawa, ON, Canada H1S
5B6}
\email{jasoncrann@cunet.carleton.ca}

\address{Institute of Mathematics of the Polish Academy of Sciences, Warsaw, Poland}
\email{jdero@impan.pl}

\address{Vrije Universiteit Brussel, Department of Mathematics and Data Science, Brussels, Belgium}

\email{jacek.krajczok@vub.be}

\begin{abstract}
 We study the category of dual operator spaces equipped with an action of a locally compact quantum group $\G$. The Fubini crossed product functor $-\rtimes^\mcF \G$ and the weak$^*$-crossed product functor $-\bar{\rtimes}\G$ are shown to be equal if and only if $\G$ has the approximation property of Haagerup and Kraus. Using the natural isomorphism $-\rtimes^\mcF\G\cong {}_{L^1(\G)}\CB(B(L^2(\G))_*, -)$, this leads to a characterization of the approximation property of $\G$ via an $L^1(\G)$-module approximation property for $B(L^2(\G))_*$. Finally, exactness of the Fubini crossed product functor is investigated and related to amenability properties of $\G$.
\end{abstract}
\maketitle

\section{Introduction}

Throughout the literature, the notion of an action of a locally compact (quantum) group $\G$ on an operator space $X$ has mainly been studied from two points of view:
\begin{enumerate}[noitemsep]
    \item Via the notion of an $L^1(\G)$-operator module, which is an operator space $X$ together with a complete contraction $\rhd: L^1(\G)\hat{\otimes}X\to X$ satisfying the module property \cites{Rua95, Rua96, RX97, Spr02, Woo02, Ari04, ARS04,  Daw10, CLR15, CN16, Cra17, CT17, Cra19, Cra21}.
    \item Via the notion of a $\G$-operator space, which is an operator space $X$ together with a complete isometry $\alpha: X\to X\ovot_\mcF L^\infty(\G)$ satisfying the coaction property $(\id \otimes \Delta)\alpha= (\alpha\otimes \id)\alpha$ \cites{Ham11, SS17, An21, CN22, An23, DH24, DR26}.
\end{enumerate}
These above two approaches of studying actions of locally compact quantum groups on operator spaces are intimately related: every $\G$-operator space $(X, \alpha)$ leads to an $L^1(\G)$-operator module via \begin{equation}\label{connnection}
    \omega\rhd x =  (\id \otimes \omega)\alpha(x), \quad x\in X, \quad \omega\in L^1(\G).
\end{equation}  Conversely, for every $L^1(\G)$-operator module $(X, \rhd)$, there is a unique completely contractive map $\alpha: X\to X\ovot_\mcF L^\infty(\G)$ satisfying \eqref{connnection}, but it should be emphasized that $\alpha$ need not be completely isometric in general. We write ${}_{L^1(\G)}\Mod$ (resp.\ ${}_{L^1(\G)}\NMod$) for the category of (resp.\ dual) $L^1(\G)$-operator
modules and designate the superscript $\|\cdot\|$
to indicate the full subcategory consisting of modules whose associated
coaction $\alpha$ is completely isometric. See Section \ref{subsecFSF} for a more thorough discussion on these categories.

A central tool in the study of an action of a locally compact quantum group $\G$ on an operator space is the crossed product construction, which aims to encode both the operator space and the $\G$-action on it in one single operator space. There are two relevant constructions:
\begin{enumerate}[nosep]
    \item Given $X\in {}_{L^1(\G)}\Mod$ with induced coaction $\alpha: X\to X\ovot_\mcF L^\infty(\G)$, we define the \emph{Fubini crossed product}
$$X\rtimes_\alpha^\mcF\G := \{z\in X\ovot_\mcF B(L^2(\G))\mid (\alpha\otimes \id)(z) = (\id \otimes \Delta_l)(z)\}.$$
This construction can be interpreted as the cotensor product $X\stackrel{\G}\square B(L^2(\G))$ or alternatively as the fixed point space of a natural action $X\ovot_\mcF B(L^2(\G))\curvearrowleft \G$. The notion of the Fubini crossed product has appeared in the literature, albeit only in the setting of $\G$-operator spaces \cites{Ham11,An21, CN22, An23, DH24,DR26}.
\item If $(X, \alpha)\in {}_{L^1(\G)}\NMod$, one can also define the \emph{weak$^*$-crossed product}
$$
X\bar{\rtimes}_\alpha\G := [\alpha(X)(1\otimes L^\infty(\check{\G}))]^{w *}\subseteq X\ovot_\mcF B(L^2(\G)).
$$
\end{enumerate}
It is easy to see that $X\bar{\rtimes}_\alpha\G\subseteq X\rtimes_\alpha^\mcF\G$ for all $(X, \alpha)\in {}_{L^1(\G)}\NMod$. The possibility that this inclusion is strict is directly related to an analytic property of the locally compact quantum group $\G$, which is called the \emph{approximation property (AP)}:

\textbf{Theorem \ref{main}.} \textit{Let $\G$ be a locally compact quantum group. Then $\G$ has the AP if and only if $X\bar{\rtimes}_\alpha\G= X\rtimes_\alpha^\mcF\G$ for all $(X, \alpha)\in{}_{L^1(\G)}\NMod^{\|\cdot\|}$.
}

The AP was introduced by Haagerup and Krauss in the case of classical groups \cite{HK94} and was afterwards generalized and studied in the quantum setting \cites{KR99, Cra19, DKV24} of Kac algebras and general locally compact quantum groups. Loosely speaking, the AP can be seen as a very weak form of amenability. It carries important information about the quantum group and still has useful implications to properties of $\G$ concerning harmonic analysis and dynamical systems. Our work can be seen as further exploration of such relations. In the case of classical locally compact groups, the forward implication of Theorem \ref{main} was first proven in \cite{CN22}, while its converse was proven in \cite{An23}. For discrete quantum groups, the forward implication was proven in \cite{CKN25}. 

The Fubini crossed product $-\rtimes^\mcF\G$ is a functor from the category of $L^1(\G)$-operator modules to the category of operator spaces in a natural way. A crucial conceptual observation is that it is \emph{representable}: if we view the predual $B(L^2(\G))_*$ with its natural $L^1(\G)$-operator module structure, then there is a natural isomorphism
\begin{equation}\label{representthis}
    -\rtimes^\mcF \G \cong {}_{L^1(\G)}\CB(B(L^2(\G))_*, -)
\end{equation}
of (covariant) functors. This places the Fubini crossed product directly within the framework of homological algebra. Its homological properties are then reflected in $L^1(\G)$-module properties of $B(L^2(\G))_*$. 

Naturally, one then wishes to investigate if the AP of $\G$ can be seen as an appropriate $L^1(\G)$-module approximation property of $B(L^2(\G))_*$. Theorem \ref{main} therefore motivates the task to characterize the weak$^*$-crossed product $X\bar{\rtimes}_\alpha\G$ inside the Fubini crossed product $X\rtimes_\alpha^\mcF\G\cong {}_{L^1(\G)}\CB(B(L^2(\G))_*,X)$ for $X\in {}_{L^1(\G)}\NMod^{\|\cdot\|}$. To do this, let us write ${}_{L^1(\G)}\mathcal{DF}(B(L^2(\G))_*,X)$ for the space of completely bounded $L^1(\G)$-linear maps that have finite $L^1(\G)$-rank (Definition \ref{d1})  and which are decomposable via completely positive maps in an appropriate way (Definition \ref{d2}). We then find:

\textbf{Theorem \ref{crossedproductfiniterank}.} \textit{For any $(X, \alpha)\in {}_{L^1(\G)} \operatorname{NMod}^{\|\cdot\|}$, we have
        $X\bar{\rtimes}_\alpha \G\subseteq \overline{{}_{L^1(\G)}\mathcal{DF}(B(L^2(\G))_*,X)}^{w*}.$
The converse inclusion holds if $\G$ is strongly regular.}

To prove Theorem \ref{crossedproductfiniterank}, we use an equivariant version of the Stinespring theorem (Theorem \ref{equivariant Stinespring}), which is of independent interest. Combining Theorem \ref{main} and Theorem \ref{crossedproductfiniterank}, leads to the following:

\textbf{Definition \ref{DAP}.} \textit{Let $(M, \alpha)$ be a left $\G$-$W^*$-algebra. We say that $M_*\in {}_{L^1(\G)}\Mod$ has the \emph{$L^1(\G)$-decomposable approximation property} ($L^1(\G)$-DAP) if 
$\overline{{}_{L^1(\G)}\mathcal{DF}(M_*, X)}^{w*} = {}_{L^1(\G)}\CB(M_*, X)$
for every $X\in {}_{L^1(\G)}\NMod^{\|\cdot\|}$.}

\textbf{Theorem \ref{BDAP}.}\textit{
     Let $\G$ be a locally compact quantum group. If $\G$ has the AP, then $B(L^2(\G))_*$ has the $L^1(\G)$-DAP. The converse holds if $\G$ is strongly regular.}

Thus, the $L^1(\G)$-DAP provides the desired module-approximation property. To further strengthen the analogy, we also prove:

\textbf{Theorem \ref{LDAP}.}\textit{
     $\G$ has the AP if and only if $L^1(\hat{\G})$ has the $L^1(\hat{\G})$-DAP.}

To prove Theorem \ref{LDAP}, our main strategy is to understand exactly when ${}_{L^1(\G)}\mathcal{DF}(L^1(\G),X)= \alpha(X)$ for $(X, \alpha)\in {}_{L^1(\G)}\NMod^{\|\cdot\|}$. We prove that this happens precisely when the coaction $\alpha: X\to X\ovot_\mcF L^\infty(\G)$ admits a \emph{universal lift} $\alpha^u: X\to X\ovot_\mcF C_0^u(\G)^{**}$ (Proposition \ref{prop5}) (see Section \ref{SectionDAP} for precise definitions). In the setting of $\G$-$W^*$-algebras, the notion of the universal lift was first introduced and studied in \cite{DCK24}; in particular it then always exists. For general operator spaces however, the question of its existence for all modules is intimately related to the AP, as the following result shows:

\textbf{Theorem \ref{thm1}.}\textit{
Let $\G$ be a locally compact quantum group. The following conditions are equivalent:
\begin{enumerate}[noitemsep]
\item $\G$ has the AP.
\item Every $X\in {}_{L^1(\hat\G\times \hat\G)}\NMod^{\|\cdot\|}$ admits the universal lift.
\end{enumerate}}

Finally, we also treat the question of $1$-exactness of the Fubini crossed product functor. Through the natural isomorphism \eqref{representthis}, this question naturally leads to the investigation of $L^1(\G)$-projectivity of $B(L^2(\G))_*$, which turns out to be equivalent with finiteness of $\G$ (Corollary \ref{projectivitytraceclass}). When passing to appropriate subcategories of ${}_{L^1(\G)}\Mod$, $1$-exactness of the Fubini crossed product functor turns out to be equivalent with amenability of $\G$:

\textbf{Theorem \ref{exactnness}.} \textit{ Let $\G$ be a locally compact quantum group.
    \begin{enumerate}[noitemsep]
        \item $-\rtimes^\mcF \G$ is $1$-exact on ${}_{L^1(\G)}\Mod$ if and only if $\G$ is finite.
        \item $-\rtimes^\mcF \G$ is $1$-exact on ${}_{L^1(\G)}\NMod$ if and only if $\G$ is amenable.
        \item If $\G$ is amenable, then $-\rtimes^\mcF \G$ is $1$-exact on ${}_{L^1(\G)}\NMod^{\|\cdot\|}$. The converse holds if $\hat{\G}$ is amenable.
    \end{enumerate}}

\section{Preliminaries}\label{SectionPreliminaries}

All vector spaces in this paper are defined over the field $\C$. The symbol $\odot$ denotes the (algebraic) tensor product (over $\C$). We assume that inner products of pre-Hilbert spaces are anti-linear in the first variable. Given a subset $S$ of a normed linear space $V$, we write $[S]$ for the norm-closure of the linear span of $S$. More generally, if $(V, \tau)$ is a topological vector space and $S\subseteq V$, we write $[S]^\tau$ for the $\tau$-closure of the linear span of $S$ inside $V$. The multiplier $C^*$-algebra of a $C^*$-algebra $C$ is denoted by $\M(C)$. We let $\ten$ simultaneously denote the tensor product of Hilbert spaces and minimal tensor product of $C^*$-algebras; which one being clear from context. Given weak$^*$-closed operator spaces $X\subseteq\BH$ and $Y\subseteq\BK$, the weak$^*$-closure of $X\odot Y$ inside $B(\mathcal{H}\ten\mathcal{K})$ is denoted $X\oten Y$. In particular, $M\oten N$ denotes the von Neumann algebra tensor product of $M$ and $N$. Operator spaces are always assumed to be complete. The operator space projective tensor product of operator spaces $X$ and $Y$ is denoted $X\pten Y$.

\subsection{Operator spaces and modules}\label{operatormodules} Let $A$ be a completely contractive Banach algebra, that is, a Banach algebra $A$ with an operator space structure for which multiplication extends to a complete contraction $m_A:A\pten A\rightarrow A$. Equivalently,
$$\normof{[a_{ij}b_{kl}]}_{mn}\leq\normof{[a_{ij}]}_{m}\normof{[b_{kl}]}_n, \ \ \ [a_{ij}]\in M_m(A), \quad [b_{kl}]\in M_n(A).$$
An operator space $X$ is a left $A$-\textit{operator module} if it is a left Banach $A$-module such that the module map $m_X:A\pten X\rightarrow X$ is completely contractive. We will often write $m_X(a\otimes x)= a\rhd x$ for $a\in A$ and $x\in X$.  Right $A$-operator modules are defined similarly, and the module action is then denoted by $\lhd$. We denote the category of left (resp. right) $A$-operator modules with completely bounded left (resp. right) $A$-linear maps by ${}_A\Mod$ (resp. $\Mod_A$). The morphism spaces are denoted ${}_A\CB(X,Y)$ (resp. $\CB_A(X,Y)$).
We denote by ${}_{A}\NMod$ the subcategory of ${}_{A}\Mod$ with objects $X\in {}_{A}\Mod$ for which $X$ is a dual operator space (with fixed predual, determining its weak$^*$-topology) such that the maps $X\ni x \mapsto a\rhd x \in X$ are weak$^*$-continuous for all $a\in A$. For example, if $X\in\Mod_A$, then $X^*\in{}_A\NMod$ via
$$\la a\rhd f,x\ra=\la f,x\lhd a\ra, \ \ \ a\in A, \ f\in X^*, \ x\in X.$$  If $X\in \NMod_A$, we write $X_*$ for the space of weak$^*$-continuous functionals $X\to \C$, which is an $A$-submodule of $X^*$.
The morphisms of ${}_{A}\NMod$ are the weak$^*$-continuous completely bounded $A$-linear maps, denoted ${}_{A}\CB^\sigma(X,Y)$. Similarly for $\NMod_{A}$.

The identification $ A_+:= A  \oplus_1\C$ turns the unitization of $ A  $ into a unital completely contractive Banach algebra, and any $X\in{}_{A}\Mod$ becomes a \emph{unital} $A_+$-operator module via the extended action
$$(a,\lambda)\rhd x=a\rhd x +\lambda x, \ \ \ a\in A, \ \lambda\in\C, \ x\in X.$$
Here, unitality means that the unit $(0,1)$ of $A_+$ acts as the identity on $X$. When $A=M_*$ is the predual of a von Neumann bialgebra $(M, \Delta)$, the notation $A_+$ should not be confused with the normal positive linear functionals on $M$, which we will denote by $M_*^+$.

Let $X\in\Mod_A$ and $Y\in {}_A\Mod$. The \emph{$ A$-module tensor product} of $X$ and $Y$ is the quotient space $X\pten_{ A  }Y:=(X\pten Y)/N$, where
$$N=[(x\lhd a)\ten y- x\otimes (a\rhd y) \mid x\in X, \ y\in Y, \ a\in A].$$
There are canonical completely isometric isomorphisms $(X\pten_A Y)^*\cong\CB_A(X,Y^*)\cong {}_A\CB(Y,X^*)$ \cite[Corollary 3.5.10]{BLM04}. If, in addition, $B$ is another completely contractive Banach algebra and $Y\in{}_A\Mod_B$, meaning it is both a left $A$-operator and right $B$-operator module and the module actions commute, then for any $Z\in {}_B\Mod$ there is a completely isometric isomorphism 
\begin{equation}\label{associativity}
    (X\pten_{A}Y)\pten_{B}Z\cong X\pten_{A}(Y\pten_{B}Z)
\end{equation}
mapping $(x\ten_A y)\ten_B z$ to $x\ten_A (y \ten_B z)$ for all $x\in X$, $y\in Y$, and $z\in Z$ \cite[Theorem 3.4.10]{BLM04}, where $x\ten_A y$ denotes the image of $x\ten y$ in the module tensor product $X\pten_A Y$. (The proof follows from the universal properties of tripartite module tensor products (see \cite[Theorem 2.6]{BMP00} for the analogous proof for module Haagerup tensor products).)

We call $X\in {}_A \Mod$ \emph{$A$-projective} if for every $Y,Z \in {}_A \Mod$, every $A$-linear complete quotient map $q: Y \twoheadrightarrow Z$, every morphism  $\varphi: X \to Z$ and every $\epsilon > 0$, there exists a morphism
$\tilde{\varphi}_\epsilon: X \to Y$ such that $\|\tilde{\varphi}_\epsilon\|_{\cb}< \|\varphi\|_{\cb}+ \epsilon$ and $q\circ \tilde{\varphi}_\epsilon= \varphi$, i.e., the following diagram commutes:
\begin{equation*}
\begin{tikzcd}
                     &Y \arrow[d, two heads, "q"]\\
X \arrow[ru, dotted, "\widetilde{\varphi}_\varepsilon"] \arrow[r, "\varphi"] &Z
\end{tikzcd}
\end{equation*}

We call $X\in {}_A \Mod$ \emph{$A$-injective} if for every $Y,Z \in {}_A \Mod$, every completely isometric $A$-linear map $\iota:Y\hookrightarrow Z$, and every morphism $\varphi:Y\rightarrow X$, there exists a morphism $\widetilde{\varphi}:Z\rightarrow X$ such that $\normof{\widetilde{\varphi}}_{\cb}=\normof{\varphi}_{\cb}$ and $\widetilde{\varphi}\circ \iota=\varphi$, that is, the following diagram commutes:
\begin{equation*}
\begin{tikzcd}
Z \arrow[rd, dotted, "\widetilde{\varphi}"]\\
Y \arrow[u, hook, "\iota"] \arrow[r, "\varphi"] &X
\end{tikzcd}
\end{equation*}

Our notions of projective and injective modules are the operator module analogues of extremely projective/injective Banach modules from \cite{Hel11,Hel13}. They coincide with the 1-projective/1-injective modules from \cite{Cra17,Cra21}, and recover the usual notions from operator space theory \cite{Bl} when $A=\mathbb{C}$. In that vein, the proof of \cite[Theorem 3.5]{Bl} can be adapted to the operator module setting to show that $X\in {}_{A}\Mod$ is $A$-projective implies $X^*\in \Mod_A$ is $A$-injective.

A sequence of the form
$$\begin{tikzcd}
0 \arrow[r] & X \arrow[r, "\iota"] & Y \arrow[r, "\phi"] & Z
\end{tikzcd}$$
in ${}_A\Mod$ is \textit{1-exact} if $\Ker(\phi)= \Imm(\iota)$ and $\iota$ is a complete isometry. A short exact sequence 
$$\begin{tikzcd}
0 \arrow[r] & X \arrow[r, "\iota"] & Y \arrow[r, "\phi"] & Z \arrow[r] & 0
\end{tikzcd}$$
is \textit{$1$-exact} if $\Ker(\phi)= \Imm(\iota)$, $\iota$ is a complete isometry and $\phi$ is a complete quotient map. The latter metric strengthening of a short exact sequence is commonly featured in operator space theory as it captures important properties such as nuclearity \cite[Theorem 14.6.1]{ER00}, 1-exactness \cite[Theorem 14.4.1]{ER00} (see also \cite[Remark \S1]{Pis95}) and the local lifting property \cite[Theorem 3.3]{Dong09}. 

If $B$ is another completely contractive Banach algebra, a covariant linear functor $\mathcal{F}:\Mod_A\to\Mod_B$ is \textit{left 1-exact} (resp. \textit{1-exact}) if for every 1-exact sequence of the form
$$\begin{tikzcd}
0 \arrow[r] & X \arrow[r, "\iota"] & Y \arrow[r, "\phi"] & Z
\end{tikzcd}$$
(resp. short 1-exact sequence), the image sequence
$$\begin{tikzcd}
0 \arrow[r] & \mathcal{F}(X) \arrow[r, "\mathcal{F}(\iota)"] & \mathcal{F}(Y) \arrow[r, "\mathcal{F}(\phi)"] & \mathcal{F}(Z)
\end{tikzcd}$$
is 1-exact (resp. short 1-exact). We call $X\in {}_A \Mod$ $\emph{$A$-flat}$ if the functor $(-)\pten_A X:\Mod_A\to\Mod_{\mathbb{C}}$ is 1-exact. Note that $(-)\pten_A X $ automatically preserves complete quotient maps (by a factorization argument and the analogous operator space property of $(-)\pten X$). Standard arguments show that $X$ is $A$-flat if and only if $X^*$ is injective in $\Mod_A$ (see, e.g., \cite{Hel93}*{Theorem VII.1.42} or \cite[Proposition 3.5.9]{Wo99} for the relative homology case). 

We call $X\in {}_A \operatorname{NMod}$  \emph{weak$^*$-injective} if for every weak$^*$-continuous $A$-linear complete isometry $\iota: Y\hookrightarrow Z$, every weak$^*$-continuous completely bounded $A$-linear map $\phi: Y \to X$ and every $\epsilon > 0$, there exists a weak$^*$-continuous $A$-linear completely bounded map $\phi_\epsilon: Z \to X$ such that $\|\phi_\epsilon\|_{\operatorname{cb}}< \|\phi\|_{\cb} + \epsilon$ and such that $\phi_\epsilon\circ \iota= \phi$. Standard duality arguments show that $X\in {}_A \Mod$ is $A$-projective if and only if $X^*\in\NMod_{A}$ is weak$^*$-injective. Note that $A$-injectivity in ${}_A\NMod$, defined similarly to that in ${}_A\Mod$ but with additional continuity assumptions on the morphisms, is a strictly stronger property. Indeed, $A=\C$ is weak$^*$-injective in ${}_{\C}\NMod$, but not injective in ${}_{\C}\NMod$ (if it were, every closed subspace of every Banach space would be proximinal). 

The following Proposition is well-known in the setting of relative Banach (resp. operator module) homology (see \cite[Proposition III.1.14]{Hel89} (resp. \cite[Theorem 3.3.4]{Wo99})). We include a proof in our setting for completeness.

\begin{Prop}\label{exactness homfunctor} Let $A$ be a completely contractive Banach algebra and let $M\in {}_A \Mod$.
\begin{enumerate}[noitemsep]
    \item The functor ${}_A\CB(M,-)$ is left 1-exact. 
\item ${}_A \CB(M,-)$ preserves complete quotient mappings if and only if $M$ is $A$-projective.
\item  $M$ is $A$-projective if and only if the covariant functor ${}_A \CB(M,-)$ is $1$-exact.
\end{enumerate}
\end{Prop}
\begin{proof} (1) Let $\begin{tikzcd}
0 \arrow[r] & X \arrow[r, "\iota"] & Y \arrow[r, "\phi"] & Z
\end{tikzcd}$
be 1-exact, and let $\iota_*:={}_A\CB(M,\iota)$, and $\phi_*:={}_A\CB(M,\phi)$. It is clear that $\iota_*$ is also a complete isometry, so it suffices to show that $\Ker(\phi_*)= \Imm(\iota_*)$. If $t \in {}_A \CB(M,Y)$ with $\phi_*(t) = 0$, then $\phi(t(m)) = 0$ for all $m\in M$, and thus $t(m)= \iota(x_m)$ for some unique $x_m\in X$. We can then well-define the map
    $s: M \to X: m \mapsto x_m.$
    By construction, it satisfies $\iota \circ s = t$. From this, it follows that $s$ is $A$-linear and $\|s\|_{\cb}=\|t\|_{\cb} <\infty$.

    (2)  Assume first that ${}_A \CB(M,-)$ preserves complete quotient maps. Let $Y,Z \in {}_A \Mod$, $q: Y \to Z$ an $A$-linear complete quotient map, $\varphi: M \to Z$ an $A$-linear completely bounded map and $\epsilon > 0$. By assumption, $q_*: {}_A \CB(M,Y)\to {}_A \CB(M,Z)$ is also a (complete) quotient map. Since $\varphi \in {}_A \CB(M,Z)_{<\|\varphi\|_{\cb} + \epsilon}$, it follows that there exists $\tilde{\varphi}_\epsilon\in {}_A \CB(M,Y)_{< \|\varphi\|_{\cb}+\epsilon}$ with $\varphi= q_*(\tilde{\varphi}_\epsilon)$, i.e. $q\circ \tilde{\varphi}_\epsilon=\varphi$ and $\|\tilde{\varphi}_\epsilon\|< \|\varphi\|_{\cb}+ \epsilon$. 

    Conversely, assume that $M$ is $A$-projective and that $q: Y \to Z$ is an $A$-linear complete quotient mapping. We need to show that for every $n \ge 1$, the map $(q_*)_n: M_n({}_A \CB(M, Y))\to M_n({}_A \CB(M,Z))$ is a quotient map. Note however that we have a commutative diagram 
$$\begin{tikzcd}
{M_n({}_A \CB(M,Y))} \arrow[rr, "(q_*)_n"] \arrow[d, "\cong"] &  & {M_n({}_A \CB(M,Z))} \arrow[d, "\cong"] \\
{{}_A \CB(M, M_n(Y))} \arrow[rr, "(q_n)_*"]                   &  & {{}_A \CB(M, M_n(Z))}                  
\end{tikzcd}$$
where the vertical arrows are completely isometric isomorphisms. It thus suffices to show that $(q_n)_*$ is a quotient map for every $n \ge 1$. 

Given $\varphi \in {}_A \CB(M, M_n(Z))$ with $\|\varphi\|_{\cb}< 1$ and noting that 
$q_n: M_n(Y)\to M_n(Z)$ is a complete quotient map, the $A$-projectivity of $M$ (applied with $\epsilon:= 1-\|\varphi\|_{\cb}$) yields $\psi \in {}_A \CB(M, M_n(Y))$ with $q_n \circ \psi = \varphi$ and $\|\psi\|_{\cb}< \|\varphi\|_{\cb}+\epsilon=1$, whence $(q_n)_*$ is a quotient map.

(3) This follows immediately from $(1)$ and $(2)$.
\end{proof}

\begin{Prop}\label{flatness}
     Let $A$ be a completely contractive Banach algebra and $M \in {}_A \Mod$. The functor ${}_A \CB(M,-)$ is $1$-exact on ${}_A \NMod$ if and only if $M$ is flat in ${}_A \Mod$.
\end{Prop}
\begin{proof}
Assume first that $M$ is flat in ${}_A \Mod$. Given a $1$-exact sequence $$\begin{tikzcd}
0 \arrow[r] & X \arrow[r, "\iota"] & Y \arrow[r, "\phi"] & Z \arrow[r] & 0
\end{tikzcd}$$
in ${}_{A} \operatorname{NMod}$, the fact that $M$ is flat in ${}_A\Mod$ entails that
$$ \begin{tikzcd}
0 \arrow[r] & Z_*\pten_{A} M \arrow[r, "\phi_*\ten\id"] & Y_*\pten_{A} M \arrow[r, "\iota_*\ten\id"] & X_*\pten_{A} M \arrow[r] & 0
\end{tikzcd}$$
is $1$-exact. Consequently, its dual sequence 
$$ \begin{tikzcd}
0 \arrow[r] & {}_{A} \CB(M, X) \arrow[r] & {}_{A} \CB(M, Y) \arrow[r] & {}_{A} \CB(M, Z) \arrow[r] & 0
\end{tikzcd}$$
is 1-exact. 

Conversely, if ${}_A \CB(M, -)$ is 1-exact on ${}_{A} \operatorname{NMod}$, then  for any 1-exact sequence 
$$\begin{tikzcd}
0 \arrow[r] & X \arrow[r, "\iota"] & Y \arrow[r, "\phi"] & Z \arrow[r] & 0
\end{tikzcd}$$
in $\operatorname{Mod}_{A}$, the sequence
$$ \begin{tikzcd}
0 \arrow[r] & {}_{A} \CB(M, Z^*) \arrow[r] & {}_{A} \CB(M, Y^*) \arrow[r] & {}_{A} \CB(M, X^*) \arrow[r] & 0
\end{tikzcd}$$
is 1-exact. The pre-adjoint sequence
$$ \begin{tikzcd}
0 \arrow[r] & X\pten_{A} M \arrow[r, "\iota \otimes \id"] & Y\pten_{A} M \arrow[r, "\phi\ten\id"] & Z\pten_{A} M \arrow[r] & 0
\end{tikzcd}$$
is 1-exact, entailing the $A$-flatness of $M$.    
\end{proof}

For operator spaces $X\subseteq\BH$ and $Y\subseteq\BK$ with $Y$ weak$^*$-closed, their \emph{Fubini tensor product}  \cite{Ham11} is defined as
$$X\oten_{\mathcal{F}} Y:=\{z\in B(\mathcal{H}\ten\mathcal{K})\mid (\id\ten\om)z\in X, \ (\rho\ten\id)z\in Y \ \forall \ \om\in\BK_*, \ \rho\in\BH_*\}.$$
Let $X_i\subseteq B(\mathcal{H}_i)$, and $Y_i\subseteq B(\mathcal{K}_i)$, $i=1,2$, be operator spaces with $Y_1$ and $Y_2$ weak$^*$-closed. If $\phi:X_1\to X_2$ and $\psi:Y_1\to Y_2$ are complete contractions with $\psi$ weak$^*$-continuous, the amplifications 
$$(\phi\ten\id_{B(\mathcal{K}_i)}):X_1\oten_{\mathcal{F}} B(\mathcal{K}_i)\to X_2\oten_{\mathcal{F}} B(\mathcal{K}_i), \ \ \ \ 
(\id_{B(\mathcal{H}_i)}\ten\psi):B(\mathcal{H}_i)\oten_{\mathcal{F}} Y_1\to B(\mathcal{H}_i)\oten_{\mathcal{F}} Y_2$$
are well-defined complete contractions such that
$$(\phi\ten\id_{B(\mathcal{K}_2)})\circ(\id_{B(\mathcal{H}_1)}\ten\psi)|_{X_1\oten_{\mathcal{F}} Y_1}=(\id_{B(\mathcal{H}_2)}\ten\psi)\circ(\phi\ten\id_{B(\mathcal{K}_1)})|_{X_1\oten_{\mathcal{F}} Y_1}$$
defines a complete contraction $\phi\ten\psi:=X_1\oten_{\mathcal{F}} Y_1\to X_2\oten_{\mathcal{F}} Y_2$ \cite[Proposition 1.1]{Ham11} (see also \cite[3.8, 3.9]{Ham82}).
If, in addition, $\phi$ and $\psi$ are surjective complete isometries, then $\phi\ten\psi$ is a completely isometric isomorphism $X_1\oten_{\mathcal{F}} Y_1\cong X_2\oten_{\mathcal{F}} Y_2$ \cite[Proposition 1.1]{Ham11}. It follows that the Fubini tensor product is (1) independent of the inclusions $X\subseteq\BH$ and $Y\subseteq\BK$ (up to completely isometric isomorphism), and (2) injective in the sense that if $\phi: X_1\hookrightarrow X_2$ and $\psi:Y_1\hookrightarrow Y_2$ are complete isometries (the latter weak$^*$-continuous) then $\phi\ten\psi:X_1\oten_{\mathcal{F}} Y_1\hookrightarrow X_2\oten_{\mathcal{F}} Y_2$ is a complete isometry. Furthermore, we have $$(X_1\oten_{\mathcal{F}}Y_1)\cap (X_2\oten_{\mathcal{F}}Y_2)=(X_1\cap X_2)\oten_{\mathcal{F}}(Y_1\cap Y_2).$$

When both $X\subseteq\BH$ and $Y\subseteq\BH$ are weak$^*$-closed and hence dual operator spaces with operator preduals $X_*$ and $Y_*$, there is a canonical completely isometric isomorphism 
$X\oten_{\mathcal{F}} Y\cong (X_*\pten Y_*)^*(\cong\CB(Y_*,X))$ \cite[Proposition 3.3]{Ru92}. In this case, $X\oten_{\mathcal{F}} Y$ is a weak$^*$-closed subspace of $B(\mathcal{H}\ten\mathcal{K})$ containing $X\odot Y$, so that $X\oten Y\subseteq X\oten_{\mathcal{F}} Y$. Equality holds when $X=M$ and $Y=N$ are von Neumann algebras \cite{Tom70}, but, more generally, we have:

\begin{Lem}\label{slicemapproperty}
    If $X,Y$ are dual operator spaces, then $X\bar\otimes Y=X\bar\otimes_{\F} Y$ provided:
    \begin{enumerate}[noitemsep]
        \item $Y$ is an injective von Neumann algebra, or
        \item $Y$ has $w^*$OAP.
    \end{enumerate}
    \end{Lem}
    
    \begin{proof}
        A von Neumann algebra is injective if and only if it has the $w^*$CPAP, hence $(1)\implies (2)$. If $(2)$ holds, then $X\ovot Y = X\ovot_\mcF Y$ by \cite{Kr91}*{Theorem 2.6}.
    \end{proof}

Let $X\subseteq\BH$ be an operator space and let $Y$ be a dual operator space with predual $Y_*$. Since
$$\BH\oten_{\mathcal{F}} Y\ni z\mapsto (\omega\mapsto (\id\ten\omega)z)\in \CB(Y_*,\BH),$$
is a completely isometric isomorphism,
its restriction to $X\oten_{\mathcal{F}} Y$ yields a canonical completely isometric isomorphism $X\oten_{\mathcal{F}} Y\cong\CB(Y_*,X)$ (by the definition of the Fubini tensor product).  Combined with the universal property of the projective tensor product $\pten$ \cite[(1.5.11)]{BLM04}, it follows that
$$\CB(Y_*\pten Z,X)\cong\CB(Z,\CB(Y_*,X))\cong\CB(Z,X\oten_{\mathcal{F}} Y)$$
completely isometrically for any operator space $Z$.

Frequently, we will use the fact that if $S,T$ are Banach spaces and $f: S^*\to T^*$ is a weak$^*$-continuous linear isometry, then $f(S^*)$ is weak$^*$-closed in $T^*$ \cite{Rud91}*{Theorem 4.14}. In particular, if $X,Y$ are dual operator spaces and $f: X\to Y$ is a weak$^*$-continuous (complete) isometry, then $f(X)$ is weak$^*$-closed in $Y$.

\subsection{Locally compact quantum groups}\label{prelim:locallycompactquantumgroups}
Our main objects of study are locally compact quantum groups, whose modern definition was proposed by Kustermans-Vaes \cites{KV00,KV03}. They are generalizations of locally compact (Hausdorff) groups, in the spirit of non-commutative topology and non-commutative measure theory. By definition, a \emph{locally compact quantum group} $\G$ is described by a von Neumann algebra $L^{\infty}(\G)$, a comultiplication $\Delta_{\G}: L^{\infty}(\G)\rightarrow L^{\infty}(\G)\bar\otimes L^{\infty}(\G)$, which is a  normal, unital $*$-homomorphism satisfying $(\Delta_{\G}\otimes\id)\Delta_{\G}=(\id\otimes\Delta_{\G})\Delta_{\G}$ and two normal semifinite faithful weights $\varphi,\psi$ (Haar integrals) which satisfy the condition of left (respectively, right) invariance. As is common in the field, we denote the GNS Hilbert space for $\varphi$ by $L^2(\G)$ and treat $L^{\infty}(\G)$ as acting on $L^2(\G)$. The GNS Hilbert space for $\psi$ is identified with $L^2(\G)$. An important result states that with any locally compact quantum group $\G$ we can associate its dual $\hat{\G}$, which is also a locally compact quantum group and $L^{\infty}(\hat{\G})\subseteq B(L^2(\G))$. The dual of $\hat{\G}$ is equal to $\G$ and in fact the operation $\G\mapsto \hat{\G}$ extends the Pontryagin duality of locally compact, abelian groups. The predual space of $L^{\infty}(\G)$ is denoted $L^1(\G)$; it carries a completely contractive Banach algebra structure with multiplication $L^1(\G)\hat{\otimes}L^1(\G)\ni \omega\otimes\nu\mapsto \omega\star\nu=(\omega\otimes\nu)\Delta_{\G}\in L^1(\G)$.

With the quantum group $\G$, we can associate its \emph{opposite} $\G^{\operatorname{op}}$ and \emph{commutant} $\G'$ \cite[Section 4]{KV03}. By definition, $L^{\infty}(\G^{\operatorname{op}})=L^{\infty}(\G)$ and $\Delta_{\G^{\operatorname{op}}}(x)= \Delta_{\G}(x)_{21}\,(x\in L^{\infty}(\G))$ using the leg numbering notation. Consider the canonical anti-isomorphism $j_{\G}: L^{\infty}(\G)\rightarrow L^{\infty}(\G)'$ given by $j_{\G}(x)=J_{\vp} x^* J_{\vp}$, where $J_{\vp}$ is the modular conjugation of the weight $\vp$. Then $\G'$ is defined by $L^{\infty}(\G')=L^{\infty}(\G)'$ and $\Delta_{\G'}(j_{\G}(x))=(j_{\G}\otimes j_{\G})\Delta_{\G}(x)$ for $x\in L^{\infty}(\G)$. We will frequently use the quantum group $\check{\G}=(\hat{\G})'$, which should be thought of as a right version of the dual of $\G$. Let us also introduce a canonical self-adjoint unitary $u_{\G}=\nu_{\G}^{i/8} J_{\vp}J_{\hat{\vp}}$, where $\nu_{\G}>0$ is the \emph{scaling constant}.

A fundamental construction gives us the left \emph{Kac-Takesaki operator} $\ww^{\G}\in L^{\infty}(\G)\bar\otimes L^{\infty}(\hat{\G})$, which is the unitary operator uniquely characterized by
\[
((\omega\otimes\id)\ww^{\G *}) \Lambda_{\vp}(x)=\Lambda_{\vp}((\omega\otimes\id)\Delta_{\G}(x)), \quad \omega \in L^1(\G), \quad x\in \mathfrak{N}_{\vp}=\{x\in L^{\infty}(\G) \;|\; \vp(x^*x)<+\infty\}.
\]
There is also a right Kac-Takesaki operator; the unitary $\vv^{\G}\in L^{\infty}(\check{\G})\bar\otimes L^{\infty}(\G)$ uniquely determined by
\[
((\id\otimes\omega)\vv^{\G})\Lambda_{\psi}(x)=
\Lambda_{\psi}((\id\otimes\omega)\Delta_{\G}(x)),\quad  \omega\in L^1(\G), \quad x\in \mathfrak{N}_{\psi}.
\]
The Kac-Takesaki operators implement the comultiplication via
\[
\Delta_{\G}(x)=\ww^{\G *}(\I\otimes x ) \ww^{\G} = \vv^{\G}(x\otimes \I) \vv^{\G *}, \quad x\in L^\infty(\G). \]
The slices of these unitary operators generate respective algebras, in the sense that
\[
L^{\infty}(\G)=[(\id\otimes\omega)\ww^{\G} \;|\;
\omega\in L^1(\hat{\G})]^{w*},\quad 
L^{\infty}(\hat\G)= [(\omega\otimes\id)\ww^{\G} \;|\;
\omega\in L^1(\G)]^{w*}
\]
and similarly for $\vv^{\G}$. Taking norm-closures instead, we obtain non-degenerate C$^*$-subalgebras
\[
C_0(\G)=[ (\id\otimes\omega)\ww^{\G} \;|\;
\omega\in L^1(\hat{\G})],\quad 
C_0(\hat\G)=[(\omega\otimes\id)\ww^{\G} \;|\;
\omega\in L^1(\G)].
\]
The comultiplication $\Delta_{\G}$ restricts to a non-degenerate $*$-homomorphism $C_0(\G)\rightarrow \M(C_0(\G)\otimes C_0(\G))$, denoted with the same symbol. By definition, one says that $\G$ is \emph{compact} if $1\in C_0(\G)$ and $\G$ is \emph{discrete} if $\hat{\G}$ is compact. The operators $\ww^{\G},\vv^{\G}$ satisfy a number of properties, let us mention
\[
\ww^{\hat{\G}}=\ww^{\G *}_{21},\quad
\ww^{\check{\G}}=\vv^{\G},\quad
(\Delta_{\G}\otimes\id)\ww^{\G}=
\ww^{\G}_{13}\ww^{\G}_{23},\quad 
(\id\otimes\Delta_{\G})\vv^{\G}=
\vv^{\G}_{12} \vv^{\G}_{13}
\]
and refer to the literature \cite{KV00, KV03, VD14} for more. The operators $\ww^{\G},\vv^{\G}$ give us two canonical ways of extending the comultiplication to $B(L^2(\G))$:
\[\begin{split}
\Delta_{\G,l}&: B(L^2(\G))\ni x\mapsto \ww^{\G *}(\I\otimes x)\ww^{\G} \in 
L^{\infty}(\G)\bar\otimes B(L^2(\G)),\\
\Delta_{\G,r}&: B(L^2(\G))\ni y\mapsto \vv^{\G }(y\otimes \I)\vv^{\G *} \in 
B(L^2(\G))\bar\otimes L^{\infty}(\G).
\end{split}\]
These maps can be seen as describing actions $\G\curvearrowright B(L^2(\G))$ and $B(L^2(\G))\curvearrowleft \G$ (see subsection \ref{actions}).

We will need the notion of a (right/left) unitary representation of $\G$. By definition, a right unitary representation of $\G$ on a Hilbert space $\mcH$ is a unitary operator $\uu\in B(\mcH)\bar\otimes L^{\infty}(\G)$ satisfying $(\id\otimes \Delta_{\G})\uu=\uu_{12} \uu_{13}$. Similarly, a unitary operator $\widetilde{\uu}\in L^{\infty}(\G)\bar\otimes B(\mcH)$ is a left unitary representation if $(\Delta_{\G}\otimes\id) \widetilde{\uu}=\widetilde{\uu}_{13} \widetilde{\uu}_{23}$. We will use only unitary representations, hence we will sometimes skip this adjective. Note that $\uu\in B(\mcH)\ovot L^\infty(\G)$ is a right unitary $\G$-representation on $\mcH$ if and only if $\tilde{\uu}= \uu_{21}$ is a left unitary $\G$-representation on $\mcH$. Operator $\vv^{\G}$ is a right $\G$-representation on $L^2(\G)$ (the right regular representation) and $\ww^{\G}$ is a left representation on $L^2(\G)$ (the left regular representation). 

The C$^*$-algebra $C_0(\G)$ comes also with its \emph{universal} variant $C_0^u(\G)$ \cite{K01}. It is a C$^*$-algebra equipped with a surjective $*$-homomorphism $\Lambda_{\G}: C_0^u(\G)\rightarrow C_0(\G)$ (\emph{reducing map}) and a comultiplication, i.e.~a co-associative, non-degenerate $*$-homomorphism $\Delta^u_{\G}: C_0^u(\G)\rightarrow \M(C_0^u(\G)\otimes C_0^u(\G))$. Also the left and right Kac-Takesaki operators $\ww^{\G},\vv^{\G}$ admit universal lifts, which come in several variants. There is a unique unitary bicharacter $\WW^{\G}\in \M(C_0^u(\G)\otimes C_0^u(\hat{\G}))$ satisfying $(\Lambda_{\G}\otimes \Lambda_{\hat{\G}})(\WW^{\G})=\ww^{\G}$. Applying the reducing map to one of the legs, we obtain $\wW^{\G}$ and $\Ww^{\G}$. Of particular importance is the operator $\wW^{\G}\in \M(C_0(\G)\otimes C_0^u(\hat{\G}))$, because of its fundamental property: for any Hilbert space $\mcH$ and unitary operator $\uu\in L^{\infty}(\G)\bar\otimes B(\mcH)$, $\uu$ is a left unitary representation if and only if there is a non-degenerate $*$-homomorphism $\pi: C_0^u(\hat{\G})\rightarrow B(\mcH)$ such that $\uu=(\id\otimes \pi)(\wW^{\G})$ (similar properties hold for $\vv^{\G}$). Most of the analytical structure admits a lift from the reduced to universal level, see \cite[sections 8, 9]{K01}. Using the half-lifted Kac-Takesaki operators, we can introduce the half-lifted comultiplications $\Delta_{\G}^{u,r}: C_0(\G)\rightarrow \M(C_0^u(\G)\otimes C_0(\G))$ and $\Delta_{\G}^{r,u}:C_0(\G)\rightarrow \M(C_0(\G)\otimes C_0^u(\G))$ via
\[
\Delta_{\G}^{u,r}(x)=\Ww^{\G *}(1\otimes x)\Ww^{\G},\quad 
\Delta_{\G}^{r,u}(x)=\vV^{\G }(x\otimes 1)\vV^{\G *},\quad x\in C_0(\G),
\]
see e.g.~\cite[Section 2]{DKV25}.

For any locally compact quantum group $\G$, we define $\lambda_{\G}: L^1(\G)\ni \omega\mapsto (\omega\otimes \id)\ww^{\G}\in  C_0(\hat{\G})$ and $\lambda_{\G}^u: C_0^u(\G)^*\ni \mu\mapsto (\mu\otimes\id)\Ww^{\G}\in M(C_0(\hat{\G}))$. With these (injective) maps, we define the \emph{Fourier algebra} as $\A(\G)=\lambda_{\hat{\G}}(L^1(\hat{\G}))$ and \emph{Fourier-Stieltjes algebra} $\B(\G)=\lambda^u_{\hat{\G}}(C_0^u(\hat{\G})^*)$. Together with the multiplication of $L^{\infty}(\G)$ and the operator space structure pulled from $L^1(\hat{\G})$ (resp.~$C_0^u(\hat{\G})^*$), both are completely contractive Banach algebras. Next, we say that $b\in L^{\infty}(\G)$ is a \emph{left completely bounded multiplier} of $\A(\G)$ if $ba\in \A(\G)$ for all $a\in \A(\G)$, and the associated map $\A(\G)\rightarrow \A(\G)$ is completely bounded. The set of such elements is denoted $\M^l_{\cb}(\A(\G))$ and we have $\A(\G)\subseteq \B(\G)\subseteq \M^l_{\cb}(\A(\G))\subseteq \M(C_0(\G))$, with $\A(\G)$ being an ideal in $B(\G)$. Under the (completely isometric) identification $\A(\G)\cong L^1(\hat{\G})$, any element $b\in \M^l_{\cb}(\A(\G))$ gives a completely bounded map $\Theta^l(b)_* \in \CB( L^1(\hat{\G}))$, we denote its dual by $\Theta^l(b)=(\Theta^l(b)_*)^*\in \CB^\sigma(L^{\infty}(\hat{\G}))$. Together with the product of $L^{\infty}(\G)$ and the norm $\|b\|_{\cb}=\|\Theta^l(b)\|_{\cb}$, the space $\M^l_{\cb}(\A(\G))$ becomes a Banach algebra.

The restriction of functionals gives an injective map $L^1(\G)\rightarrow \M^l_{\cb}(\A(\G))^*$ and we define $Q^l(\A(\G))\subseteq \M^l_{\cb}(\A(\G))^*$ as the norm-closure of its image. It turns out that $\M^l_{\cb}(\A(\G))\cong Q^l(\A(\G))^*$ via the canonical pairing \cite[Theorem 3.4]{HuNeufangRuan}. Whenever we speak of the weak$^*$-topology on $\M^l_{\cb}(\A(\G))$, it is this predual that we have in mind. For a more thorough discussion see \cite{JNR09, HuNeufangRuan, Cra19, DKV24,DKV25} and references therein.

Using the above objects, one can introduce several approximation properties (see \cite{BT03, Bra17, DKV24, Cra19} and references therein). One says that a locally compact group $\G$ is \emph{strongly amenable} (or that $\hat{\G}$ is \emph{coamenable}) if the Banach algebra $\A(\G)$ has a bounded approximate unit. A significant weakening of this property is \emph{the approximation property (AP)}: $\G$ has the AP if there is a net $(a_i)_{i\in I}$ in the Fourier algebra $\A(\G)$ which converges to $1$ in the weak$^*$-topology of $\M^l_{\cb}(\A(\G))$. By \cite[Theorem 4.4]{DKV24}, the AP admits an equivalent characterization: $\G$ has the AP if and only if there is a net $(a_i)_{i\in I}$ in $\A(\G)$ such that $(\Theta^l(a_i))_{i\in I}$ converges to the identity in the stable point-weak$^*$ topology of $\CB^\sigma(L^{\infty}(\hat{\G}))$. Finally, we say that $\G$ is \emph{amenable} if there is a state $m\in L^{\infty}(\G)^*$ which is left invariant, in the sense that
\[
m((\omega\otimes\id)\Delta_{\G}(x))=m(x)\omega(1), \quad x\in L^{\infty}(\G),\quad \omega\in L^1(\G).
\] It is known that strong amenability of $\G$ implies amenability and the AP of $\G$. In the case of classical groups strong amenability is equivalent to amenability, however whether this equivalence holds for all locally compact quantum groups is an important open problem (see also \cite{Cra19}*{Theorem 7.2}).

Whenever convenient and clear from the context, we will simplify the notation and write e.g.~$\Delta$ instead of $\Delta_{\G}$. 

\subsection{Actions and crossed products}\label{actions} Let $\G$ be a locally compact quantum group. A \emph{right $\G$-$W^*$-algebra} $(M, \alpha)$ consists of a von Neumann algebra $M$ together with a unital, normal, isometric $*$-homomorphism $\alpha: M \to M \ovot L^\infty(\G)$ satisfying the coaction property $(\id \otimes \Delta)\alpha= (\alpha \otimes \id)\alpha$. We will often employ the more intuitive notation $M\stackrel{\alpha}\curvearrowleft\G$ and refer to $\alpha$ as a \emph{$\G$-action}.
The \emph{Podleś density condition} 
\begin{equation}\label{Podles}
    [\alpha(M)(1\otimes L^\infty(\G))]^{\text{w}^*} = M \ovot L^\infty(\G)
\end{equation}
is automatically satisfied \cite{KaS15}*{Proposition 2.9}. Moreover, we have \cite{Vae01}*{Theorem 2.7}
\begin{equation}\label{fixed}
    \alpha(M)= \{z\in M \ovot L^\infty(\G)\mid (\alpha\otimes \id)(z) = (\id \otimes \Delta)(z)\}.
\end{equation}
We write $\Fix(M, \alpha):= \{m\in M\mid\alpha(m)= m\otimes 1\}$ for the space of \emph{fixed points} of the action. For example, $(L^\infty(\G), \Delta)$ and $(B(L^2(\G)), \Delta_r)$ are right $\G$-$W^*$-algebras and $\Fix(L^\infty(\G),\Delta) = \C1$ and $\Fix(B(L^2(\G)), \Delta_r)= L^\infty(\hat{\G})$. Given two right $\G$-$W^*$-algebras $(M, \alpha)$ and $(N, \beta)$, a completely bounded map (not necessarily normal) $\varphi: M \to N$ is called \emph{$\G$-equivariant} if $(\varphi\otimes \id)\circ \alpha= \beta \circ \varphi$. 

Of course, similar definitions and results hold for \emph{left} $\G$-$W^*$-algebras, and we shall need both left and right $\G$-actions in this paper. For simplicity, we keep working with right $\G$-$W^*$-algebras in this subsection.

For future use, we also record the following equivariant version of the Stinespring theorem (see \cite[Theorem 2.1]{Pau82} for the case of group $C^*$-dynamical systems). Its proof is inspired by the proof of \cite{DCDR24}*{Proposition 5.6}.

\begin{Theorem}[Equivariant Stinespring theorem]\label{equivariant Stinespring}  Let $\G$ be a locally compact quantum group, let $(M, \alpha)$ be a right $\G$-$W^*$-algebra, let $\uu \in B(\mcH)\ovot L^\infty(\G)$ be a unitary $\G$-representation and let $\varphi: M \to B(\mcH)$ be a normal $\G$-equivariant completely positive map, in the sense that
$$(\varphi\otimes \id)(\alpha(m)) = \uu(\varphi(m)\otimes 1)\uu^*, \quad m \in M.$$
There exists a Hilbert space $\mathcal{H}_\varphi$, a bounded operator $T_\varphi: \mcH\to \mcH_\varphi$, a normal unital $*$-representation $\pi_\varphi: M \to B(\mcH_\varphi)$ and a unitary $\G$-representation $\uu_\varphi\in B(\mcH_\varphi)\ovot L^\infty(\G)$ such that $(T_\varphi\otimes  1)\uu = \uu_\varphi(T_\varphi\otimes 1)$ and
$$\varphi(m)= T_\varphi^* \pi_\varphi(m) T_\varphi, \quad (\pi_\varphi\otimes \id)\alpha(m) = \uu_\varphi(\pi_\varphi(m)\otimes 1)\uu_\varphi^*, \quad m\in M.$$
\end{Theorem} 
\begin{proof} We adapt the usual proof of the Stinespring theorem to the equivariant setting.  Write $\mcH_\varphi:= M\otimes_\varphi \mcH$ for the separation-completion of the space $M \odot \mcH$ endowed with the semi-inner product uniquely determined by $\langle x \otimes \xi, y \otimes \eta\rangle_\varphi := \langle \xi, \varphi(x^*y)\eta\rangle$ for $x,y \in M$ and $\xi, \eta \in \mcH$. The image of $x\otimes \xi$ under the natural map $M \odot \mcH \to \mcH_\varphi$ will be denoted by $x\otimes_\varphi \xi$. Define $T_\varphi: \mcH\to \mcH_\varphi$ by
$T_\varphi(\xi)=1 \otimes_\varphi \xi$, so $T_\varphi$ is bounded with $\|T_\varphi\|^2 = \|\varphi(1)\|$. Define also a unital $*$-representation 
$$\pi_\varphi: M\to B(\mcH_\varphi), \quad \pi_\varphi(m)(x\otimes_\varphi \xi) = mx\otimes_\varphi\xi, \quad m,x\in M, \quad \xi \in \mcH.$$
Since
$\omega_{x\otimes_\varphi \xi, x'\otimes_\varphi \xi'}(\pi_\varphi(m))= \langle \xi, \varphi(x^*mx')\xi'\rangle$ for $m,x,x'\in M$ and $\xi, \xi'\in \mcH,$ we see that
the normality of $\varphi$ implies the normality of $\pi_\varphi$. It is clear that $\varphi(m)= T_\varphi^* \pi_\varphi(m)T_\varphi$ for $m\in M$.  

Note that we have a unitary
   $$(M\otimes_\varphi \mcH)\otimes L^2(\G)\cong (M \ovot L^\infty(\G))\otimes_{\varphi\otimes \id} (\mcH \otimes L^2(\G)): (m\otimes_\varphi \xi)\otimes \eta \mapsto (m\otimes 1)\otimes_{\varphi\otimes\id}(\xi \otimes \eta).$$
This, together with the $\G$-equivariance of $\varphi$, allows us to define the isometry
$$\uu_\varphi: (M\otimes_\varphi \mcH)\otimes L^2(\G)\to (M\ovot L^\infty(\G))\otimes_{\varphi\otimes  \id}(\mcH\otimes L^2(\G))\cong (M\otimes_\varphi \mcH)\otimes L^2(\G)$$
by the formula 
$$\uu_\varphi((m\otimes_\varphi \xi)\otimes \eta)= \alpha(m)\otimes_{\varphi\otimes \id} \uu(\xi \otimes \eta), \quad m\in M, \quad \xi \in \mcH, \quad \eta \in L^2(\G).$$
It follows from the Podleś condition \eqref{Podles}, the normality of $\varphi$, and the identity
$$\alpha(m)\otimes_{\varphi\otimes \id} (\xi \otimes g\eta)= \alpha(m)(1\otimes g) \otimes_{\varphi\otimes \id}(\xi \otimes \eta), \quad m\in M, \quad \xi \in \mcH, \quad \eta \in L^2(\G), \quad g \in L^\infty(\G),$$
that $\uu_\varphi \in B(\mcH_\varphi \otimes L^2(\G))$ is a unitary. Similarly as in the proof of \cite{DCDR24}*{Proposition 5.6}, one verifies next that $\uu_\varphi \in B(\mcH_\varphi)\ovot L^\infty(\G)$ is a unitary $\G$-representation such that
$$(\pi_\varphi\otimes \id)\alpha(m)= \uu_\varphi(\pi_\varphi(m)\otimes 1)\uu_\varphi^*, \quad m\in M.$$
If $\xi \in \mcH$ and $\eta \in L^2(\G)$, we find
\begin{align*}
    (T_\varphi\otimes 1)\uu(\xi \otimes \eta)&\cong (1\otimes 1)\otimes_{\varphi\otimes \id} \uu(\xi \otimes \eta) \cong \uu_\varphi((1 \otimes_\varphi \xi)\otimes \eta) = \uu_\varphi(T_\varphi\otimes 1)(\xi \otimes \eta),
\end{align*}
from which we conclude that $(T_\varphi\otimes 1)\uu = \uu_\varphi(T_\varphi\otimes 1)$.
\end{proof}

Given a right $\G$-$W^*$-algebra $(M, \alpha)$, we define the \emph{crossed product}
$M\bar{\rtimes}_\alpha \G := [\alpha(M)(1\otimes L^\infty(\check{\G}))]^{\text{w}^*}.$  It is a von Neumann subalgebra of $M\ovot B(L^2(\G))$ which is also described by
\begin{equation}\label{FubiniGWstar}
    M\bar{\rtimes}_\alpha \G = \{z\in M \ovot B(L^2(\G))\mid (\alpha \otimes \id)(z) = (\id \otimes \Delta_l)(z)\}.
\end{equation}
Define the normal isometric $*$-homomorphisms
\begin{align*}
    \alpha^{\rtimes}: M\bar{\rtimes}_\alpha \G \to (M\bar{\rtimes}_\alpha \G)\ovot L^\infty(\check{\G})&: z \mapsto (\id \otimes \check{\Delta}_r)(z)=\check{\vv}_{23}z_{12}\check{\vv}_{23}^*,\\
    \alpha_{\ad}: M\bar{\rtimes}_\alpha \G \to (M\bar{\rtimes}_\alpha \G)\ovot L^\infty(\G)&: z \mapsto (\id \otimes  \Delta_r)(z)= \vv_{23}z_{12}\vv_{23}^*.
\end{align*}
Then $(M\bar{\rtimes}_\alpha \G, \alpha^\rtimes)$ is a $\check{\G}$-$W^*$-algebra and  $(M\bar{\rtimes}_\alpha \G, \alpha_{\ad})$ is a $\G$-$W^*$-algebra. Moreover,
\begin{equation}
    \alpha(M)= \Fix(M\bar{\rtimes}_\alpha \G,\alpha^\rtimes).
\end{equation}
As an example, we note that the $*$-isomorphism
\begin{equation}\label{crossed}
    \Delta_l: B(L^2(\G))\to L^\infty(\G)\bar{\rtimes}_\Delta \G
\end{equation} is both $\G$ and $\check{\G}$-equivariant. We call $M\stackrel{\alpha}\curvearrowleft\G$ \emph{inner} if there exists a unitary $\uu \in M \ovot L^\infty(\G)$ such that $(\id \otimes \Delta)(\uu)= \uu_{12}\uu_{13}$ and $\alpha(m)= \uu(m\otimes 1)\uu^*$ for all $m\in M$. In that case,  the map
\begin{equation}\label{innercrossed}
    (M\bar{\rtimes}_\alpha \G, \alpha^\rtimes)\to (M\ovot L^\infty(\check{\G}), \id \otimes \check{\Delta}):  z \mapsto \uu^*z\uu
\end{equation}
is a $\check{\G}$-equivariant $*$-isomorphism \cite{DR26}*{Lemma 6.1}. As an example, note that $M\bar{\rtimes}_\alpha \G\stackrel{\alpha_{\ad}}\curvearrowleft \G$ is an inner action, so that the map
\begin{equation}\label{crossed2}
 ((M\bar{\rtimes}_\alpha \G)\bar{\rtimes}_{\alpha_{\ad}}\G, (\alpha_{\ad})^\rtimes) \to ((M\bar{\rtimes}_\alpha \G)\ovot L^\infty(\check{\G}), \id_{M\bar{\rtimes} \G}\otimes \check{\Delta}): z \mapsto \vv_{23}^* z \vv_{23}
\end{equation}
is a $\check{\G}$-equivariant $*$-isomorphism which maps $(\alpha\otimes \id)(z)=(\id \otimes \Delta_l)(z)\mapsto \alpha^\rtimes(z)$ for all $z\in M\bar{\rtimes}_\alpha \G$. On the other hand, the map
\begin{equation}\label{TTconcrete}
    (M\ovot B(L^2(\G)), \id \otimes \check{\Delta}_r)\to ((M\bar{\rtimes}_\alpha \G)\bar{\rtimes}_{\alpha^\rtimes}\check{\G}, (\alpha^{\rtimes})_{\ad}): z \mapsto \vv_{23}^*(\alpha \otimes \id)(z)\vv_{23}
\end{equation}
is a $\check{\G}$-equivariant $*$-isomorphism, which is known as the \emph{Takesaki-Takai duality} \cite{Vae01}*{Theorem 2.6}. This $*$-isomorphism  acts on generators by
$$\alpha(m)\mapsto \alpha(m)\otimes 1, \quad 1\otimes \check{x}\mapsto 1\otimes \check{\Delta}(\check{x}), \quad 1\otimes y'\mapsto 1 \otimes 1 \otimes y', \quad m \in M, \quad \check{x}\in L^\infty(\check{\G}), \quad y'\in L^\infty(\G)'.$$ In particular, $z\mapsto \alpha^\rtimes(z)$ for $z\in M\bar{\rtimes}_\alpha\G$.

\section{Equivariant operator spaces and crossed products}\label{SectionEquivariantOperatorSpaces}

Throughout this section, we fix a locally compact quantum group $\G$. Then $L^1(\G)$, together with the convolution product $\mu\star\nu:= (\mu\otimes\nu)\circ \Delta$, becomes a completely contractive Banach algebra. In this section, we make a careful study of some relevant concepts concerning its $L^1(\G)$-operator modules. More precisely, given such an $L^1(\G)$-module, we define and study the associated saturation space, Fubini crossed product and fixed point space. Restricting to the subcategory of dual $L^1(\G)$-operator-modules, we also study the notions of weak$^*$-crossed product and non-degeneracy. 

This section is strongly inspired by the articles \cite{An21, An23}, which consider only (completely isometric) coactions of classical locally compact groups and their duals on (dual) operator spaces. We develop a unified theory in the framework of locally compact quantum groups which also deals with the case where the coactions are not completely isometric.  A classical locally compact group $G$ is coamenable, and in particular $\hat{G}$ has the AP. This fact underlies many statements and proofs in \cite{An21, An23} and is the primary obstacle when extending results from the classical to the quantum setting. Nevertheless, almost all results of \cite{An21, An23} admit suitable analogues for locally compact quantum groups. Although we always indicate the corresponding earlier result, our proofs often differ substantially from those in \cite{An21, An23}. 

\subsection{Fubini crossed product, saturation and fixed points}\label{subsecFSF} Let $X$ be an operator space.
Using the canonical completely isometric identification
$\CB(L^1(\G)\hat{\otimes} X, X)\cong \CB(X, X\bar{\otimes}_\mcF L^\infty(\G))$,
a complete contraction $\rhd: L^1(\G)\hat{\otimes} X \to X$ corresponds to a complete contraction
$\alpha= \alpha_\rhd : X \to X \ovot_\mcF L^\infty(\G)$ via
$$(\id \otimes \omega)\alpha(x)= \omega\rhd x, \quad \omega\in L^1(\G), \quad x\in X.$$
The module property $\omega \rhd (\omega'\rhd x)= (\omega\star\omega')\rhd x$ is easily seen to be equivalent with the coaction identity $(\alpha\otimes \id)\alpha= (\id \otimes \Delta)\alpha.$ Thus, an $L^1(\G)$-operator module $(X, \rhd)$ is the same thing as an operator space $X$ together with a complete contraction $\alpha: X\to X\ovot_\mcF L^\infty(\G)$ satisfying the coaction identity.
We will often employ the more intuitive notation $X\stackrel{\alpha}\curvearrowleft \G$, and refer  to $\alpha$ as a \emph{$\G$-action} on $X$. We will frequently write $(X,\alpha)\in {}_{L^1(\G)}\Mod$ or simply $X\in {}_{L^1(\G)}\Mod$. If also $Y\in {}_{L^1(\G)}\Mod$ with associated $\G$-action $Y\stackrel{\beta}\curvearrowleft \G$, then a completely bounded map $\phi: X\to Y$ is an $L^1(\G)$-linear map if and only it is $\G$-equivariant in the sense that $(\phi\otimes \id)\alpha(x)= \beta(\phi(x))$ for all $x\in X$. Sometimes, we denote such a map by $\phi: (X, \alpha)\to (Y, \beta)$.

It is important to emphasize that the coaction $\alpha$ does not need to be injective, and non-trivial examples of such coactions appear naturally throughout the theory of operator modules:

\begin{Exa}\label{interestingex} Through the completely isometric isomorphism  $$\Phi: L^\infty(\G)\oplus_\infty\C\cong\CB(L^1(\G)_+, \C) =  L^1(\G)_+^*, \quad \Phi(x, \mu)(\omega + \lambda 1)= \omega(x) + \lambda \mu,$$
the von Neumann algebra $X:=L^\infty(\G)\oplus_{\infty} \C$ obtains a natural left $L^1(\G)$-operator module structure. It is explicitly given by
$$\omega\rhd (x, \lambda)= ((\id \otimes \omega)\Delta(x), \omega(x)), \quad \omega \in L^1(\G), \quad x\in L^\infty(\G), \quad \lambda\in \C.$$
The associated coaction $\alpha_{\rhd}: X \to X \ovot_\mcF L^\infty(\G)$ is not injective, since
$(\id\otimes\omega)\alpha_{\rhd}(0,1)=
\omega\rhd (0,1)= (0,0)$ for all $\omega\in L^1(\G)$, so that $\alpha_{\rhd}(0,1)=0$.
\end{Exa}

We will be interested in several subcategories of ${}_{L^1(\G)}\Mod$. Let us start by describing ${}_{L^1(\G)}\NMod$ in terms  of the coaction $\alpha$:

\begin{Lem}
    Let $(X, \rhd)\in {}_{L^1(\G)}\Mod$ with the associated coaction $\alpha: X\to X\ovot_\mcF L^\infty(\G)$. Assume moreover that $X$ is a dual operator space. Then $\alpha$ is weak$^*$-continuous if and only if $(X, \rhd)\in {}_{L^1(\G)}\NMod$. 
\end{Lem}
\begin{proof} If $\alpha$ is weak$^*$-continuous and $\omega\in L^1(\G)$, then $X\ni x\mapsto \omega\rhd x= (\id \otimes \omega)\alpha(x)\in X$ is trivially weak$^*$-continuous, so $(X, \rhd)\in {}_{L^1(\G)}\NMod$. Conversely, assume that $(X, \rhd)\in {}_{L^1(\G)}\NMod$. We need to show that $\alpha$ is weak$^*$-continuous. Recalling that $(X\ovot_\mcF L^\infty(\G))_*= X_*\hat{\otimes} L^1(\G)$, it is enough to show that $\Omega \circ\alpha$ is a weak$^*$-continuous functional on $X$ for every $\Omega\in X_*\hat{\otimes}L^1(\G)$. In fact, by \cite[Corollary 2.5.11]{Ped89} it is enough to show that $\Omega\circ \alpha$ is weak$^*$-continuous on the closed unit ball of $X$. So let $(x_i)_{i\in I}$ be a net which weak$^*$-converges to $x\in X$, where $\|x\|\le 1$ and $\|x_i\|\le 1$ for all $i\in I$. We will show that $\Omega(\alpha(x_i)) \to \Omega(\alpha(x))$. To this end, fix $\epsilon >0$. Choose $\tilde{\Omega}= \sum_{j=1}^n \mu_j\otimes \omega_j\in X_*\odot L^1(\G)$ such that $\|\Omega-\tilde{\Omega}\|\le \epsilon$. We then have
\begin{align*}
    |\tilde{\Omega}(\alpha(x_i))- \tilde{\Omega}(\alpha(x))| &\le \sum_{j=1}^n |\mu_j(\omega_j \rhd x_i- \omega_j \rhd x)| \xrightarrow[i\in I]{} 0,
\end{align*}
where the weak$^*$-continuity of the maps $X\ni x\mapsto \omega_j\rhd x\in X$ was used. Consequently, we may choose $i_0\in I$ such that $|\tilde{\Omega}(\alpha(x_i))-\tilde{\Omega}(\alpha(x))| \le \epsilon$ for $i\ge i_0$. For $i\ge i_0$, we then have $|\Omega\circ \alpha(x_i)-\Omega\circ \alpha(x)| \le 
3\epsilon$, finishing the proof.
\end{proof}

In the sequel, we will also consider the following subcategories of ${}_{L^1(\G)}\Mod$:
\begin{itemize}
    \item ${}_{L^1(\G)}\Mod^{\|\cdot\|}$ is the full subcategory of ${}_{L^1(\G)}\Mod$ with objects
$X \in {}_{L^1(\G)}\Mod$ for which the natural map
$$j_X: X \to \CB(L^1(\G), X), \quad j_X(x)(\omega)  =  \omega \rhd x, \quad x\in X, \quad  \omega\in L^1(\G)$$
is a complete isometry (equivalently, $\alpha_\rhd: X\to X\ovot_\mcF L^\infty(\G)$ is a complete isometry).
\item ${}_{L^1(\G)}\NMod^{\|\cdot\|} = {}_{L^1(\G)}\NMod\cap {}_{L^1(\G)}\Mod^{\|\cdot\|}$. 
\end{itemize}

Given $(X, \alpha)\in {}_{L^1(\G)}\Mod$, we associate several relevant spaces:
\begin{align*}
    \operatorname{Fix}(X, \alpha)&:=\{x\in X\mid \alpha(x)= x\otimes 1\}= \{x\in X\mid \forall \omega \in L^1(\G): \omega\rhd x = \omega(1)x\}, \\
    \Sat(X, \alpha) &:= \{z \in X \ovot_\mcF L^\infty(\G)\mid (\alpha\otimes \id)(z)= (\id \otimes  \Delta)(z)\},\\
    X\rtimes^\mcF_\alpha \G &:= \{z\in X \ovot_\mcF B(L^2(\G))\mid (\alpha\otimes \id)(z)=(\id \otimes \Delta_l)(z)\}.
\end{align*}

Suppose that $\phi: (X, \alpha)\to (Y, \beta)$ is a completely isometric $L^1(\G)$-module isomorphism and that $X\subseteq B(\mcH)$ and $Y\subseteq B(\mcK)$ are embeddings implementing $X\ovot_\mcF L^\infty(\G)\subseteq B(\mcH\otimes L^2(\G))$ and $Y\ovot_\mcF L^\infty(\G)\subseteq B(\mcK\otimes L^2(\G))$. Then $(\phi\otimes \id)\Sat(X, \alpha)= \Sat(Y, \beta)$, so $\Sat(X, \alpha)$ is unique up to completely isometric isomorphism. A similar statement is true for $X\rtimes_\alpha^\mcF \G$ (cf.\ \cite{An23}*{Proposition 3.8}). The operator space $\operatorname{Fix}(X, \alpha)$ is called the space of \emph{fixed points} of $X$. We  call $\Sat(X, \alpha)$  the \emph{saturation space} of $X$. It always holds that $\alpha(X)\subseteq \Sat(X, \alpha)$ and we call $(X,\alpha)$ \emph{saturated} if $\alpha(X)= \Sat(X, \alpha)$. It follows from \eqref{fixed} that right $\G$-$W^*$-algebras (when viewed as left $L^1(\G)$-modules) are always saturated. The operator space $X\rtimes_\alpha^\mcF \G$ is called the \emph{Fubini crossed product} of $X$ with respect to the $\G$-action $X\stackrel{\alpha}\curvearrowleft \G$. It is clear that $\Sat(X,\alpha)\subseteq X\rtimes_\alpha^\mcF \G$, and that $(1\otimes L^\infty(\check{\G}))(X\rtimes_\alpha^\mcF \G)(1\otimes L^\infty(\check{\G}))\subseteq X\rtimes_\alpha^\mcF \G$. This means that $X\rtimes_\alpha^\mcF \G$ is an (algebraic) $L^\infty(\check{\G})$-$L^\infty(\check{\G})$-bimodule in a natural way.

\begin{Rem}
To understand the claim $(1\otimes L^\infty(\check{\G}))(X\rtimes_\alpha^\mcF \G)(1\otimes L^\infty(\check{\G}))\subseteq X\rtimes_\alpha^\mcF \G$, one needs to consider an embedding $X\subseteq B(\mcH)$, so that $X\rtimes_\alpha^\mcF \G\subseteq B(\mcH)\ovot B(L^2(\G))$. The expression $(1\otimes L^\infty(\check{\G}))(X\rtimes_\alpha^\mcF \G)(1\otimes L^\infty(\check{\G}))$ can then be understood to be a multiplication in the algebra $B(\mcH)\ovot B(L^2(\G))$, where it makes sense. The end result lands in $X\rtimes_\alpha^\mcF \G$, and does not depend on the choice of embedding $X\subseteq B(\mcH)$. Similar comments concerning such multiplications also apply to other places in this paper (for example in Proposition \ref{relations}, Lemma \ref{lemma3} Proposition \ref{TT1} and many more), and will not be explicitly commented on anymore.
\end{Rem}
We can view $\Fix(-)$, $\Sat(-)$ and $-\rtimes^\mcF \G$ as functors ${}_{L^1(\G)}\Mod\to {}_{\C} \Mod$ in the obvious way. We will now prove that these functors are \emph{representable}. For this, we recall that if $(M, \alpha)$ is a left $\G$-$W^*$-algebra, then $M_*$ becomes a left $L^1(\G)$-operator module via
$$\omega \rhd \eta:= (\omega\otimes \eta)\circ  \alpha, \quad \omega\in L^1(\G), \quad \eta\in M_*.$$
In particular, the left $\G$-$W^*$-algebras $(\C, \tau)$, $(L^\infty(\G), \Delta)$ and $(B(L^2(\G)), \Delta_l)$ lead to left $L^1(\G)$-operator module structures on $\C, L^1(\G)$ and $B(L^2(\G))_*$. 

\begin{Prop}\label{p:natural} There are natural isomorphisms
    $\Fix(-)\cong {}_{L^1(\G)}\CB(\C, -)$, $\Sat(-)\cong {}_{L^1(\G)}\CB(L^1(\G), -)$  and $-\rtimes^\mcF \G \cong {}_{L^1(\G)}\CB(B(L^2(\G))_*, -)$ of covariant functors ${}_{L^1(\G)}\Mod\to {}_{\C}\Mod$.
\end{Prop}
\begin{proof} Given $(X, \alpha)\in {}_{L^1(\G)}\Mod$, consider the completely isometric isomorphism
$$X\ovot_\mcF B(L^2(\G))\to \CB(B(L^2(\G))_*,X):  z \mapsto \Phi_z: (B(L^2(\G))_*\ni \eta \mapsto (\id \otimes \eta)(z)\in X).$$

  If $z\in X \bar\otimes_\mcF B(L^2(\G))$, we have that 
    \begin{align*}
        &\Phi_z\in {}_{L^1(\G)} \CB(B(L^2(\G))_*, X)\\
        \iff&\forall \omega \in L^1(\G), \forall \eta\in B(L^2(\G))_*:  \Phi_z(\omega \rhd \eta)= \omega \rhd \Phi_z(\eta)\\
         \iff &\forall \omega \in L^1(\G), \forall \eta\in B(L^2(\G))_*: (\id \otimes \omega \otimes \eta)(\id \otimes \Delta_l)(z)= (\id \otimes \omega) \alpha((\id \otimes \eta)(z))\\
        \iff &\forall \omega \in L^1(\G), \forall \eta\in B(L^2(\G))_*: (\id \otimes \omega \otimes \eta)(\id \otimes \Delta_l)(z)= (\id \otimes \omega \otimes \eta)(\alpha\otimes \id)(z)\\
        \iff &z \in X\rtimes_\alpha^\mcF \G.
    \end{align*}
    Thus, by restriction, we obtain the completely isometric isomorphism
    $X\rtimes_\alpha^\mcF \G \cong {}_{L^1(\G)}\CB(B(L^2(\G))_*,X).$ These isomorphisms are clearly natural in $X$,  so we obtain the natural isomorphism ${}_{L^1(\G)}\CB(B(L^2(\G))_*,-)\cong -\rtimes^\mcF \G.$ The natural isomorphisms  $\Fix(-)\cong {}_{L^1(\G)}\CB(\C, -)$, $\Sat(-)\cong {}_{L^1(\G)}\CB(L^1(\G), -)$ are proven in a similar way.
\end{proof}

The following result relates the functors $\Fix(-), \Sat(-)$ and $-\rtimes^\mcF \G$.
\begin{Prop}[cf.\ \cite{An23}*{Proposition 3.2, Definition 3.3, Proposition 3.10}]\label{relations} 
    Fix $(X, \alpha)\in {}_{L^1(\G)}\Mod$.
    \begin{enumerate}[noitemsep]
        \item There is a natural action $X \ovot_\mcF B(L^2(\G))\stackrel{\tilde{\alpha}}\curvearrowleft\G$ given by
        \[
\tilde{\alpha}: X\bar\otimes_\mcF B(L^2(\G))
\rightarrow (X\bar\otimes_\mcF B(L^2(\G)))\bar\otimes_{\F} L^{\infty}(\G):
z\mapsto
\ww_{32} (\alpha\otimes\id)(z)_{132}\ww^*_{32}.
\]
Then $(X\ovot_\mcF B(L^2(\G)), \tilde{\alpha})\in {}_{L^1(\G)}\Mod$ and
$\Fix(X\ovot_\mcF B(L^2(\G)), \tilde{\alpha}) = X\rtimes_\alpha^\mcF \G.$
\item There is a natural action $X\rtimes_\alpha^\mcF \G\stackrel{\alpha^\rtimes}\curvearrowleft \check{\G}$ given by
$$\alpha^\rtimes: X\rtimes_\alpha^\F \G \to (X\rtimes_\alpha^\F \G) \ovot_\F L^\infty(\check{\G}): z \mapsto \check{\vv}_{23}z_{12}\check{\vv}_{23}^*.$$
Then $(X\rtimes_\alpha^\mcF \G, \alpha^\rtimes)\in {}_{L^1(\check{\G})}\Mod^{\|\cdot\|}$ and $\Fix(X\rtimes_\alpha^\mcF \G, \alpha^\rtimes) = \Sat(X, \alpha)$.
    \end{enumerate}
\end{Prop}
\begin{proof} (1) is proven using routine verifications.

(2) The statement that $\alpha^\rtimes$ determines a well-defined action $X\rtimes_\alpha^\mcF \G\curvearrowleft \check{\G}$ was observed in \cite{DR26}*{Proposition 4.2} (the proof does not use that $\alpha$ is a complete isometry).

By definition, $\Sat(X, \alpha)= (X\rtimes_\alpha^\mcF \G)\cap (X\ovot_\mcF L^\infty(\G))$. Given $z\in X\rtimes_\alpha^\mcF \G$, it follows from the fact that $\check{\vv}\in L^\infty(\G)'\ovot L^\infty(\check{\G})$ that we have $\alpha^\rtimes(z) = z\otimes 1$ if and only if $z\in X\ovot_\mcF L^\infty(\G)$. The equality $ \Sat(X, \alpha)=\Fix(X\rtimes_\alpha^ \mcF \G, \alpha^\rtimes)$ is then clear.
\end{proof}

\begin{Lem}[cf.\ \cite{An23}*{Lemma 4.11}]\label{saturationsaturated}
    If $(X, \alpha)\in {}_{L^1(\G)}\Mod$, then $(\Sat(X, \alpha), \id \otimes \Delta)\in {}_{L^1(\G)}\Mod^{\|\cdot\|}$ and $(\Sat(X, \alpha), \id \otimes \Delta)$ is saturated.
\end{Lem}
\begin{proof} If $z \in \Sat(X, \alpha)$ and $\omega \in L^1(\G)$, we have $(\id \otimes \id \otimes \omega)(\id \otimes \Delta)(z) = \alpha((\id \otimes \omega)(z))\in \alpha(X)\subseteq \Sat(X, \alpha)$, so $(\id \otimes \Delta)\Sat(X,\alpha)\subseteq \Sat(X, \alpha)\ovot_\mcF L^\infty(\G)$. It is thus clear that $(\Sat(X, \alpha), \id \otimes \Delta)\in {}_{L^1(\G)}\Mod^{\|\cdot\|}$.

Fix now an inclusion $X\subseteq B(\mcH)$. If $s\in \Sat(\Sat(X, \alpha), \id \otimes \Delta)\subseteq X\ovot_\mcF L^\infty(\G)\ovot_\mcF L^\infty(\G)$ and $\omega\in B(\mcH)_*$, we have
\begin{align*}
    (\id \otimes \Delta)((\omega\otimes \id\otimes \id)(s)) &= (\omega \otimes \id \otimes \id \otimes \id)(\id \otimes \id \otimes \Delta)(s)\\
    &= (\omega\otimes \id \otimes \id \otimes \id)(\id \otimes \Delta\otimes \id)(s)\\
    &= (\Delta \otimes \id)((\omega \otimes \id \otimes \id)(s)),
\end{align*}
so it follows from \eqref{fixed} that 
$s\in (B(\mcH)\ovot_\mcF \Delta(L^\infty(\G)))\cap (X\ovot_\mcF L^\infty(\G)\ovot_\mcF L^\infty(\G))= X \ovot_\mcF \Delta(L^\infty(\G)).$ Thus, there is $z\in X\ovot_\mcF L^\infty(\G)$  such that $s=(\id \otimes \Delta)(z)$. Using that $s\in \Sat(X, \alpha)\ovot_\mcF L^\infty(\G)$, we then find
\begin{align*}
    (\id \otimes \id \otimes \Delta)(\alpha\otimes \id)(z)&= (\alpha\otimes\id \otimes \id)(s) =  (\id \otimes \Delta\otimes \id)(s) = (\id \otimes \id \otimes \Delta)(\id \otimes \Delta)(z),
\end{align*}
from which we infer that $z\in \Sat(X, \alpha)$. Thus, $s\in (\id \otimes \Delta)(\Sat(X, \alpha))$.
\end{proof}

\begin{Prop}[cf.\ \cite{An23}*{Corollary 3.28}]\label{prop1'}
If $X\in {}_{L^1(\G)}\Mod$, then $(X\rtimes^{\F}_\alpha \G,\alpha^{\rtimes})$ is saturated.
\end{Prop}

\begin{proof} One can prove this by slightly adapting the proofs/statements of \cite{An23}*{Lemma 2.11, Corollary 3.28}. However, one can argue more directly that \eqref{fixed} (applied to the right $\check{\G}$-$W^*$-algebra $(B(L^2(\G)), \check{\Delta}_r)$) shows that $$\check{\Delta}_r(B(L^2(\G)))= \{z\in B(L^2(\G))\ovot L^\infty(\check{\G})\mid (\check{\Delta}_r\otimes \id)(z) = (\id \otimes \check{\Delta})(z)\}.$$
The proof then proceeds in exactly the same way as the proof of Lemma \ref{saturationsaturated}, by making use of the relation $(\Delta_l\otimes \id)\check{\Delta}_r=(\id \otimes \check{\Delta}_r)\Delta_l$, which is easy to check on $L^{\infty}(\G)$ and $L^{\infty}(\check{\G})$.
\end{proof}

We now discuss that when restricting to the subcategory ${}_{L^1(\G)}\Mod^{\|\cdot\|}$, the saturation space $\Sat(X, \alpha)$ can be seen as a sort of completion of $(X, \alpha)$.
To this end, recall from \cite{DR26}*{Definition 4.3} the following definition:

\begin{Def}\label{def1}
  We call $(X,\alpha)\in {}_{L^1(\G)}\Mod^{\|\cdot\|}$ $\G$-complete if for every $(Y, \beta)\in {}_{L^1(\G)}\Mod^{\|\cdot\|}$ for which $X\subseteq Y$ and $\beta\vert_X=\alpha$ holds and for every $y\in Y$, the condition $L^1(\G)\rhd y \subseteq X$ implies that $y\in X$.
\end{Def} 

Combining \cite{DR26}*{Proposition 4.4 (2)} and Proposition \ref{relations} (2), it is immediately clear that $(X, \alpha)$ is $\G$-complete if and only if $(X, \alpha)$ is saturated. 

In \cite{Ham11}*{Definition 5.5}, Hamana considers\footnote{Hamana works in the context of classical locally compact groups, but the definition readily extends to the quantum setting.} the notion of the $\G$-completion of $(X,\alpha)\in {}_{L^1(\G)}\Mod^{\|\cdot\|}$ : it consists of a $\G$-complete $(Y, \beta)\in {}_{L^1(\G)}\Mod^{\|\cdot\|}$ together with a $\G$-equivariant complete isometry $\iota: (X, \alpha)\to (Y, \beta)$ which satisfies the following universal property: if $\iota': (X, \alpha)\to (Y', \beta')$ is a $\G$-equivariant complete isometry where $(Y', \beta')\in {}_{L^1(\G)}\Mod^{\|\cdot\|}$ is $\G$-complete, then there exists a unique $\G$-equivariant complete isometry $j: Y\to Y'$ such that $j\circ \iota = \iota'$. This universal property determines $(Y, \beta)$ up to $\G$-equivariant completely isometric isomorphism. Hamana proceeds to show that the $\G$-completion of $(X, \alpha)$ always exists by realizing it explicitly as an $L^1(\G)$-submodule of the $\G$-injective envelope of $X$ \cite{Ham11}*{Proposition 5.6} and his proof works also (upon making minor modifications) in the quantum setting. Using the notion of saturation, we can now give a much more efficient and elementary proof.

\begin{Prop}\label{Gcompletion} Given $(X, \alpha)\in {}_{L^1(\G)}\Mod^{\|\cdot\|}$,
    $(\Sat(X, \alpha), \id \otimes \Delta)$ together with the natural $\G$-equivariant complete isometry $\alpha: X\to \Sat(X, \alpha)$ is the $\G$-completion of $(X, \alpha)$.
\end{Prop}
\begin{proof} We verify explicitly the universal property. By Lemma \ref{saturationsaturated}, $(\Sat(X, \alpha), \id \otimes \Delta)$ is $\G$-complete. Fix a $\G$-equivariant complete isometry $\iota': (X, \alpha)\to (Y', \beta')$, where $(Y', \beta') \in {}_{L^1(\G)}\Mod^{\|\cdot\|}$ is $\G$-complete, and recall that $\iota'\otimes\id$ is completely isometric. Then the composition
$$\Sat(X,\alpha)\cong (\iota'\otimes \id)(\Sat(X,\alpha)) = \Sat(\iota'(X), \beta')\subseteq \Sat(Y', \beta')=\beta'(Y')\cong Y'$$
is a required $\G$-equivariant complete isometry $j: \Sat(X, \alpha)\to Y'$. Its uniqueness follows from the fact that if $\tilde{j}: \Sat(X, \alpha)\to Y'$ is another $\G$-equivariant complete isometry satisfying $\tilde{j}\circ \alpha = \iota'$, then for $\omega \in L^ 1(\G)$ and $z\in \Sat(X, \alpha)$,
$$\omega\rhd j(z) = j(\omega\rhd z) = j(\alpha((\id \otimes \omega)(z)))= \tilde{j}(\alpha((\id \otimes \omega)(z))) =\tilde{j}(\omega\rhd z) = \omega\rhd \tilde{j}(z),$$
so that $j = \tilde{j}$ (since $\beta'$ is completely isometric).
\end{proof}

\subsection{Non-degeneracy and weak*-crossed product} 

We start with the following lemma, whose proof is quite standard.

\begin{Lem}
    Let $(X, \alpha)\in {}_{L^1(\G)}\NMod$. Then the following equalities are true:
    \begin{align}
        \label{j}&[\alpha(X)(1\otimes L^\infty(\check{\G}))]^{w*} = [(1\otimes L^\infty(\check{\G}))\alpha(X)]^{w*}= [(1\otimes L^\infty(\check{\G}))\alpha(X)(1\otimes L^\infty(\check{\G}))]^{w*},\\
        \label{k}&[\alpha(X)(1\otimes B(L^2(\G)))]^{w*} = [(1\otimes B(L^2(\G)))\alpha(X)]^{w*}= [(1\otimes B(L^2(\G)))\alpha(X)(1\otimes B(L^2(\G)))]^{w*}.
    \end{align}
\end{Lem}
\begin{proof} If $\xi,\eta\in \LL^2(\G)$ and  $x\in X$, we compute
\[\begin{split}
\quad\;
(1\otimes (\id\otimes \omega_{\xi,\eta})(\vv))\alpha(x)&=
(\id\otimes\id\otimes\omega_{\xi,\eta})(\vv_{23} \alpha(x)_{12} )\\
&=
(\id\otimes\id\otimes\omega_{\xi,\eta})((\id\otimes\Delta)(\alpha(x))\vv_{23}) \\
&=
(\id\otimes\id\otimes\omega_{\xi,\eta})((\alpha\otimes \id)(\alpha(x))\vv_{23}) =
\sum_{i\in I}\alpha( \omega_{\xi,e_i}\rhd x)
(1\otimes (\id\otimes \omega_{e_i,\eta})(\vv)),
\end{split}\]
where $\{e_i\}_{i\in I}$ is an orthonormal basis for $\LL^2(\G)$ and where the sum converges in the weak$^*$-topology. This shows the inclusion 
$[(1\otimes L^\infty(\check{\G}))\alpha(X)]^{w*}\subseteq [\alpha(X)(1\otimes L^\infty(\check{\G}))]^{w*}$. The converse inclusion follows similarly, by using that
$$\alpha(x)(1\otimes (\id \otimes \omega_{\xi, \eta})(\vv^*))= \sum_{i\in I} (1\otimes (\id \otimes \omega_{\xi, e_i})(\vv^*))\alpha(\omega_{e_i, \eta}\rhd x), \quad x\in X, \quad \xi, \eta \in L^2(\G).$$
The equalities \eqref{j} are then clear. On the other hand, using that $[L^\infty(\check{\G})L^\infty(\G)']^{w*}= B(L^2(\G))$ and using the equalities \eqref{j}, we find
\begin{align*}
    [\alpha(X)(1\otimes B(L^2(\G)))]^{w*} &= [\alpha(X)(1\otimes L^\infty(\check{\G}))(1\otimes L^\infty(\G)')]^{w*}\\
    &= [(1\otimes L^\infty(\check{\G}))\alpha(X)(1\otimes L^\infty(\G)')]^{w*}\\
    &= [(1\otimes L^\infty(\check{\G}))(1\otimes L^\infty(\G)')\alpha(X)]^{w*}= [(1\otimes B(L^2(\G)))\alpha(X)]^{w*}
\end{align*}
and the equalities \eqref{k} are also clear.
\end{proof}

\begin{Def}
    Given $(X, \alpha)\in {}_{L^1(\G)}\NMod$, the weak$^*$-crossed product $X\bar{\rtimes}_\alpha \G$ is the dual operator space given by \eqref{j}.
\end{Def}

It is evident that 
$X\bar{\rtimes}_\alpha \G \subseteq X\rtimes_\alpha^\F \G$. This becomes an equality if $(X,\alpha)$ is a right $\G$-$W^*$-algebra (cf.\ \eqref{FubiniGWstar}), but we will see that the inclusion can be strict in the more general case of dual operator spaces. It is  clear that $\alpha^\rtimes(X\bar{\rtimes}_\alpha \G)\subseteq (X\bar{\rtimes}_\alpha \G)\ovot L^\infty(\check{\G})$, so that $(X\bar{\rtimes}_\alpha \G, \alpha^\rtimes)\in {}_{L^1(\check{\G})}\NMod^{\|\cdot\|}$. Note that, by a slight abuse of notation, we write $\alpha^{\rtimes}$ for the dual action of $\check{\G}$ on both $X\rtimes^{\mcF}_\alpha \G$ and $X \bar{\rtimes}_\alpha \G$.

\begin{Prop}[cf.\ \cite{An23}*{Proposition 2.3, Lemma 3.9, Corollary 3.24}]\label{lemma3}
Let $(X, \alpha) \in {}_{L^1(\G)}\NMod$. The following properties are equivalent:
\begin{enumerate}[noitemsep]
     \item $X\bar\otimes \LL^\infty(\G)=[(1\otimes y)\alpha(x)\mid y\in \LL^\infty(\G),x\in X]^{w*},$
    \item $X\bar\otimes \LL^\infty(\G)=[\alpha(x)(1\otimes y)\mid y\in \LL^\infty(\G),x\in X]^{w*}$,
    \item $X\bar\otimes \LL^\infty(\G)=[(1\otimes y)\alpha(x)(1\otimes z)\mid y,z\in \LL^\infty(\G),x\in X]^{w*}$.
    \item $X\bar\otimes B(\LL^2(\G))=[(1\otimes b)\alpha(x)\mid b\in B(\LL^2(\G)),x\in X]^{w *},$
    \item $X\bar\otimes B(\LL^2(\G))=[\alpha(x)(1\otimes b)\mid b\in B(\LL^2(\G)),x\in X]^{w*}$,
    \item $X\bar\otimes B(\LL^2(\G))=[(1\otimes b)\alpha(x)(1\otimes c)\mid b,c\in B(\LL^2(\G)),x\in X]^{w *}$,
    \item $X=[\omega\rhd x \mid \omega\in \LL^1(\G),x\in X]^{w*} $,
\end{enumerate}
\end{Prop}
Note that the conditions $(1),(2),(3)$ can be seen as versions of the Podleś condition. Note also that here we use the weak$^*$-tensor product and not the Fubini tensor product.
\begin{proof} The equivalences $(4)\iff (5)\iff (6)$ follow from \eqref{k}. The implication $(x)\implies (7)$ is trivial for $x\in \{1, 2, 3,4,5,6\}$ and the implication $(y)\implies (3)$ is trivial for $y\in \{1,2\}$.

$(7)\implies (6)$ Assume $(7)$ holds. Then
\[
\begin{split}
&\quad\ [(1\otimes b)\alpha(x)(1\otimes c)\mid b,c\in B(L^2(\G)), x\in X]^{w*}\\
&\supseteq
[(\I\otimes b)\alpha(\omega\rhd x)(\I\otimes c)\mid b,c\in B(\LL^2(\G)),\omega\in \LL^1(\G), x\in X]^{w*}\\
&=
[(\id\otimes \id\otimes\omega)\bigl(
(\I\otimes b\otimes \I)
(\alpha\otimes\id)(\alpha(x))
(\I\otimes c\otimes\I)\bigr)\mid b,c\in B(\LL^2(\G)),\omega\in \LL^1(\G),x\in X]^{w*}\\
&=
[(\id\otimes \id\otimes\omega)\bigl(
(\I\otimes b\otimes \I)
(\id\otimes\Delta)(\alpha(x))
(\I\otimes c\otimes\I)\bigr)\mid b,c\in B(\LL^2(\G)),\omega\in \LL^1(\G), x\in X]^{w*}\\
&=
[(\id\otimes \id\otimes\omega)\bigl(
(\I\otimes b\otimes \I)
\vv_{23}
\alpha(x)_{12}
\vv_{23}^{ *}
(\I\otimes c\otimes\I)\bigr)\mid b,c\in B(\LL^2(\G)),\omega\in \LL^1(\G), x\in X]^{w*}\\
&=
[(\id\otimes \id\otimes\omega)\bigl(
(\I\otimes b\otimes \I)
\alpha(x)_{12}
(\I\otimes c\otimes\I)\bigr)\mid b,c\in B(\LL^2(\G)),\omega\in \LL^1(\G), x\in X]^{w*}\\
&=
[(\I\otimes b)
\alpha(x)
(\I\otimes c)\mid b,c\in B(\LL^2(\G)), x\in X]^{w*}\\
&=
[(\omega \rhd x)\otimes d\mid d\in B(\LL^2(\G)),\omega\in\LL^1(\G), x\in X]^{w*}= X\ovot B(L^2(\G)),
\end{split}
\]
where the assumption was used in the last equality.

$(4)\implies (1)$ By assumption, we have $X\ovot B(L^2(\G))=
[(\I\otimes b)\alpha(x)\mid x\in X,b\in B(\LL^2(\G))]^{w*},$
hence after applying $\id\otimes \Delta_l$
\[\begin{split}
[x\otimes \Delta_l(b)\mid x\in X,b\in B(\LL^2(\G))]^{w*}=
[(\I\otimes \Delta_l(b))(\id\otimes \Delta)(\alpha(x))\mid x\in X,b\in B(\LL^2(\G))]^{w*}.
\end{split}\]
It follows that
\[\begin{split}
&\quad\;
[x\otimes (y\otimes \I)\Delta_l(b)\mid x\in X,b\in B(\LL^2(\G)),y\in \LL^{\infty}(\G)]^{w*}\\
&=
[(\I\otimes (y\otimes \I)\Delta_l(b))(\alpha\otimes \id)(\alpha(x))\mid x\in X,b\in B(\LL^2(\G)),y\in \LL^{\infty}(\G)]^{w*}.
\end{split}\]
The Podleś condition for the left $\G$-$W^*$-algebra $(B(L^2(\G)), \Delta_l)$ (cf.\ \eqref{Podles}) then gives
$$X\ovot L^\infty(\G)\ovot B(L^2(\G))=[(\I\otimes y\otimes b)(\alpha\otimes \id)(\alpha(x))\mid x\in X,b\in B(\LL^2(\G)),y\in \LL^{\infty}(\G)]^{w*}.$$
Slicing off the third leg gives
$X\bar\otimes\LL^{\infty}(\G)=
[(\I\otimes y)
\alpha(\omega\rhd x)\mid x\in X,\omega \in \LL^1(\G),y\in \LL^{\infty}(\G)]^{w*}.$
Since the implication $(4)\implies (7)$ holds, we conclude that 
$X\ovot L^\infty(\G)= [(1\otimes y)\alpha(x)\mid x\in X, y\in L^\infty(\G)]^{w*},$
where the weak$^*$-continuity of $\alpha$ was used for the first time in the proof. Thus, the implication $(4)\implies (1)$ holds and the implication $(5)\implies (2)$ is proven similarly.
\end{proof}

We will use the following natural notion of non-degenerate part:

\begin{Def}
    If $X\in {}_{L^1(\G)}\NMod$, we write $X_{\operatorname{nd}}:= [\omega\rhd x\mid \omega\in L^1(\G), x\in X]^{w*}$ and we call it the \emph{non-degenerate part of $X$}. We call $X$ \emph{non-degenerate} if $X= X_{\operatorname{nd}}$ (equivalently, one of the conditions in Proposition \ref{lemma3} is satisfied).
\end{Def}

\begin{Lem}\label{ndlemma} Given $(X, \alpha)\in {}_{L^1(\G)}\NMod$, the non-degenerate part $X_{\operatorname{nd}}$ is the largest weak$^*$-closed $L^1(\G)$-submodule of $X$ that is non-degenerate. The space \eqref{k} is equal to $X_{\operatorname{nd}}\ovot B(L^2(\G))$.
\end{Lem}
\begin{proof}
    It follows trivially from the module property that $X_{\operatorname{nd}}$ is an $L^1(\G)$-submodule of $X$. It is also immediately clear that it contains every weak$^*$-closed non-degenerate $L^1(\G)$-submodule of $X$. The non-degeneracy of $(X_{\operatorname{nd}}, \alpha)$ follows from $L^1(\G)=[L^1(\G)\star L^1(\G)]$.

    Since $(X_{\operatorname{nd}}, \alpha)$ is non-degenerate, it follows from Proposition \ref{lemma3} that
    \begin{align*}
        X_{\operatorname{nd}}\ovot B(L^2(\G))&= [(1\otimes B(L^2(\G)))\alpha(X_{\operatorname{nd}})(1\otimes B(L^2(\G)))]^{w*}\\
        &\subseteq [(1\otimes B(L^2(\G)))\alpha(X)(1\otimes B(L^2(\G)))]^{w*} = X_{\operatorname{nd}}\ovot B(L^2(\G)),
    \end{align*}
    so $X_{\operatorname{nd}}\ovot B(L^2(\G))$ coincides with the space \eqref{k}.
\end{proof}

\begin{Cor}\label{cor1}
If $X\in {}_{L^1(\G)}\NMod$ is non-degenerate, then $\alpha(X)\subseteq X\ovot L^\infty(\G)$.
\end{Cor}
\begin{proof}
    This follows from Proposition \ref{lemma3} (1).
\end{proof}

\begin{Rem}
The converse of Corollary \ref{cor1} is not true. Indeed, consider the quantum group $\G$ with $\check{\G}=\oon{SL}(3,\R)$. Since $\oon{SL}(3,\R)$ does not have the AP \cite{LDLS11,HDL13}, the forthcoming Theorem \ref{main'} implies that there exists $X\in {}_{L^1(\G)}\NMod^{\|\cdot\|}$ which is degenerate. On the other hand, the group von Neumann algebra $L^{\infty}(\G)\cong \mathscr{L}(\oon{SL}(3,\R))$ is injective \cite{C76}, thus $X\bar\otimes_{\mathcal{F}} L^{\infty}(\G)=X\bar\otimes L^{\infty}(\G)$ by Lemma \ref{slicemapproperty}.
\end{Rem}

\begin{Rem}
     We do not know an example of $(X, \alpha)\in {}_{L^1(\G)}\NMod$ where the inclusion $\alpha(X)\subseteq X\bar\otimes L^{\infty}(\G)$ does not hold. If such an example exists, it will not be easy to construct: on the one hand, Corollary \ref{cor1} implies that $(X, \alpha)$ must necessarily be degenerate; on the other hand, the von Neumann algebra $L^\infty(\G)$ is not allowed to possess the slice map property (which is e.g.\ automatic if $\check{\G}$ has the AP, see Corollary \ref{smp}). 
\end{Rem}

\begin{Lem}[cf.\ \cite{An23}*{Proposition 2.5}]\label{satnd} If $(X, \alpha)\in {}_{L^1(\G)}\NMod$, then $(\Sat(X, \alpha), \id\otimes \Delta)\in {}_{L^1(\G)}\NMod^{\|\cdot\|}$. Setting $Y:= \operatorname{span}\{(\id \otimes \omega)(z)\mid \omega\in L^1(\G), z\in \Sat(X, \alpha)\}\subseteq X,$ we have $\Sat(X, \alpha)_{\operatorname{nd}}= \overline{\alpha(Y)}^{w*}$. In particular, if $(X, \alpha)\in {}_{L^1(\G)}\NMod^{\|\cdot\|}$, then $\Sat(X, \alpha)$ is non-degenerate if and only if $(X,\alpha)$ is both saturated and non-degenerate.
\end{Lem}
\begin{proof} If $z\in \Sat(X, \alpha)$ and $\omega \in L^1(\G)$, we have
$(\id \otimes \id \otimes \omega)(\id \otimes \Delta)(z) = \alpha((\id \otimes \omega)(z)),$ from which we immediately conclude that $\Sat(X, \alpha)_{\operatorname{nd}}= \overline{\alpha(Y)}^{w*}$. If $(X, \alpha)\in {}_{L^1(\G)}\NMod^{\|\cdot\|}$, then $\overline{\alpha(Y)}^{w*}= \alpha(\overline{Y}^{w*})$ and the last claim immediately follows as well.
\end{proof}

\begin{Prop}[cf.\ \cite{An23}*{Corollary 3.28}]\label{prop1}
    If $X\in {}_{L^1(\G)}\NMod$, then $(X\bar{\rtimes}_\alpha \G,\alpha^{\rtimes})\in {}_{L^1(\check{\G})}\NMod^{\|\cdot\|}$ is non-degenerate.
\end{Prop}

\begin{proof}
    Using the Podleś condition for $(L^\infty(\check{\G}),\check{\Delta})$ (cf.\ \eqref{Podles}), we have   
\[\begin{split}
\quad\;
[ \check{\omega}\rhd z\mid \check{\omega}\in \LL^1(\check{\G}),z\in X\bar{\rtimes}_\alpha \G]^{w *}&=
[(\id\otimes \id\otimes \check{\omega})
\alpha^{\rtimes}
( (1\otimes \check{y})
\alpha(x)
)
\mid \check{\omega}\in \LL^1(\check{\G}),x\in X, \check{y}\in L^\infty(\check{\G})]^{w *}\\
&=
[(\id\otimes \id\otimes \check{\omega})
\bigl(
 (1\otimes
\check{\Delta}(\check{y}))
\alpha(x)_{12}
\bigr)
\mid \check{\omega}\in \LL^1(\check{\G}),x\in X, \check{y}\in L^\infty(\check{\G})]^{w *}\\
&=
[
(1\otimes (\id\otimes \check{\omega})
\check{\Delta}(\check{y}))
\,\alpha(x)
\mid \check{\omega}\in \LL^1(\check{\G}),x\in X, \check{y}\in L^\infty(\check{\G})]^{w *}\\
&=
[
(1\otimes \check{y})
\alpha(x)
\mid x\in X, \check{y}\in \LL^{\infty}(\check{\G})]^{w *}=
X\bar{\rtimes}_\alpha \G,
\end{split}\]
hence $(X\bar{\rtimes}_\alpha \G,\alpha^{\rtimes})$ is non-degenerate.
\end{proof}

\begin{Prop}\label{saturationnd} Let $(X, \alpha)\in {}_{L^1(\G)}\NMod$.
\begin{enumerate}[noitemsep]
    \item  $\Sat(X, \alpha)= \Sat(X_{\operatorname{nd}}, \alpha)$.
    \item If $(X_{\operatorname{nd}}, \alpha)$ is saturated, then so is $(X, \alpha)$.
    \item   If $X\in {}_{L^1(\G)}\NMod^{\|\cdot\|}$ and $(X_{\operatorname{nd}}, \alpha)$ is saturated, then $(X, \alpha)$ is non-degenerate.
\end{enumerate}  
\end{Prop}
\begin{proof} $(1)$ Given $z\in \Sat(X,\alpha)$ and $\mu, \nu \in L^1(\G)$, we have the identity
    $$(\id \otimes (\mu\star \nu))(z) = (\id \otimes \mu)\alpha((\id \otimes \nu)(z))\in X_{\operatorname{nd}}.$$
    Since $L^1(\G)= [L^1(\G)\star L^1(\G)]$, it follows that $z\in X_{\operatorname{nd}}\ovot_\mcF L^\infty(\G)$. The inclusion $\Sat(X, \alpha)\subseteq \Sat(X_{\operatorname{nd}}, \alpha)$ then follows, while the other inclusion is trivial. 

$(2)$ follows immediately from $(1)$.

$(3)$ If $X\in {}_{L^1(\G)}\NMod^{\|\cdot\|}$ and $(X_{\operatorname{nd}}, \alpha)$ is saturated, we get
    $$\alpha(X)\subseteq  \Sat(X, \alpha)= \Sat(X_{\operatorname{nd}}, \alpha)=\alpha(X_{\operatorname{nd}})\subseteq \alpha(X),$$
    whence $X= X_{\operatorname{nd}}$ by injectivity of $\alpha$.
\end{proof}

\subsection{Takesaki-Takai duality} The two different notions of crossed products introduced in the previous subsections lead to two versions of the Takesaki-Takai duality. They are crucial tools in the theory.

\begin{Rem}\label{doublecrossed} Recall that $\check{\check{\G}}\cong \G$ via the isomorphism $x\mapsto u_\G x u_\G$.
    Given $(X, \check{\check{\alpha}})\in {}_{L^1(\check{\check{\G}})}\Mod$, we then obtain $(X, \alpha)\in {}_{L^1(\G)}\Mod$ via 
    $$\alpha(x):= (1\otimes u_\G)\check{\check{\alpha}}(x)(1\otimes u_\G) \in X\ovot_\mcF L^\infty(\G), \quad x\in X.$$
    Note also that
    $\Fix(X, \check{\check{\alpha}})= \Fix(X, \alpha).$
    
    In particular, if $X\in {}_{L^1(\G)}\Mod$, the above procedure allows us to turn $(X\rtimes_\alpha^\mcF \G\rtimes_{\alpha^\rtimes}^\mcF\check{\G}, \alpha^{\rtimes\rtimes})\in {}_{L^1(\check{\check{\G}})}\Mod^{\|\cdot\|}$ into an $L^1(\G)$-module via the coaction
$$\kappa: X\rtimes_\alpha^\mcF \G \rtimes_{\alpha^\rtimes}^\mcF\check{\G}\to (X\rtimes_\alpha^\mcF \G \rtimes_{\alpha^\rtimes}^\mcF\check{\G})\ovot_\mcF L^\infty(\G): z \mapsto u_{\G,4}\check{\check{\vv}}_{34}z_{123}\check{\check{\vv}}_{34}^* u_{\G,4}.$$
Since $\check{\check{\vv}}= (u_\G \otimes u_\G)\vv(u_\G \otimes u_\G)$ and $(u_\G \otimes 1)\vv(u_\G \otimes 1)= \ww_{21}$, we find 
$$\kappa(z) = \ww_{43}z_{123}\ww_{43}^*, \quad z \in X\rtimes_\alpha^\mcF \G \rtimes_{\alpha^\rtimes}^\mcF\check{\G}.$$
Similarly, if $X\in {}_{L^1(\G)}\NMod$, we have $((X\bar{\rtimes}_\alpha \G)\rtimes_{\alpha^\rtimes}^\mcF \check{\G}, \kappa)\in {}_{L^1(\G)}\NMod^{\|\cdot\|}.$
\end{Rem}

\begin{Prop}[cf.\ \cite{An23}*{Proposition 4.4, Proposition 4.7}] \label{TT1}  Let $(X, \alpha)\in {}_{L^1(\G)}\Mod$. \begin{enumerate}[noitemsep]
    \item The map
\begin{equation}\label{TTiso}
    \Phi: X\rtimes_\alpha^\F \G \rtimes_{\alpha^\rtimes}^\F \check{\G} \to \operatorname{Sat}(X, \alpha)\ovot_\mcF B(L^2(\G)): z \mapsto \vv_{23}z\vv_{23}^*
\end{equation}
is a completely isometric isomorphism. Endowing $X\rtimes_\alpha^\mcF \G\rtimes_{\alpha^\rtimes}^\mcF \check{\G}$ with its natural $\G$-action $$\kappa: X\rtimes_\alpha^\mcF \G\rtimes_{\alpha^\rtimes}^\mcF \check{\G}\to (X\rtimes_\alpha^\mcF \G\rtimes_{\alpha^\rtimes}^\mcF \check{\G})\ovot_\mcF L^\infty(\G): z\mapsto \ww_{43}z_{123}\ww_{43}^*$$ (see Remark \ref{doublecrossed}) and $\Sat(X, \alpha)\ovot_\mcF B(L^2(\G))$ with its natural $\G$-action $\lambda: \Sat(X, \alpha)\ovot_\mcF B(L^2(\G))\to \Sat(X, \alpha)\ovot_\mcF B(L^2(\G))\ovot_{\mcF} L^\infty(\G)$ given by
$$\lambda(v)= \ww_{43}(\id \otimes \Delta \otimes \id)(v)_{1243}\ww_{43}^*= \ww_{43}\vv_{24}v_{123}\vv_{24}^*\ww_{43}^*$$ (see Proposition \ref{relations}), the map $\Phi$ is $\G$-equivariant.
\item The natural map
    $$\Psi:= \Phi^{-1}\circ (\alpha\otimes \id): X \ovot_\mcF B(L^2(\G)) \to X\rtimes_\alpha^\mcF \G \rtimes_{\alpha^\rtimes}^\mcF\check{\G}: z \mapsto \vv_{23}^*(\alpha\otimes \id)(z)\vv_{23}$$
    is $\G$-equivariant, where $X\ovot_\mcF B(L^2(\G))$ is endowed with the $\G$-action from Proposition \ref{relations} and where $X\rtimes_\alpha^\mcF \G \rtimes_{\alpha^\rtimes}^\mcF\check{\G}$ is endowed with the $\G$-action $\kappa$ considered above. If moreover $(X, \alpha)\in {}_{L^1(\G)}\Mod^{\|\cdot\|}$, then $\Psi$ is surjective if and only if $(X, \alpha)$ is saturated. 
\end{enumerate}
\end{Prop}
\begin{proof} (1)
    We endow $(X \rtimes_\alpha^\F \G)\ovot_\mcF B(L^2(\G))$ with two $\check{\G}$-actions
    $$\gamma(z)= \vv_{43}(\id \otimes \check{\Delta}_r\otimes \id)(z)_{1243}\vv_{43}^*, \quad \delta(z)= (\id \otimes \check{\Delta}_r\otimes \id)(z)_{1243}.$$
    We claim that
    $$\Phi: ((X \rtimes_\alpha^\F \G)\ovot_\mcF B(L^2(\G)), \gamma)\to ((X \rtimes_\alpha^\F \G)\ovot_\mcF B(L^2(\G)), \delta): z \mapsto \vv_{23}z\vv_{23}^*$$
    is $\check{\G}$-equivariant. Indeed, since $\vv\in L^\infty(\check{\G})\ovot L^\infty(\G)$, the map $\Phi$ has the correct codomain. Moreover, as 
    $(\check{\Delta}_r\otimes \id)(\vv) = (\check{\Delta}\otimes \id)(\vv) = \vv_{12}^*\vv_{23}\vv_{12}= \vv_{13}\vv_{23}$, for $z\in (X \rtimes_\alpha^\F \G)\ovot_\mcF B(L^2(\G))$ we find
    \begin{align*}
        \delta(\Phi(z))&= (\id \otimes \check{\Delta}_r \otimes \id)(\vv_{23}z \vv_{23}^*)_{1243}\\
        &= (\vv_{24}\vv_{34}(\id \otimes \check{\Delta}_r\otimes \id)(z)\vv_{34}^*\vv_{24}^*)_{1243}\\
        &= \vv_{23}\vv_{43}
 (\id \otimes \check{\Delta}_r\otimes \id)(z)_{1243} \vv_{43}^* \vv_{23}^*= (\Phi\otimes \id)\gamma(z).
 \end{align*}
 Since $\Phi$ is a $\check{\G}$-equivariant isomorphism, it restricts to an isomorphism
 $$\Phi: \Fix((X \rtimes_\alpha^\F \G)\ovot_\mcF B(L^2(\G)), \gamma)\to \Fix((X \rtimes_\alpha^\F \G)\ovot_\mcF B(L^2(\G)), \delta).$$
However, using the fact that $\check{\ww}= \vv$ and thus $\check{\Delta}_l$ is implemented by $\vv$, we find that
\begin{align*}
    \Fix((X \rtimes_\alpha^\F \G)\ovot_\mcF B(L^2(\G)), \gamma) &= \{z\in (X \rtimes_\alpha^\F \G) \ovot_\mcF B(L^2(\G))\mid (\alpha^\rtimes \otimes \id)(z)= (\id_{X\rtimes_\alpha^\mcF \G} \otimes \check{\Delta}_l)(z) \}\\
    &=  X\rtimes_\alpha^\F \G \rtimes_{\alpha^\rtimes}^\F \check{\G}.
\end{align*}
On the other hand, it follows from Lemma \ref{relations} (2) that
$$\Fix((X \rtimes_\alpha^\F \G)\ovot_\mcF B(L^2(\G)),\delta)= \Fix(X\rtimes_\alpha^\F \G, \alpha^\rtimes)\ovot_\mcF B(L^2(\G))= \operatorname{Sat}(X, \alpha)\ovot_\mcF B(L^2(\G)).$$
It remains to show that $\Phi$ transforms the action $\kappa$ into the action $\lambda$. This follows from the following calculation, where $z\in X\rtimes_\alpha^\mcF\G \rtimes_{\alpha^\rtimes}^\mcF\check{\G}$,
\begin{align*}
    \lambda \Phi(z) &= \lambda(\vv_{23}z\vv_{23}^*) = \ww_{43}\vv_{24}\vv_{23}z_{123}\vv_{23}^*\vv_{24}^*\ww_{43}^*\\
    &= \vv_{23}\ww_{43} z_{123}\ww_{43}^* \vv_{23}^*= \vv_{23}\kappa(z)\vv_{23}^*= (\Phi\otimes \id)\kappa(z).
\end{align*}
(2) The statement about $\G$-equivariance of $\Psi$ follows since it is a composition of two $\G$-equivariant maps. If $(X, \alpha)\in {}_{L^1(\G)}\Mod^{\|\cdot\|}$, then we have the equalities
    \begin{align*}
        \Psi(X \ovot_\mcF B(L^2(\G)))&= \vv_{23}^*(\alpha(X)\ovot_\mcF B(L^2(\G)))\vv_{23}\\
        \vv_{23}^*(\operatorname{Sat}(X, \alpha)\ovot_\mcF B(L^2(\G)))\vv_{23}&= X\rtimes_\alpha^\F \G \rtimes_{\alpha^\rtimes}^\F \check{\G},
    \end{align*}
    where we used that $\alpha$ is completely isometric to ensure that $(\alpha\otimes \id)(X\ovot_\mcF B(L^2(\G))) = \alpha(X)\ovot_\mcF B(L^2(\G))$.
    Combining these equalities, the conclusion immediately follows.
\end{proof}

\begin{Prop}[cf.\ \cite{An23}*{Proposition 4.2, Proposition 4.6}]\label{TT2}
    Let $(X,\alpha)\in {}_{L^1(\G)}\NMod^{\|\cdot\|}$ 
and consider the natural map
$$\Psi: X \ovot B(L^2(\G))\to X \ovot B(L^2(\G))\ovot B(L^2(\G)): z \mapsto \vv_{23}^*(\alpha\otimes \id)(z)\vv_{23}.$$
Then
$\Psi(X_{\operatorname{nd}} \ovot B(L^2(\G)))= X\bar{\rtimes}_\alpha \G \bar{\rtimes}_{\alpha^\rtimes} \check{\G}.$ 
\end{Prop}
\begin{proof} We have the following chain of equalities:
\begin{align*}
    X\bar{\rtimes}_\alpha \G \bar{\rtimes}_{\alpha^\rtimes}\check{\G}&= [\alpha^{\rtimes}(\alpha(x)(1\otimes \check{y}))(1\otimes 1\otimes  z')\mid  x\in X, \check{y}\in L^\infty(\check{\G}), z'\in L^\infty(\G)']^{w*}\\
    &=[(\alpha(x)\otimes 1)(1\otimes \check{\Delta}(\check{y}))(1\otimes 1 \otimes z')\mid  x\in X, \check{y}\in L^\infty(\check{\G}), z'\in L^\infty(\G)']^{w*}\\
    &= [\Psi(\alpha(x)(1\otimes \check{y}z'))\mid x\in X, \check{y}\in L^\infty(\check{\G}), z'\in L^\infty(\G)']^{w*}\\
    &=\Psi([\alpha(x)(1\otimes \check{y}z'))\mid x\in X, \check{y}\in L^\infty(\check{\G}), z'\in L^\infty(\G)']^{w*})\\
    &= \Psi([\alpha(X)(1\otimes B(L^2(\G)))]^{w*})\\
    &= \Psi(X_{\operatorname{nd}} \ovot B(L^2(\G))),
\end{align*}
where the fourth equality follows from the fact that $\Psi$ maps weak$^*$-closed subspaces to weak$^*$-closed subspaces (here we use that $\alpha$ is a weak$^*$-continuous complete isometry), the fifth equality follows from 
$[L^\infty(\check{\G})L^\infty(\G)']^{w *} = B(L^2(\G))$ and the sixth equality follows from Lemma \ref{ndlemma}. 
\end{proof}

\subsection{Relations between crossed products} With the Takesaki-Takai dualities in place, we are in a position to derive some non-trivial relations between the different notions of crossed products. 

We start by proving that the Fubini crossed product can not detect saturation and the weak$^*$-crossed product can not detect non-degeneracy. These statements appear to be new for classical groups as well. 

\begin{Prop}\label{nondetection} \noindent
\begin{enumerate}[noitemsep]
\item If $(X, \alpha)\in {}_{L^1(\G)}\Mod^{\|\cdot\|}$, then 
    $\alpha(X)\rtimes_{\id \otimes \Delta}^\mcF \G= \Sat(X, \alpha)\rtimes_{\id \otimes \Delta}^\mcF \G$.
    \item If $(X,\alpha)\in {}_{L^1(\G)}\NMod$, then $X\bar{\rtimes}_\alpha\G = X_{\operatorname{nd}}\bar{\rtimes}_\alpha\G$.
\end{enumerate}
\end{Prop}
\begin{proof} (1)
The key observation is that
\begin{equation}\label{l1}
    \check{\vv}_{23}((X\rtimes_\alpha^\mcF \G)\ovot_\mcF \C)\check{\vv}_{23}^* =\alpha^\rtimes(X\rtimes_\alpha^\mcF \G) =  \vv_{23}^*(\Sat(X,\alpha)\rtimes_{\id \otimes \Delta}^\mcF \G)\vv_{23}.
\end{equation}
Indeed, we have (making use of the notations from Proposition \ref{TT1}) that
\begin{align*}
    \alpha^\rtimes(X\rtimes_\alpha^\mcF\G) &= \Sat(X\rtimes_\alpha^\mcF \G, \alpha^\rtimes)\\
    &= \Fix(X\rtimes_\alpha^\mcF \G\rtimes_{\alpha^\rtimes}^\mcF\check{\G}, \alpha^{\rtimes\rtimes})\\
    &= \Fix(X\rtimes_\alpha^\mcF \G \rtimes_{\alpha^\rtimes}^\mcF\check{\G}, \kappa)\\
    &= \vv_{23}^* \Fix(\Sat(X, \alpha)\ovot_\mcF B(L^2(\G)), \lambda) \vv_{23}\\
    &= \vv_{23}^*(\Sat(X, \alpha)\rtimes_{\id \otimes \Delta}^\mcF  \G)\vv_{23},
\end{align*}
where the first equality follows from Proposition \ref{prop1'}, the second equality follows from Proposition \ref{relations}, the third equality follows from Remark \ref{doublecrossed}, the fourth equality follows from the $\G$-equivariance of the isomorphism \eqref{TTiso} and the fifth equality follows from Proposition \ref{relations}. Consequently, \eqref{l1} is proven. 

It then follows that
\begin{align*}
    \check{\vv}_{34}((\Sat(X,\alpha)\rtimes_{\id \otimes \Delta}^\mcF \G)\ovot_\mcF \C)\check{\vv}_{34}^* &= \vv_{34}^*(\Sat(\Sat(X, \alpha), \id \otimes \Delta)\rtimes^\mcF_{\id \otimes \id \otimes \Delta} \G)\vv_{34}\\
    &= \vv_{34}^* ((\id \otimes \Delta)(\Sat(X, \alpha))\rtimes^\mcF_{\id \otimes \id \otimes \Delta} \G)\vv_{34}\\
    &=  \vv_{34}^* ((\alpha \otimes \id)(\Sat(X, \alpha))\rtimes^\mcF_{\id \otimes \id \otimes \Delta} \G)\vv_{34}\\
    &\subseteq  \vv_{34}^* (\Sat(\alpha(X), \id \otimes \Delta)\rtimes^\mcF_{\id \otimes \id \otimes \Delta} \G)\vv_{34}\\
    &= \check{\vv}_{34}((\alpha(X)\rtimes_{\id \otimes \Delta}^\mcF \G)\ovot_\mcF \C)\check{\vv}_{34}^*,
\end{align*}
where the first and fourth equality follow from \eqref{l1}, the second equality follows from Lemma \ref{saturationsaturated}, the third equality follows from the definition of $\Sat(X, \alpha)$ and the inclusion follows from the fact that if $\phi: Y \to Z$ is a $\G$-equivariant completely bounded map, then $(\phi\otimes \id)\Sat(Y) \subseteq \Sat(Z)$ (functoriality of $\Sat$).

(2) Given $\xi, \eta\in L^2(\G)$ and $x\in X$, we have
$$(1\otimes (\id \otimes\omega_{\xi, \eta})(\vv))\alpha(x)= \sum_{i\in I} \alpha(\omega_{\xi, e_i}\rhd x)(1\otimes (\id \otimes \omega_{e_i, \eta})(\vv))\in X_{\operatorname{nd}}\bar{\rtimes}_\alpha \G,$$
where $\{e_i\}_{i\in I}$ is an orthonormal basis for $L^2(\G)$ and where the series converges in the weak$^*$-topology.
\end{proof}

Also the following result is new for classical groups. Its proof is partially inspired on the proof of \cite{An21}*{Theorem 3.17}.

\begin{Prop}\label{ndfubini}
     If $(X, \alpha)\in {}_{L^1(\G)}\NMod$, then 
     $$(X\rtimes_\alpha^\mcF \G)_{\operatorname{nd}}= [\Sat(X, \alpha)(1\otimes L^\infty(\check{\G}))]^{w*} = X\bar{\rtimes}_\alpha \G.$$
\end{Prop}
\begin{proof}   Put $Y:= (X\rtimes_\alpha^\mcF \G)_{\operatorname{nd}}$.

(STEP I) We prove first that $Y \ovot B(L^2(\G))= \vv_{23}( Y\ovot B(L^2(\G)))\vv_{23}^*$.

Indeed, this will follow once we can prove that
$Y = [Y(1\otimes L^\infty(\check{\G}))]^{w*}$ and $Y = [(1\otimes L^\infty(\check{\G}))Y]^{w*}$.

Put $Z:= [Y(1\otimes L^\infty(\check{\G}))]^{w*}$. Given $y\in Y$ and $\check{x}\in L^\infty(\check{\G})$, it follows from Corollary \ref{cor1} that
    \begin{align*}
        \alpha^\rtimes(y(1\otimes \check{x}))&= \alpha^\rtimes(y)(1\otimes \check{\Delta}(\check{x}))\subseteq (Y\ovot L^\infty(\check{\G}))(\C\ovot L^\infty(\check{\G})\ovot L^\infty(\check{\G}))\subseteq Z\ovot L^\infty(\check{\G}),
    \end{align*}
    so $\alpha^\rtimes(Z)\subseteq Z\ovot L^\infty(\check{\G})$. Consequently, $Z\in {}_{L^1(\check{\G})}\NMod^{\|\cdot\|}$. We then compute
    \begin{align*}
&Z_{\operatorname{nd}}=[(\id \otimes \id \otimes \check{\mu})\alpha^\rtimes((\id\otimes \id \otimes \check{\nu})(\alpha^\rtimes(z))(1\otimes \check{x}))\mid \check{\mu}, \check{\nu}\in L^1(\check{\G}), \check{x}\in L^\infty(\check{\G}), z \in X\rtimes_\alpha^\mcF\G]^{w*}\\
        &= [(\id \otimes \id \otimes \check{\mu}\otimes \check{\nu})((\id \otimes \id \otimes \check{\Delta})(\alpha^\rtimes(z))(1\otimes \check{\Delta}(\check{x})\otimes 1))\mid \check{\mu}, \check{\nu}\in L^1(\check{\G}), \check{x}\in L^\infty(\check{\G}), z \in X\rtimes_\alpha^\mcF\G]^{w*}\\
        &= [(\id \otimes \id \otimes \check{\mu}\otimes \check{\nu})((\id \otimes \id \otimes \check{\Delta})(\alpha^\rtimes(z))(1\otimes \check{\Delta}(\check{x})(1\otimes \check{y})\otimes 1))\mid \check{\mu}, \check{\nu}\in L^1(\check{\G}), \check{x}, \check{y}\in L^\infty(\check{\G}), z \in X\rtimes_\alpha^\mcF\G]^{w*}\\
        &= [(\id \otimes \id \otimes \check{\mu}\otimes \check{\nu})((\id \otimes \id \otimes \check{\Delta})(\alpha^\rtimes(z))(1\otimes \check{x}\otimes \check{y}\otimes \check{z}))\mid \check{\mu}, \check{\nu}\in L^1(\check{\G}), \check{x}, \check{y}, \check{z}\in L^\infty(\check{\G}), z \in X\rtimes_\alpha^\mcF\G]^{w*}\\
        &= [(\id \otimes \id \otimes \check{\mu}\otimes \check{\nu})((\id \otimes \id \otimes \check{\Delta})(\alpha^\rtimes(z))(1\otimes \check{x}\otimes \check{\Delta}(\check{y})))\mid \check{\mu}, \check{\nu}\in L^1(\check{\G}), \check{x}, \check{y}\in L^\infty(\check{\G}), z \in X\rtimes_\alpha^\mcF\G]^{w*}\\
        &= [(\id \otimes \id \otimes (\check{\mu}\star \check{\nu}))(\alpha^\rtimes(z)(1\otimes \check{x}\otimes \check{y}))\mid \check{\mu}, \check{\nu}\in L^1(\check{\G}), \check{x}, \check{y}\in L^\infty(\check{\G}), z \in X\rtimes_\alpha^\mcF\G]^{w*}\\
        &= [(\id \otimes \id \otimes \check{\mu})(\alpha^\rtimes(z)(1\otimes \check{x}\otimes 1))\mid \check{\mu}\in L^1(\check{\G}), \check{x}\in L^\infty(\check{\G}), z \in X\rtimes_\alpha^\mcF\G]^{w*}= Z,
    \end{align*}
   so $Z$ is non-degenerate. But $Z\subseteq X\rtimes_\alpha^\mcF\G$, so we conclude by the maximality in Lemma \ref{ndlemma} that $Z\subseteq Y$, which shows that $Y = [Y(1\otimes L^\infty(\check{\G}))]^{w*}$. By a similar argument  $Y = [(1\otimes L^\infty(\check{\G}))Y]^{w*}$.

    (STEP II) We prove that $[\Sat(X, \alpha)(1\otimes L^\infty(\check{\G}))]^{w*}= Y$.

    Indeed, putting $S:= [\Sat(X, \alpha)(1\otimes L^\infty(\check{\G}))]^{w*}$, it is clear that $\alpha^\rtimes(S)\subseteq S\ovot L^\infty(\check{\G})$. The equality
    $$(\id \otimes \id \otimes \check{\omega})\alpha^\rtimes(z(1\otimes \check{x})) = z(1\otimes (\id \otimes \check{\omega})\check{\Delta}(\check{x})), \quad \check{\omega}\in L^1(\check{\G}), \quad z\in \Sat(X, \alpha), \quad \check{x}\in L^\infty(\check{\G})$$
    immediately shows that $S\in {}_{L^1(\check{\G})}\NMod^{\|\cdot\|}$ is non-degenerate. Since $S\subseteq X\rtimes_\alpha^\mcF\G$, another application of the maximality in Lemma \ref{ndlemma} gives $S\subseteq Y$.

Since $(Y, \alpha^\rtimes)$ is non-degenerate, Proposition \ref{lemma3} gives
    \begin{align*}
        Y \ovot B(L^2(\G))&= [\alpha^\rtimes(Y)(1\otimes 1\otimes B(L^2(\G)))]^{w*} \\
        &=[\alpha^\rtimes(Y)(1 \otimes 1 \otimes L^\infty(\G)')(1\otimes 1 \otimes L^\infty(\check{\G}))]^{w*}\subseteq [(Y \rtimes_{\alpha^\rtimes}^\F \check{\G})(1\otimes 1 \otimes L^\infty(\check{\G}))]^{w*}.
\end{align*}
Therefore, using (STEP I) and Proposition \ref{TT1}, we find
\begin{align*}
    Y \ovot B(L^2(\G))&= \vv_{23}( Y\ovot B(L^2(\G)))\vv_{23}^* \\
    &\subseteq [\vv_{23} (Y\rtimes_{\alpha^\rtimes}^\F \check{\G})\vv_{23}^* \vv_{23} (1\otimes 1 \otimes L^\infty(\check{\G}))\vv_{23}^*]^{w*}\\
    &\subseteq [(\operatorname{Sat}(X, \alpha)\ovot B(L^2(\G)))\vv_{23}(1\otimes 1\otimes L^\infty(\check{\G}))\vv_{23}^*]^{w*}\\
&\subseteq  [(\operatorname{Sat}(X, \alpha)\ovot B(L^2(\G)))(1 \otimes L^\infty(\check{\G})\ovot B(L^2(\G)))]^{w*}= S\ovot B(L^2(\G))\subseteq Y\ovot B(L^2(\G)),
\end{align*}
so we conclude that $S= Y$, as desired.

(STEP III) We prove that $S = X\bar{\rtimes}_\alpha\G$.

For $z\in \operatorname{Sat}(X, \alpha)\subseteq X \ovot_\F L^\infty(\G)$, $\check{y}\in L^\infty(\check{\G})$ and $\xi, \eta\in L^2(\G)$, we compute
\begin{align*}
  (\id \otimes \id \otimes \omega_{\xi,\eta})((\id \otimes \Delta_r)(z(1\otimes \check{y})))&=
    (\id \otimes \id \otimes \omega_{\xi,\eta})((\alpha \otimes \id)(z)(1\otimes \Delta_r(\check{y})))\\
    &= \sum_{i\in I} (\id \otimes \id \otimes \omega_{\xi, e_i})((\alpha\otimes \id)(z))(\id \otimes \id \otimes \omega_{e_i, \eta})(1\otimes \Delta_r(\check{y}))\\
    &= \sum_{i\in I} \alpha((\id \otimes \omega_{\xi, e_i})(z))(1 \otimes \underbrace{(\id \otimes \omega_{e_i, \eta})\Delta_r(\check{y})}_{\in L^\infty(\check{\G})}) \in X\bar{\rtimes}_\alpha \G,
\end{align*}
where $\{e_i\}_{i\in I}$ is an orthonormal basis of $L^2(\G)$ and the sum converges in the weak$^*$-topology. From this, we conclude that $(\id \otimes \Delta_r)(S)\subseteq (X\bar{\rtimes}_\alpha\G)\ovot_\mcF L^\infty(\G)$.
Conjugating with $\vv_{23}^*$ leads to
$$S\ovot \C \subseteq \vv_{23}^* ((X\bar{\rtimes}_\alpha \G) \ovot B(L^2(\G)))\vv_{23}\subseteq (X\bar{\rtimes}_\alpha \G)\ovot B(L^2(\G)),$$
and in particular $S\subseteq X\bar{\rtimes}_\alpha \G$.
\end{proof}

Note in particular that Proposition \ref{ndfubini} implies that $\Sat(X, \alpha)\subseteq X\bar{\rtimes}_\alpha\G$ for $(X, \alpha)\in {}_{L^1(\G)}\NMod$. The following is then immediately clear (recall also Proposition \ref{relations}):
\begin{Cor}[cf.\ \cite{An23}*{Proposition 3.10}]
    Given $(X,\alpha)\in {}_{L^1(\G)}\NMod$, we have $\Fix(X\bar{\rtimes}_\alpha\G, \alpha^\rtimes)= \Sat(X, \alpha)= \Fix(X\rtimes_\alpha^\mcF\G, \alpha^\rtimes)$.
\end{Cor}

\begin{Prop}[cf.\ \cite{An23}*{Corollary 4.8}]\label{saturatedreduced}
     If $(X, \alpha)\in {}_{L^1(\G)}\NMod$, then 
       $$\Sat(X\bar{\rtimes}_\alpha \G, \alpha^\rtimes)= \alpha^\rtimes(X\rtimes_\alpha^\mcF \G).$$
\end{Prop}
\begin{proof}
    We have
    \begin{align*}
        \Sat(X\bar{\rtimes}_\alpha \G, \alpha^\rtimes) &= \Sat((X\rtimes^\mcF_\alpha \G)_{\operatorname{nd}}, \alpha^\rtimes) = \Sat(X\rtimes^\mcF_\alpha \G, \alpha^\rtimes) =  \alpha^\rtimes(X\rtimes_\alpha^\mcF\G),
    \end{align*}
    where the first equality follows from Proposition \ref{ndfubini}, the second equality follows from Proposition \ref{saturationnd} and the third equality follows from  Proposition \ref{prop1'}.
\end{proof}

\begin{Rem}
Using Proposition \ref{Gcompletion}, Proposition \ref{saturatedreduced} can be reformulated by the statement that $X\rtimes^\mcF_\alpha\G$ is the $\check{\G}$-completion of $X\bar{\rtimes}_\alpha \G\in {}_{L^1(\check{\G})}\NMod^{\|\cdot\|}$ for every $X\in {}_{L^1(\G)}\NMod$.
\end{Rem}

\begin{Theorem}[cf.\ \cite{An23}*{Theorem 4.9}]\label{non-degenerate Fubini crossed} The following are equivalent for $(X, \alpha)\in {}_{L^1(\G)}\NMod$:
    \begin{enumerate}[noitemsep]
        \item $X\bar{\rtimes}_\alpha \G = X \rtimes_\alpha^\F \G.$
        \item $(X\rtimes_\alpha^\F \G, \alpha^\rtimes)$ is non-degenerate.
        \item $(X\bar{\rtimes}_\alpha\G, \alpha^\rtimes)$ is saturated. 
    \end{enumerate}
\end{Theorem}

\begin{proof}
   The implications $(1)\implies (2)+(3)$ are an immediate consequence of Proposition \ref{prop1} and Proposition \ref{prop1'}.
   The implication $(2)\implies (1)$ follows from Proposition \ref{ndfubini} and the implication
$(3)\implies (1)$ follows immediately from Proposition \ref{saturatedreduced}.
\end{proof}

We then find the following surprising consequence involving mixed iterated crossed products:
\begin{Cor}[cf.\ \cite{An23}*{Corollary 4.8}]\label{surprising} Let $(X, \alpha)\in {}_{L^1(\G)}\NMod$. Then
$$(X\bar{\rtimes}_\alpha\G)\rtimes_{\alpha^\rtimes}^\mcF \check{\G}= (X\rtimes_\alpha^\mcF\G)\rtimes_{\alpha^\rtimes}^\mcF\check{\G}, \quad (X\bar{\rtimes}_\alpha\G)\bar{\rtimes}_{\alpha^\rtimes}\check{\G}= (X\rtimes_\alpha^\mcF\G)\bar{\rtimes}_{\alpha^\rtimes}\check{\G}.$$
\end{Cor}

\begin{proof}
We have
\begin{align*}
    (\alpha^\rtimes\otimes \id)((X\bar{\rtimes}_\alpha\G)\rtimes_{\alpha^\rtimes}^\mcF \check{\G})&= \alpha^\rtimes(X\bar{\rtimes}_\alpha\G)\rtimes_{\id \otimes \id \otimes \check{\Delta}}^\mcF\check{\G}\\
    &= \Sat(X\bar{\rtimes}_\alpha\G, \alpha^\rtimes)\rtimes_{\id \otimes \id \otimes \check{\Delta}}^\mcF \check{\G}\\
    &= \alpha^\rtimes(X\rtimes_\alpha^\mcF\G)\rtimes_{\id \otimes \id \otimes \check{\Delta}}^\mcF \check{\G}\\
    &= (\alpha^\rtimes\otimes \id)((X\rtimes_\alpha^\mcF\G)\rtimes_{\alpha^\rtimes}^\mcF \check{\G}),
\end{align*}
where the second equality follows from Proposition \ref{nondetection} and where the third equality follows from Proposition \ref{saturatedreduced}. From this, the first equality follows. To prove the second equality, we note that
\begin{align*} (X\rtimes_\alpha^\mcF\G)\bar{\rtimes}_{\alpha^\rtimes}\check{\G}&= (X\rtimes_\alpha^\mcF\G)_{\operatorname{nd}}\bar{\rtimes}_{\alpha^\rtimes}\check{\G} = (X\bar{\rtimes}_\alpha\G)\bar{\rtimes}_{\alpha^\rtimes} \check{\G},
\end{align*}
where the first equality follows from Proposition \ref{nondetection} and the second equality follows from  Proposition \ref{ndfubini}.
\end{proof}

\section{Crossed product characterizations of the AP}\label{SectionCharacterizationsAP}

Let $\G$ be a locally compact quantum group. The work of the previous section culminates in the following characterization of the AP:

\begin{Theorem}[cf.\ \cite{An23}*{Proposition 3.26, Corollary 4.5, Theorem 4.12}]\label{main}\label{main'}
    The following conditions are equivalent:
    \begin{enumerate}[noitemsep]
        \item\label{i} $\check{\G}$ has the AP.
         \item\label{ii} Every $(X, \alpha)\in {}_{L^1(\G)}\NMod^{\|\cdot\|}$ is saturated.
                  \item\label{ii'} $\Sat(X, \alpha)= \overline{\alpha(X)}^{w*}$ for every $(X, \alpha)\in {}_{L^1(\G)}\NMod$.
        \item\label{iii} Every $(X, \alpha)\in {}_{L^1(\G)}\NMod^{\|\cdot\|}$ is non-degenerate. 
         \item\label{iv} Every non-degenerate $(X, \alpha)\in {}_{L^1(\G)}\NMod^{\|\cdot\|}$ is saturated.
                  \item\label{iv'} $\Sat(X, \alpha)= \overline{\alpha(X)}^{w*}$ for every non-degenerate $(X, \alpha)\in {}_{L^1(\G)}\NMod$.
  \item\label{v} Every saturated $(X, \alpha)\in {}_{L^1(\G)}\NMod^{\|\cdot\|}$ is non-degenerate.
        \item\label{vi} $(X\rtimes_\alpha^\F \G)\rtimes_{\alpha^\rtimes}^\F \check{\G}= (X\rtimes_\alpha^\F \G)\bar{\rtimes}_{\alpha^\rtimes} \check{\G}$ for all $(X, \alpha)\in {}_{L^1(\G)}\NMod^{\|\cdot\|}$.
                \item\label{vi'} $(X\rtimes_\alpha^\F \G)\rtimes_{\alpha^\rtimes}^\F \check{\G}= (X\rtimes_\alpha^\F \G)\bar{\rtimes}_{\alpha^\rtimes} \check{\G}$ for all $(X, \alpha)\in {}_{L^1(\G)}\NMod$.
                \item\label{vii} $(X\bar{\rtimes}_\alpha \G)\rtimes_{\alpha^\rtimes}^\F \check{\G}= (X\bar{\rtimes}_\alpha \G)\bar{\rtimes}_{\alpha^\rtimes} \check{\G}$ for all $(X, \alpha)\in {}_{L^1(\G)}\NMod^{\|\cdot\|}$.
\item\label{vii'} $(X\bar{\rtimes}_\alpha \G)\rtimes_{\alpha^\rtimes}^\F \check{\G}= (X\bar{\rtimes}_\alpha \G)\bar{\rtimes}_{\alpha^\rtimes} \check{\G}$ for all $(X, \alpha)\in {}_{L^1(\G)}\NMod^{}$.
        \item\label{viii} $Y\rtimes_\beta^\F \check{\G} = Y \bar{\rtimes}_\beta\check{\G}$ for all $(Y, \beta)\in {}_{L^1(\check{\G})}\NMod^{\|\cdot\|}$.
         \item\label{viii'} $Y\rtimes_\beta^\F \check{\G} = Y \bar{\rtimes}_\beta\check{\G}$ for all $(Y, \beta)\in {}_{L^1(\check{\G})}\NMod$.
    \end{enumerate}
\end{Theorem}
\begin{proof} The equivalences $\eqref{i}\iff \eqref{ii}\iff \eqref{iii}$ are established in \cite{An23}*{Proposition 3.26} in the context of classical groups, but the proof carries over to the setting of locally compact quantum groups. 

$\eqref{ii}\implies \eqref{viii'}$ If $\eqref{ii}$ holds, then every $(Z, \gamma)\in {}_{L^1(\check{\check{\G}})}\NMod^{\|\cdot\|}$ is saturated. In particular, given $(Y, \beta)\in {}_{L^1(\check{\G})}\NMod$, we see that $(Y\bar{\rtimes}_\beta\check{\G}, \beta^\rtimes)\in {}_{L^1(\check{\check{\G}})}\NMod^{\|\cdot\|}$ is saturated. By Theorem \ref{non-degenerate Fubini crossed}, we conclude that $Y\rtimes_\beta^\mcF\check{\G}= Y\bar{\rtimes}_\beta\check{\G}$.

Note that by Corollary \ref{surprising}, we have $\eqref{vi} = \eqref{vii}$ and $\eqref{vi'}= \eqref{vii'}$.
The implications $\eqref{viii'}\implies \eqref{viii}\implies \eqref{vii'}\implies \eqref{vii}$ are trivial. 

$\eqref{vi}\implies \eqref{ii}+\eqref{iii}$. Assume that $\eqref{vi}$ holds, then for $X\in {}_{L^1(\G)}\NMod^{\|\cdot\|}$, it follows from Proposition \ref{TT1}, Corollary \ref{surprising} and Proposition \ref{TT2} that
\begin{align*}
\vv_{23}^*(\Sat(X, \alpha)\ovot B(L^2(\G)))\vv_{23}&= (X\rtimes_\alpha^\mcF \G)\rtimes_{\alpha^\rtimes}^\mcF\check{\G}= (X\rtimes^\mcF_\alpha\G)\bar{\rtimes}_{\alpha^\rtimes} \check{\G}\\
&= (X\bar{\rtimes}_\alpha\G)\bar{\rtimes}_{\alpha^\rtimes} \check{\G}  =\vv_{23}^*(\alpha(X_{\operatorname{nd}})\ovot B(L^2(\G)))\vv_{23},\end{align*}
so $\Sat(X, \alpha)= \alpha(X_{\operatorname{nd}})$. This implies in particular that $(X, \alpha)$ is both non-degenerate and saturated.

$\eqref{v}\implies \eqref{ii'}$ Let $(X, \alpha)\in {}_{L^1(\G)}\NMod$. By assumption, $(\Sat(X, \alpha), \id \otimes \Delta)\in {}_{L^1(\G)}\NMod^{\|\cdot\|}$ is non-degenerate, so that Lemma \ref{satnd} implies that
$\Sat(X, \alpha)= \Sat(X, \alpha)_{\operatorname{nd}}\subseteq \overline{\alpha(X)}^{w*}\subseteq \Sat(X, \alpha),$
whence $\Sat(X, \alpha)=\overline{\alpha(X)}^{w*}$.

$\eqref{iv'}\implies \eqref{ii'}$ Let $(X, \alpha)\in {}_{L^1(\G)}\NMod$. By assumption, $(X_{\operatorname{nd}}, \alpha)\in {}_{L^1(\G)}\NMod$ satisfies $\Sat(X_{\operatorname{nd}}, \alpha)= \overline{\alpha(X_{\operatorname{nd}})}^{w*}$. Consequently, using  Proposition \ref{saturationnd}, we obtain $$\Sat(X, \alpha)= \Sat(X_{\operatorname{nd}}, \alpha) = \overline{\alpha(X_{\operatorname{nd}})}^{w*}\subseteq \overline{\alpha(X)}^{w*}\subseteq \Sat(X, \alpha),$$ whence $\Sat(X, \alpha)= \overline{\alpha(X)}^{w*}.$ The implication $\eqref{iv}\implies \eqref{ii}$ is proven similarly.

The implications $\eqref{ii'}\implies \eqref{ii}$ and $\eqref{iv'}\implies \eqref{iv}$ follow since $\alpha(X)$ is weak$^*$-closed for $(X, \alpha)\in \NMod^{\|\cdot\|}$ and the implications $\eqref{ii'}\implies \eqref{iv'}$ and $\eqref{iii}\implies \eqref{v}$ are trivial.
\end{proof}

\begin{Rem} Any locally compact quantum group $\G$ admits saturated $X\in {}_{L^1(\G)}\NMod$ which is degenerate. Simply take any non-zero dual operator space $X$ and endow it with the $L^1(\G)$-module structure $\omega\rhd x=0$. Thus, one can not replace ${}_{L^1(\G)}\NMod^{\|\cdot\|}$ by ${}_{L^1(\G)}\NMod$ in points \eqref{iii}, \eqref{v} of Theorem \ref{main}.
\end{Rem}

We now use Theorem \ref{main} to give a homological interpretation of the AP. 

\begin{Prop}
    The following are equivalent:
    \begin{enumerate}[noitemsep]
    \item $\check{\G}$ has the AP.
        \item The functor $-\rtimes^\mcF \G: {}_{L^1(\G)}\NMod^{\|\cdot\|}\to {}_\C\Mod$ reflects left $1$-exactness.
        \item The functor $-\bar{\rtimes}\G: {}_{L^1(\G)}\NMod^{\|\cdot\|}\to {}_\C\Mod$ reflects left $1$-exactness.
    \end{enumerate}
\end{Prop}

\begin{proof} $(1)\implies (2)$ Suppose that 
the sequence
$$
\begin{tikzcd}
0  \arrow[rr, ""] & & X\rtimes_\alpha^\mcF \G \arrow[rr, "\iota\rtimes^\mcF \G"] &  & Y\rtimes_\beta^\mcF \G \arrow[rr, "\phi\rtimes^\mcF \G"] &  & Z\rtimes_\gamma^\mcF \G
\end{tikzcd}$$ is 1-exact, where $\iota: (X, \alpha)\to (Y, \beta)$ and $\phi: (Y, \beta)\to (Z, \gamma)$ are morphisms in ${}_{L^1(\G)}\NMod^{\|\cdot\|}$. Since $(\iota\rtimes^\mcF \G)\circ \alpha = \beta \circ \iota$, it is clear that $\iota: X\to Y$ is a  complete isometry. If $x\in X$, then
$$0 = (\phi\rtimes^\mcF \G)(\iota\rtimes^\mcF \G)\alpha(x)= \gamma(\phi\iota(x)),$$
so $\phi\iota(x)= 0$. From this, we see that $\operatorname{Im}(\iota)\subseteq \Ker(\phi)$. Conversely, if $y\in Y$ with $\phi(y)= 0$, then $(\phi\rtimes^\mcF \G)\beta(y)= \gamma(\phi(y))= 0$, whence $\beta(y) = (\iota\rtimes^\mcF \G)(s)$ for some $s\in X\rtimes_\alpha^\mcF \G$. Since $\beta(y)\in Y \bar{\otimes}_\mcF L^\infty(\G)$ and since $\iota$ is a complete isometry, it follows that
    \begin{equation}\label{eq9}
    s \in (X\rtimes_\alpha^\mcF \G)\cap (X \bar{\otimes}_\mcF L^\infty(\G)) = \operatorname{Sat}(X, \alpha)= \alpha(X),
    \end{equation}
    where the last equality follows from Theorem \ref{main}. Thus, $\beta(y)= (\iota \rtimes_\mcF \G)(\alpha(x))= \beta(\iota(x))$ for some $x\in X$, i.e.\ $y\in \operatorname{Im}(\iota)$. We conclude that
    $$\begin{tikzcd}
0  \arrow[rr, ""] & & X \arrow[rr, "\iota"] &  & Y \arrow[rr, "\phi"] &  & Z
\end{tikzcd}$$
is 1-exact.

$(1)\implies (3)$ The reasoning is analogous to the implication $(1)\implies (2)$, one only needs to replace the first equality in \eqref{eq9} by $(X\bar{\rtimes}_{\alpha}\G)\cap (X\bar{\otimes}_{\mcF} L^{\infty}(\G))\subseteq \Sat(X,\alpha)$.

    $(2)\implies (1)$ Let $(X, \alpha)\in {}_{L^1(\G)}\NMod^{\|\cdot\|}$. Consider the inclusion map $\iota: \alpha(X)\to \Sat(X, \alpha)$. Since $\alpha(X)\rtimes_{\id \otimes \Delta}^\mcF \G = \Sat(X, \alpha)\rtimes_{\id \otimes \Delta}^\mcF \G$ (Proposition \ref{nondetection}), the sequence
    $$
\begin{tikzcd}
0  \arrow[rr, ""] & &\alpha(X)\rtimes_{\id \otimes \Delta}^\mcF \G \arrow[rr, "\iota\rtimes^\mcF \G"] &  & \Sat(X, \alpha)\rtimes_{\id \otimes \Delta}^\mcF \G \arrow[rr] &  & 0
\end{tikzcd}$$
is 1-exact. By assumption, this implies that 
$$
\begin{tikzcd}
0  \arrow[rr, ""] & &\alpha(X) \arrow[rr, "\iota"] &  & \Sat(X, \alpha) \arrow[rr] &  & 0
\end{tikzcd}$$
is $1$-exact, i.e.\ $(X, \alpha)$ is saturated. By Theorem \ref{main'}, $\check{\G}$ has the AP.

$(3)\implies (1)$ Similarly, by making use of the fact that $X\bar{\rtimes}_\alpha\G = X_{\operatorname{nd}}\bar{\rtimes}_\alpha\G$ for $(X, \alpha)\in {}_{L^1(\G)}\NMod^{\|\cdot\|}$ (Proposition \ref{nondetection}).
\end{proof}

\begin{Rem}
Let $\G$ be a locally compact quantum group without the AP. Using Theorem \ref{main'}, we see that there exist $X,Y\in {}_{L^1(\check{\G})}\NMod^{\|\cdot\|}$ such that $X$ is degenerate and $Y$ is non-saturated. Their direct sum $X\oplus_\infty Y\in {}_{L^1(\check{\G})}\NMod^{\|\cdot\|}$ is then both degenerate and non-saturated. We omit the details, as the existence of examples which are both non-saturated and degenerate will also follow by combining Proposition \ref{prop7} and Theorem \ref{thm1}. This answers a question raised in \cite{An23}*{Remark 3.27}.
\end{Rem}

Let us now give some further applications of Theorem \ref{main'}:

\begin{Cor}[\cite{DKV24}*{Proposition 6.10}]\label{smp}
    If $\G$ has the AP, then $L^\infty(\check{\G})$ has the $W^*$-OAP. 
\end{Cor}
\begin{proof}
   Since $\G$ has the AP, it follows from Theorem \ref{main} that for any dual operator space $X$,
    $$X\ovot_\F L^\infty(\check{\G})= X\rtimes^\F_\tau \G = X\bar{\rtimes}_{\tau}\G = X \ovot L^\infty(\check{\G}),$$
    where $\tau: X \to X\ovot_\mcF L^\infty(\G): x \mapsto x \otimes 1$ is the trivial action. In other words, $L^\infty(\check{\G})$ has property $S_\sigma$, which is known to be equivalent with the $W^*$-OAP (see e.g.\ the discussion preceding \cite{HK94}*{Theorem 2.1}).
\end{proof}

\begin{Cor}[\cite{DKV24}*{Theorem 7.1}]
    If $\G$ has the AP and $\mathbb{H}$ is a closed quantum subgroup of $\G$ (in the sense of Vaes), then $\mathbb{H}$ also has the AP.
\end{Cor}
\begin{proof}
    Since $\mathbb{H}$ is a closed quantum subgroup of $\G$  in the sense of Vaes, there is a normal, unital, isometric $*$-homomorphism $\check{\gamma}: L^\infty(\check{\mathbb{H}})\to L^\infty(\check{\G})$ intertwining the coproducts. 
    
    Assume that $(X,\alpha)\in {}_{L^1(\check{\mathbb{H}})}\NMod^{\|\cdot\|}$. Then we can view $X$ as an object of ${}_{L^1(\check{\G})}\NMod^{\|\cdot\|}$ through the coaction
    $(\id \otimes \check{\gamma})\circ \alpha: X \to X \ovot_\F L^\infty(\check{\G}).$ Since $\G$ has the AP, $(X, (\id \otimes \check{\gamma})\circ \alpha)$ is non-degenerate (Theorem \ref{main'}), and thus
    $$X = [(\id \otimes \omega)((\id \otimes \check{\gamma})\alpha(x)) \mid \omega \in L^1(\check{\G}), x \in X]^{w*}=[(\id \otimes \eta)\alpha(x)\mid \eta \in L^1(\check{\mathbb{H}}), x \in X]^{w*},$$
    where the surjectivity of $\check{\gamma}_*: L^1(\check{\G})\to L^1(\check{\H})$ was used \cite{ER00}*{Corollary 4.1.9}.
    Thus, $(X, \alpha)$ is non-degenerate as well. It follows from Theorem \ref{main} that $\mathbb{H}$ has the AP.
\end{proof}

Observe that any functional $\omega\in L^1(\G)$ gives a normal completely bounded map $\Phi( \lambda(\omega))$ on $B(L^2(\G))$ via
$$
\Phi(\lambda(\omega))(x) = (\omega\otimes \id)\Delta_l(x), \quad \omega\in L^1(\G), \quad x\in B(L^2(\G)),$$
cf.~\cite{DKV24}*{Section 3.3}.
Towards another application of Theorem \ref{main}, consider a norm-closed right ideal $\mathcal{X}\subseteq L^1(\G)$ and write
$\mathcal{X}^\perp:= \{x\in L^\infty(\G)\mid \forall \omega\in \mathcal{X}: \omega(x)= 0\}$. It is easily verified that $\Delta(\mathcal{X}^\perp)\subseteq \mathcal{X}^\perp\ovot_\mcF L^\infty(\G)$, so that $(\mathcal{X}^\perp, \Delta)\in {}_{L^1(\G)}\NMod^{\|\cdot\|}$. 
Define next the weak$^*$-closed (algebraic) $L^\infty(\check{\G})$-$L^\infty(\check{\G})$-bimodule $$\operatorname{Bim}(\mathcal{X}^\perp):= [\check{x}z\check{y} \mid \check{x},\check{y} \in L^\infty(\check{\G}), z\in \mathcal{X}^\perp]^{w*}\subseteq B(L^2(\G)).$$

\begin{Cor} Let $\G$ be a locally compact  quantum group. Then we have the equalities $\Delta_l(\operatorname{Bim}(\mathcal{X}^\perp))= \mathcal{X}^\perp\bar{\rtimes}_\Delta \G$ and $\Delta_l\left(\bigcap_{\omega\in \mathcal{X}}\Ker\Phi(\lambda(\omega))\right)= \mathcal{X}^\perp\rtimes^\mcF_\Delta \G$. If $\G$ has the AP, then $\operatorname{Bim}(\mathcal{X}^\perp)= \bigcap_{\omega\in \mathcal{X}}\Ker\Phi(\lambda(\omega))$. 
\end{Cor}
\begin{proof}
    Since
$\Delta_l(\check{x}z\check{y}) = (1\otimes \check{x}) \Delta(z) (1\otimes \check{y})$ for $\check{x},\check{y} \in L^\infty(\check{\G})$  and $z\in \mathcal{X}^\perp\subseteq L^\infty(\G)$, it is clear that
$$ \Delta_l(\operatorname{Bim}(\mathcal{X}^\perp))= \mathcal{X}^\perp\bar{\rtimes}_\Delta \G.$$ 
Let us write $K:= \bigcap_{\omega\in \mathcal{X}}\Ker\Phi(\lambda(\omega))$.
We will now prove that $\Delta_l(K)= \mathcal{X}^\perp\rtimes_\Delta^\mcF \G$. Indeed, if $x\in K$, $\mu \in B(L^2(\G))_*$ and $\omega\in \mathcal{X}$, then $\omega\bigl((\id \otimes \mu)\Delta_l(x)\bigr) = \mu\bigl(\Phi(\lambda(\omega))(x)\bigr)= 0,$
so that $\Delta_l(x)\in \mathcal{X}^\perp\ovot_\mcF B(L^2(\G))$. The (left) coaction property $(\Delta\otimes \id)\Delta_l = (\id \otimes \Delta_l)\Delta_l$ then ensures that $\Delta_l(K)\subseteq \mathcal{X}^\perp\rtimes_\Delta^\mcF \G$. Conversely, if $z\in \mathcal{X}^\perp\rtimes_\Delta^\mcF\G\subseteq L^\infty(\G)\bar{\rtimes}_\Delta \G = \Delta_l(B(L^2(\G)))$, there is $x\in B(L^2(\G))$ such that $z= \Delta_l(x)$. Since $z\in \mathcal{X}^\perp\ovot_\mcF B(L^2(\G))$, it follows straightforwardly that $x\in K$. Thus, $\mathcal{X}^\perp\rtimes_\Delta^\mcF \G= \Delta_l(K)$. It follows from Theorem \ref{main} that $K = \operatorname{Bim}(\mathcal{X}^\perp)$ if $\G$ has the AP.
\end{proof}

The equality $\operatorname{Bim}(\mathcal{X}^\perp)= \bigcap_{\omega\in \mathcal{X}}\Ker\Phi(\lambda(\omega))$ was first proven in \cite{AKT19} for abelian, compact and weakly amenable discrete groups. It was proven for classical locally compact groups with the AP in \cite{CN22}*{Theorem 5.5} and for discrete quantum groups with the AP in \cite{CKN25}*{Theorem 4.4}. The equality remains open if $\G$ does not have the AP (even in the case of classical groups) and appears to be related to the question if $\G$ automatically has Ditkin's property at infinity (see e.g.\ the discussion in \cite{An23}*{Section 5}).

\section{$L^1(\G)$-decomposable approximation property}
\label{SectionDAP}
In this section, we aim to relate the AP of a locally compact quantum group $\G$ to an appropriate module-theoretic approximation property.

Before we can define this approximation property, we need to introduce some terminology. Recall the notion of the (forced) unitization $L^1(\G)_+=L^1(\G)\oplus_1 \C$ from Subsection \ref{operatormodules}. It is an object of ${}_{L^1(\G)}\Mod$ via $\omega\rhd(\nu,\lambda)=(\omega\star \nu+\lambda\omega,0)$ for $\omega,\nu\in L^1(\G),\lambda\in \C$. We emphasize that this $L^1(\G)$-module structure on $L^1(\G)_+$ is not the one induced from the natural $\G$-$W^*$-algebra structure on $L^\infty(\G)\oplus_\infty\C$ (cf.\ Example \ref{interestingex}). Then, for any operator space $F$, $L^1(\G)_+\hat{\otimes}F$ is an object of ${}_{L^1(\G)}\Mod$ via $\omega\rhd(\rho\otimes f)=(\omega\rhd \rho)\otimes f$ for $\omega\in L^1(\G),\rho\in L^1(\G)_+, f\in F$.

\begin{Def}\label{d1}
    A morphism $\phi:X\to Y$ between $X,Y \in {}_{L^1(\G)} \operatorname{Mod}$ is said to have \textit{finite $L^1(\G)$-rank} if there exists a finite-dimensional operator space $F$,  completely bounded $L^1(\G)$-module maps $u:X\to L^1(\G)_+\pten F$ and $v:L^1(\G)_+\pten F\to Y$ such that the diagram
\begin{equation}\label{e:diagram} \begin{tikzcd}
 & L^1(\G)_+\pten F \arrow[dr, "v"]  \\
 X \arrow[ru, "u"] \arrow[rr, "\phi"]  && Y
\end{tikzcd}
\end{equation}
commutes. 
\end{Def}

In the language of \cite{Cra21}, this property means that $\phi$ admits a completely bounded factorization through a finitely generated free $L^1(\G)$-module, and hence may be viewed as an operator module analogue of a finite-rank map. The set ${}_{L^1(\G)}\mathcal{F}(X,Y)$ of finite $L^1(\G)$-rank morphisms forms a subspace of ${}_{L^1(\G)}\CB(X,Y)$. If $X\in {}_{L^1(\G)}\Mod$ is essential and $L^1(\G)$ is unital (i.e.\ when $\G$ is discrete), one may replace $L^1(\G)_+\hat{\otimes} F$ by $L^1(\G)\hat{\otimes} F$ in the factorization \eqref{e:diagram}. This follows from the fact that $L^1(\G)$ is complemented in $L^1(\G)_+$ as an $L^1(\G)$-module so that a morphism $v: L^1(\G)\hat{\otimes} F \to Y$ automatically extends to a morphism $L^1(\G)_+\hat{\otimes} F\to Y$, and that, by essentiality of $X$, a morphism $u:X\to L^1(\G)_+\pten F$ automatically maps into $L^1(\G)\pten F$.

If $M$ is a von Neumann algebra, there is a vector space isomorphism
$$M_n(M_*)\to M_n(M)_*: [\omega_{ij}]\mapsto \left([m_{ij}]\mapsto \sum_{i,j=1}^n \omega_{ij}(m_{ij})\right).$$
We say that $[\omega_{ij}]\in M_n(M_*)$ is positive if it is positive when regarded as a functional on $M_n(M)$. Given von Neumann algebras $M$ and $N$, a bounded linear map $\varphi: M_*\to N_*$ is called completely positive (= cp) if $\varphi_n: M_n(M_*)\to M_n(N_*)$ maps positive elements of $M_n(M_*)$ to positive elements of $M_n(N_*)$ for all $n\ge 1$. We then write $\varphi\in \CP(M_*, N_*)$. It is easy to verify that $\varphi: M_*\to N_*$ is cp if and only if $\varphi^*: N\to M$ is cp in the usual sense. Identifying $L^1(\G)_+ \cong (L^\infty(\G)\oplus_\infty \C)_*$, we can then make sense of the following definition:
\begin{Def}\label{d2}
    In the case where $X=M_*$ is the predual of a left $\G$-$W^*$-algebra $M$, we say that $\phi$ in $_{L^1(\G)}\mathcal{F}(M_*,Y)$ is \textit{decomposable} if there exists a factorization (\ref{e:diagram}) for which the map $u:M_*\to L^1(\G)_+\pten F$ has the property that $(\id\ten f^*)u$ lies in the span of $_{L^1(\G)}\mathcal{CP}(M_*,L^1(\G)_+)$ for all $f^*\in F^*$. The resulting subspace of $_{L^1(\G)}\mathcal{F}(M_*,Y)$ will be denoted by  $_{L^1(\G)}\mathcal{DF}(M_*,Y)$.
\end{Def}

 Note that when $\G=\{e\}$ is trivial, ${}_{L^1(\{e\})}\mathcal{F}(M_*,Y)$ is the usual space of finite rank mappings $M_*\to Y$, and since every element in $(M_*)^*\cong M$ is a linear combination of positives, it follows that ${}_{L^1(\{e\})}\mathcal{F}(M_*,Y)={}_{L^1(\{e\})}\mathcal{DF}(M_*,Y)$. 
 
 Recall that for any operator spaces $W,Z$, the space $\mathcal{CB}(W,Z^*)$ carries a canonical weak$^*$-topology with respect to the identification $\mathcal{CB}(W,Z^*)\cong (W\widehat{\otimes} Z)^*$ \cite{ER00}*{Proposition 7.1.2}. Whenever we speak about the weak$^*$-topology on $\mathcal{CB}(W,Z^*)$ (or its subspaces), we will have in mind this particular predual. If $A$ is a completely contractive Banach algebra, $W\in {}_A\Mod$ and $Z\in \Mod_A$, then ${}_A \CB(W,Z^*)$ is weak$^*$-closed in $\CB(W,Z^*)$.

\begin{Def}\label{DAP}
    Let $(M, \alpha)$ be a left $\G$-$W^*$-algebra. We say that $M_*\in {}_{L^1(\G)}\Mod$ has the \emph{$L^1(\G)$-decomposable approximation property} ($L^1(\G)$-DAP) if 
$\overline{{}_{L^1(\G)}\mathcal{DF}(M_*, X)}^{w*} = {}_{L^1(\G)}\CB(M_*, X)$
for every $X\in {}_{L^1(\G)}\NMod^{\|\cdot\|}$.
\end{Def}

Let us note that $\C\in {}_{L^1(\G)}\Mod$ trivially has the $L^1(\G)$-DAP. The main goal of this section is to show that there is a close connection between the approximation property of $\G$, the $L^1(\G)$-DAP of $B(L^2(\G))_*$ and the $L^1(\check{\G})$-DAP of $L^1(\check{\G})$.

We will use the notion of the weak$^*$-Haagerup tensor product, which we briefly recall. For a comprehensive introduction and proofs of the following claims, see \cite{BS92}. Let $\mcH,\mcK$ be Hilbert spaces. Considering the Haagerup tensor product $\otimes_h$ of operator spaces \cite{ER00}*{Chapter 9}, there is a canonical injective linear map $\theta: B(\mcH)\odot B(\mcK)\hookrightarrow (B(\mcH)_*\otimes_h B(\mcK)_*)^*$ given by $$\theta(x\otimes y)(x^*\otimes y^*) = x^*(x)y^*(y),  \quad x\in B(\mcH), \quad y\in B(\mcK), \quad x^*\in B(\mcH)_*, \quad y^*\in B(\mcK)_*.$$
We then define $B(\mcH)\otimes_{w^* h}B(\mcK)$ to be the weak$^*$-closure of $B(\mcH)\odot B(\mcK)$ in $(B(\mcH)_*\otimes_h B(\mcK)_*)^*$ (using the embedding $\theta$). In fact, we have the equality $B(\mcH)\otimes_{w^* h}B(\mcK)= (B(\mcH)_*\otimes_{ h}B(\mcK)_*)^*$. Next, if $M\subseteq B(\mcH)$, $N\subseteq B(\mcK)$ are von Neumann algebras, then $M\otimes_{w^* h} N$ is the weak$^*$-closure of $M\odot N\subseteq B(\mcH)\otimes_{w^* h} B(\mcK)$. The assignment $m\otimes n\mapsto \bigl( T\mapsto mT n\bigr)$ establishes completely isometric identifications
\[
M\otimes_{w^* h}N\cong 
{}_{M'}\CB_{N'}(K(\mcK), B(\mcH))\cong 
{}_{M'}\CB_{N'}^\sigma(B(\mcK), B(\mcH)),
\]
where the subscripts $M',N'$ indicate that the spaces on the right consist of left-$M'$ and right-$N'$ bimodule maps. We will use this identification in the sequel. Whenever $(m_i)_{i\in I}, (n_i)_{i\in I}$ are families in $M,N$ such that $\sum_{i\in I} m_i m_i^*$ and $\sum_{i\in I} n_i^* n_i$ converge in SOT, then $\sum_{i\in I} m_i \otimes n_i$ converges weak$^*$ in $M\otimes_{w^* h}N$. Every element of $M\otimes_{w^* h} N$ has such a description. The corresponding map acts by $B(\mcK)\ni T\mapsto \sum_{i\in I} m_i T n_i\in B(\mcH)$ and this series converges in $\sigma$-WOT.

Another notion we will need in this and the next section is \emph{(strong) regularity}, which we now briefly discuss. See \cites{Va05, Timm08, HHKT25} and references therein for more details. Let $\G$ be a locally compact quantum group and consider its canonical C$^*$-algebraic action $\Delta: \G\curvearrowright C_0(\G)$. We say that $\G$ is \emph{regular} if the reduced C$^*$-crossed product $\G\ltimes_r C_0(\G)$ is isomorphic to the algebra of compact operators $K(L^2(\G))$. Similarly, we say that $\G$ is \emph{strongly regular} if the full C$^*$-crossed product $\G\ltimes_f C_0(\G)$ is isomorphic to $K(L^2(\G))$ \cite[Section 2]{Va05}. It is known that strong regularity implies regularity \cite[Page 6]{Va05} and we obtain the same property, if we take this isomorphism to be canonical or just abstract (this easily follows from the representation theory of compact operators \cite[Theorem I.10.7]{Dav96}).

Algebraic quantum groups are regular \cite[Example 7.3.4 iv)]{Timm08} and quantum groups which are compact, discrete, classical or dual-to-classical are strongly regular. Let us provide a brief argument for this claim. First, by \cite[Proposition 2.2]{HHKT25}, it is enough to prove that $\G$ is strongly regular, when $\G$ is classical or compact. For classical groups, this is a well-known harmonic-analytic fact (see \cite[Theorem 4.24]{W07}). On the other hand, compact quantum groups are strongly amenable, hence both crossed products coincide \cite[Theorem 3.14]{Ch22} and it is enough to argue about regularity of $\G$. It holds, as compact quantum groups are algebraic quantum groups.

\begin{Prop}\label{prop4} 
    Let $\G$ be a strongly regular locally compact quantum group. The map
    \begin{equation}\label{restriction1}
        {}_{L^\infty(\hat{\G})}\mathcal{CP}^\sigma_{L^\infty(\hat{\G})}(B(L^2(\G)), B(L^2(\G)))\to \mathcal{CP}{}_{L^1(\G)}^\sigma(L^\infty(\G), B(L^2(\G))): \Phi\mapsto \Phi\vert_{L^\infty(\G)}
    \end{equation}
    is a well-defined bijection.
\end{Prop}

\begin{proof} Since $\ww\in L^\infty(\G)\ovot L^\infty(\hat{\G})$, it is clear that $${}_{L^\infty(\hat{\G})}\mathcal{CP}^\sigma_{L^\infty(\hat{\G})}(B(L^2(\G)), B(L^2(\G)))\subseteq \CP_{L^1(\G)}^\sigma(B(L^2(\G)), B(L^2(\G)),$$ so the map \eqref{restriction1} is well-defined. Since $[L^\infty(\G) L^\infty(\hat{\G})]^{w*} = B(L^2(\G))$, it is also clear that \eqref{restriction1} is injective. 

It remains to argue the surjectivity of \eqref{restriction1}. To this end, consider $\varphi \in \CP^\sigma_{L^1(\G)}(L^\infty(\G), B(L^2(\G)))$. By (the left version of) Theorem \ref{equivariant Stinespring}, there exists a Hilbert space $\mcH$, a normal, unital $*$-representation $\pi: L^{\infty}(\G)\rightarrow B(\mcH)$, a (left) unitary representation $\uu\in L^{\infty}(\G) \bar\otimes B(\mcH)$ and a bounded linear map $T:  L^2(\G)\rightarrow \mcH$ such that
$$
\varphi(x)= T^* \pi(x) T, \quad
(\id\otimes \pi) \Delta(x)=
\uu^*(1\otimes \pi(x))\uu,\quad 
\uu(1\otimes T)=(1\otimes T) \ww,
\quad x\in L^\infty(\G).$$
Since $\G$ is strongly regular, it follows from \cite{Va05}*{Theorem 5.1} applied to the trivial subgroup (see also the discussion preceding \cite{Va05}*{Theorem 5.1}), that there exist a Hilbert space $\mcK$ and a unitary $S: \mcH \to L^2(\G)\otimes \mcK $ such that $(1\otimes S) \uu = \ww_{12}(1\otimes S)$ and $S \pi(x)= ( x\otimes 1)S$ for all  $x\in L^\infty(\G)$. Choose an orthonormal basis $\{e_i\}_{i\in I}$ for $\mcK$ and define the bounded linear operators
$$
T_i:  L^2(\G)\otimes \mcK \to L^2(\G): \eta\otimes \xi \mapsto \langle e_i, \xi\rangle\, \eta.$$
Given $\omega \in L^1(\G)$, we have
\begin{align*}
    T_iST (\omega\otimes \id)(\ww) = T_i S (\omega\otimes \id)(\uu) T = T_i ((\omega\otimes \id)(\ww)\otimes 1) ST = (\omega\otimes \id)(\ww) T_i ST,
\end{align*}
so it follows that $\check{y}_i:= T_i ST \in L^\infty(\hat{\G})'= L^\infty(\check{\G})$. The series $\sum_{i\in I} \check{y}_i^* \check{y}_i $ converges strongly, hence we can consider $\sum_{i\in I} \check{y}_i^*\otimes \check{y}_i\in L^{\infty}(\check{\G})\otimes_{w^* h} L^{\infty}(\check{\G})$, which leads to the map
$$\Phi\in {}_{L^\infty(\hat{\G})}\CP_{L^\infty(\hat{\G})}^\sigma(B(L^2(\G)), B(L^2(\G))), \quad \Phi(x) = \sum_{i\in I} \check{y}_i^* x\check{y}_i, \quad x\in B(L^2(\G)).$$
Given $x\in L^\infty(\G)$, we have 
$
\Phi(x)=\sum_{i\in I} \check{y}_i^* x \check{y}_i =
\sum_{i\in I} T^* S^* T_i^* xT_i  ST = T^*S^*(x\otimes 1)ST = T^* \pi(x) T = \varphi(x)$. 
Thus, the map $\Phi$ extends $\varphi$. 
\end{proof}

\begin{Cor}\label{cor2}
    Let $\G$ be an amenable, strongly regular locally compact quantum group. The restriction map
    \begin{equation}\label{injection}
     L^\infty(\check{\G})\otimes_{w^* h} L^\infty(\check{\G})\cong {}_{L^\infty(\hat{\G})}\CB^\sigma_{L^\infty(\hat{\G})}(B(L^2(\G)), B(L^2(\G)))\to   \CB_{L^1(\G)}
    ^\sigma(L^\infty(\G), B(L^2(\G)))
    \end{equation}
    is bijective.
\end{Cor}

\begin{proof} By the same argument as in the proof of Proposition \ref{prop4}, the map \eqref{injection} is a well-defined injection.
Using a similar argument as in \cite{Cra17}*{Lemma 5.7} (making use of the essentiality of $B(L^2(\G))_*\in {}_{L^1(\G)}\Mod$), there is a completely isometric identification
\begin{equation}\label{eq6}
\CB_{L^1(\G)}^\sigma(L^\infty(\G), B(L^2(\G))) \cong \CB_{L^1(\G)}(C_0(\G), B(L^2(\G))): \varphi\mapsto \varphi\vert_{C_0(\G)}
\end{equation}
which preserves the property of being completely positive.
Since $\G$ is amenable, $B(L^2(\G))$ is injective in $\Mod_{L^1(\G)}$ by the $L^1(\G)$-variant of \cite[Theorem 5.5]{CN16} or by using \cite{DR26}*{Proposition 6.2}. Then a reasoning analogous to \cite[Proposition 5.5]{Cra17} gives
    $$ \CB_{L^1(\G)}(C_0(\G), B(L^2(\G))) = \operatorname{span}  \CP_{L^1(\G)}(C_0(\G), B(L^2(\G))).$$
The surjectivity of \eqref{injection} then follows from the surjectivity of \eqref{restriction1}.
\end{proof}

\begin{Rem}
    It is not clear if the bijection \eqref{injection} is (completely) isometric.
\end{Rem}

Recall from Proposition \ref{p:natural} that we can identify ${}_{L^1(\G)}\CB(B(L^2(\G))_*,X)=X\rtimes^\mcF_\alpha\G$ for all $(X, \alpha)\in {}_{L^1(\G)}\Mod.$ This will be used without further mention in the following result.

\begin{Theorem}\label{crossedproductfiniterank}
    Let $\G$ be a locally compact quantum group. For any $(X, \alpha)\in {}_{L^1(\G)} \operatorname{NMod}^{\|\cdot\|}$, \begin{equation}\label{inclusion}
        X\bar{\rtimes}_\alpha \G\subseteq \overline{{}_{L^1(\G)}\mathcal{DF}(B(L^2(\G))_*,X)}^{w*}.
    \end{equation}
The converse inclusion holds if $\G$ is strongly regular.
\end{Theorem}
\begin{proof}
    Let $x_1,\dotsc,x_n\in X$ and $\check{x}_1,\dotsc,\check{x}_n\in L^\infty(\check{\G})$. The map $\phi\in \ _{L^1(\G)}\CB(B(L^2(\G))_*,X)$ associated to $\sum_{j=1}^n\alpha(x_j)(1\ten \check{x}_j)\in X\bar{\rtimes}_\alpha \G$ is characterized by
$$\phi(\omega)=\sum_{j=1}^n(\id\ten \ \omega)(\alpha(x_j)(1\ten \check{x}_j))=\sum_{j=1}^n(\id\ten \ \check{x}_j\cdot\omega)\alpha(x_j)=\sum_{j=1}^n \pi(\check{x}_j\cdot\omega)\rhd x_j,$$
where $\pi:B(L^2(\G))_*\to L^1(\G)$ is the restriction map, and where the $L^1(\G)$-module action $\rhd$ corresponds to the coaction $\alpha$.
Now, for any $f\in L^1(\G)$ and $T\in B(L^2(\G))$, 
\begin{align*}\la T,\check{x}_j\cdot(f\rhd\omega)\ra&=
\la T\check{x}_j,f\rhd\omega\ra=
\la\Delta_l(T\check{x}_j), f\ten \omega\ra=
\la\Delta_l(T)(1\ten \check{x}_j), f\ten \omega\ra\\
&=\la \Delta_l(T),f\ten (\check{x}_j\cdot \omega)\ra=
\la T, f\rhd(\check{x}_j\cdot\omega)\ra,
\end{align*}
hence the map $B(L^2(\G))_*\ni \omega\mapsto\check{x}_j\cdot\omega\in B(L^2(\G))_*$ is an $L^1(\G)$-morphism. Since $\pi$ and the canonical inclusion $\iota:L^1(\G)\hookrightarrow L^1(\G)_+$ are also $L^1(\G)$-morphisms, we obtain an $L^1(\G)$-morphism $u:B(L^2(\G))_*\to L^1(\G)_+\pten \ \ell^1_n$ given by
$$u(\omega)=\sum_{j=1}^n\iota(\pi(\check{x}_j\cdot\omega))\ten e_j,\quad \omega \in B(L^2(\G))_*,$$
where $\{e_j\}_{j=1}^{n}$ denotes the canonical basis of $\ell^1_n:= \ell^1\text{-} \bigoplus_{j=1}^n \C$. Note that $\normof{u}_{\cb}\leq \sum_{j=1}^n\normof{\check{x}_j}$. 
Let $v:L^1(\G)_+\pten \ \ell^1_n\to X$ be given by
$$v(\sum_{j=1}^n (f_j,z_j)\ten e_j)=\sum_{j=1}^n \bigl(f_j\rhd x_j + z_jx_j\bigr), \ \ \ f_j\in L^1(\G), \ z_j\in\mathbb{C}.$$
Then $v$ is an $L^1(\G)$-morphism, $\normof{v}_{\cb}\leq\max_{1\leq j\leq n}\normof{x_j}$ and $\phi=v u$. Thus, $\phi$ has finite $L^1(\G)$-rank. Moreover, for every $f^*\in\ell^\infty_n=(\ell^1_n)^*$, the adjoint of the map
$$(\id\ten  f^*)u:B(L^2(\G))_*\ni \omega\mapsto\sum_{j=1}^nf^*(e_j)\iota(\pi(\check{x}_j\cdot\omega)) \in  L^1(\G)_+$$
is easily seen to be
$$L^\infty(\G)\oplus_\infty \C \ni (x,\lambda)\mapsto\sum_{j=1}^n f^*(e_j)x\check{x}_j\in B(L^2(\G)).$$
By the polarization identity with respect to the sesquilinear form $B(L^2(\G))^{\times 2}\ni (u,v)\mapsto u^*xv\in B(L^2(\G))$, we may write this map as
$$(x,\lambda)\mapsto\sum_{j=1}^n\frac{f^*(e_j)}{4}\sum_{k=0}^3i^k(1+(-i)^k\check{x}_j)^*x(1+(-i)^k\check{x}_j),$$
i.e., as a linear combination of completely positive $L^1(\G)$-module maps. Thus, $\phi\in \ _{L^1(\G)}\mathcal{DF}(B(L^2(\G))_*,X)$.

Since $X\bar{\rtimes}_\alpha \G$ is the weak*-closure of elements of the form $\sum_{j=1}^n\alpha(x_j)(1\ten \check{x}_j)$, and the isomorphism $X\rtimes_\alpha ^\mcF \G \cong {}_{L^1(\G)}\CB(B(L^2(\G))_*, X)$ is weak*-weak* homeomorphic, the inclusion \eqref{inclusion} follows.

We now prove the converse inclusion under the additional assumption that $\G$ is strongly regular.

Take $\phi\in {}_{L^1(\G)}\mathcal{DF}(B(L^2(\G))_*,X)$. Since the left hand side of \eqref{inclusion} is weak$^*$-closed, it is enough to prove $\phi\in X\bar{\rtimes}_\alpha \G$. By definition, there is a finite-dimensional operator space $F$ and completely bounded $L^1(\G)$-module maps $u: B(L^2(\G))_*\rightarrow L^1(\G)_+\wh{\otimes} F$, $v: L^1(\G)_+\wh{\otimes} F\rightarrow X$ such that $\phi=v u$ and additionally $(\id\otimes f^*)u \in \lin {}_{L^1(\G)} \CP(B(L^2(\G))_* ,L^1(\G)_+)$ for $f^*\in F^*$.
Choose a basis $\{e_p\}_{p=1}^{n}$ in $F$ with dual basis $\{e_p^*\}_{p=1}^{n}$ in $F^*$, and define $u_p=(\id\otimes e_p^*)u: B(L^2(\G))_*\to L^1(\G)_+$ and $v_p=v(\cdot\otimes e_p): L^1(\G)_+\to X$. Next, write $u_p=\sum_{k=0}^{3} i^k  u_{p,k}$ where $u_{p,k}\in {}_{L^1(\G)} \CP(B(L^2(\G))_*, L^1(\G)_+)$. Since $\phi=\sum_{p=1}^{n} \sum_{k=0}^{3} i^k v_p u_{p,k}$, it is enough to prove
\begin{equation}\label{eq4}
v_p \; u_{p,k}\in X\bar{\rtimes}_\alpha\G,\quad 1\le p \le n, \quad 0\le k \le 3.
\end{equation}
Consequently, we fix $p,k$ and prove \eqref{eq4} for this choice.

  Since $B(L^2(\G))_*\in {}_{L^1(\G)}\Mod$ is essential and $\omega \star L^1(\G)_+ \subseteq L^1(\G)$ for any $\omega\in L^1(\G)$, we have in fact $u_{p,k}\in {}_{L^1(\G)}\CP(B(L^2(\G))_*,L^1(\G))$. Thus, $u_{p,k}^*\in \CP_{L^1(\G)}^\sigma(L^\infty(\G), B(L^2(\G)))$ and Proposition \ref{prop4} allows us to find a family $\check{y}_{ j}\in L^{\infty}(\check{\G})\,(j\in J )$ such that $\sum_{j\in J } \check{y}_{ j}^* \check{y}_{ j}$ converges in SOT and
\[
u^*_{p,k}(x)=\sum_{j\in J } \check{y}_{ j}^* \,x\, \check{y}_{j},\quad x\in L^\infty(\G).
\]
 Let $\mcF$ be the directed set of finite, non-empty subsets of $J$. For $F\in \mcF$ define
$\phi_{p,k,F}: B(L^2(\G))_*\rightarrow X$ by
$\phi_{p,k,F}(\rho)= \sum_{j\in F} v_p(\iota(\pi(\check{y}_j \cdot \rho\cdot  \check{y}_j^*)))$ for $\rho\in B(L^2(\G))_*$. 
For $\rho\in B(L^2(\G))_*$ we have
\[
\phi_{p,k,F}(\rho)=
\sum_{j\in F} v_p (\iota( \pi(\check{y}_j\cdot \rho \cdot \check{y}_j^*)) )=
\sum_{j\in F} v_p ( \pi(\check{y}_j \cdot\rho \cdot\check{y}_j^*)\rhd  \eps)=
\sum_{j\in F} 
\pi(\check{y}_j\cdot \rho \cdot \check{y}_j^*)\vartriangleright
v_p ( \eps),
\]
where $\eps=(0,1)\in L^1(\G)_+$ is the unit. This map corresponds to $\sum_{j\in F} (1\otimes \check{y}_j^*) \alpha(v_p(\eps)) (1\otimes \check{y}_j)\in X\bar{\rtimes}_\alpha \G$. To conclude \eqref{eq4}, it is left to prove that $\phi_{p,k,F}\xrightarrow[F\in \mcF]{} v_p u_{p,k}$ weak$^*$. Equivalently, take $\Omega\in B(L^2(\G))_*\wh{\otimes} X_*$. We need to show
\begin{equation}\label{eq5}
\la \phi_{p,k,F} - v_p u_{p,k} , \Omega \ra \xrightarrow[F\in \mcF]{}0.
\end{equation}
Observe that
\begin{align*}
    \|\phi_{p,k,F}\|_{\cb}&\le \|v_p\|_{\cb}
\bigl\|B(L^2(\G))_*\ni \rho\mapsto \sum_{j\in F} \pi(\check{y}_j\cdot \rho\cdot \check{y}_j^*)\in L^1(\G)\bigr\|_{\cb}\\
&\le \|v_p\|_{\cb}
\bigl\|B(L^2(\G))_*\ni \rho\mapsto \sum_{j\in F} \check{y}_j\cdot \rho\cdot \check{y}_j^*\in B(L^2(\G))_*\bigr\|_{\cb}\\
&=  \|v_p\|_{\cb}\bigl\|B(L^2(\G))\ni x\mapsto \sum_{j\in F} \check{y}_j^*x\check{y}_j\in B(L^2(\G))\bigr\|_{\cb}\\
&= 
\|v_p\|_{\cb} \bigl\|\sum_{j\in F} \check{y}_j^* \check{y}_j\bigr\|\le 
\|v_p\|_{\cb} \bigl\|\sum_{j\in J} \check{y}_j^* \check{y}_j\bigr\|,
\end{align*}
hence the net $(\phi_{p,k,F})_{F\in \mcF}$ is uniformly bounded and by approximation it is enough to consider $\Omega=\Omega_1\otimes \Omega_2\in B(L^2(\G))_*\odot X_*$. With this reduction, we calculate
\[\begin{split}
&\quad\;
\la \phi_{p,k,F} - v_p u_{p,k} , \Omega \ra =
\la \sum_{j\in F}\iota \pi(\check{y}_j \cdot \check{y}_j^*) - u_{p,k} , 
\Omega_1  \otimes v_p^*(\Omega_2) \ra=
\la 
v_p^*(\Omega_2) , 
\bigl(\sum_{j\in F}\iota \pi(\check{y}_j \cdot \check{y}_j^*) - u_{p,k} \bigr)
\Omega_1  \ra\\
&=
\la 
P(v_p^*(\Omega_2)) , 
\bigl(\sum_{j\in F}\pi(\check{y}_j \cdot \check{y}_j^*) - u_{p,k} \bigr)
\Omega_1  \ra=
\bigl\la 
\sum_{j\in F} \check{y}_j^* P(v_p^*(\Omega_2))  \check{y}_j  - u_{p,k}^* \bigl(P(v_p^*(\Omega_2)) \bigr), \Omega_1 \bigr\ra 
\xrightarrow[F\in \mcF]{}0,
\end{split}\]
where $P: (L^1(\G)_+)^*\rightarrow L^\infty(\G)$ is the canonical projection. This shows the convergence \eqref{eq5} and ends the proof.
\end{proof}

\begin{Theorem}\label{BDAP}
    Let $\G$ be a locally compact quantum group. If $\G$ has the AP, then $B(L^2(\G))_*$ has the $L^1(\G)$-DAP. The converse holds if $\G$ is strongly regular. 
\end{Theorem}
\begin{proof}
    Suppose that $\G$ has the AP. For any $(X, \alpha)\in {}_{L^1(\G)} \operatorname{NMod}^{\|\cdot\|}$, it follows from Theorem \ref{main} and Theorem \ref{crossedproductfiniterank} that
$${}_{L^1(\G)}\CB(B(L^2(\G))_*,X)=X\rtimes_\alpha^\mcF \G = X\bar{\rtimes}_\alpha\G\subseteq \overline{{}_{L^1(\G)}\mathcal{DF}(B(L^2(\G))_*,X)}^{w*},$$
so $B(L^2(\G))_*$ has the $L^1(\G)$-DAP. Conversely, assume that $\G$ is strongly regular and that $B(L^2(\G))_*$ has the $L^1(\G)$-DAP. If $(X, \alpha)\in {}_{L^1(\G)}\NMod^{\|\cdot\|}$, then Theorem \ref{crossedproductfiniterank} shows that
$$X\bar{\rtimes}_\alpha \G = \overline{{}_{L^1(\G)}\mathcal{DF}(B(L^2(\G))_*,X)}^{w*} = {}_{L^1(\G)}\CB(B(L^2(\G))_*,X)= X\rtimes_\alpha^\mcF \G.$$
Theorem \ref{main} implies that $\G$ has the AP.
    \end{proof}

Our next goal will be to prove that $\G$ has the AP if and only if $L^1(\check{\G})$ has the $L^1(\check{\G})$-DAP (Theorem \ref{LDAP}). To prove this, we will make a detour through some general theory, which is of independent interest.

Consider the half-lifted right Kac-Takesaki operator $\vV\in \M(C_0(\check{\G})\otimes C_0^u(\G))$. It can also be seen as a unitary in the von Neumann algebra $L^{\infty}(\check{\G})\bar\otimes C_0^u(\G)^{**}$. Consequently, we can look at the von Neumann algebraic version of the half-lifted comultiplication
\[
\Delta^{r,u,\vN}: L^{\infty}(\G)\ni x\mapsto \vV(x\otimes 1)\vV^*\in L^{\infty}(\G)\bar\otimes C_0^u(\G)^{**}.
\]
Similarly, we can lift the comultiplication $\Delta$ to
\[
\Delta^{u,r,\vN}: L^{\infty}(\G)\ni x\mapsto \Ww^* (1\otimes x) \Ww \in C_0^u(\G)^{**}\bar\otimes L^{\infty}(\G),
\]
using the half-lifted $\Ww\in \M(C_0^u(\G)\otimes C_0(\hat{\G}))\subseteq C_0^u(\G)^{**}\bar\otimes L^{\infty}(\hat{\G})$.

More generally, when $(M, \alpha)$ is a right $\G$-$W^*$-algebra, we can consider its \emph{universal lift} 
\[
\alpha^u: M\rightarrow M\bar\otimes C_0^u(\G)^{**},
\]
which is the unique normal, unital, injective $*$-homomorphism satisfying $(\alpha\otimes\id)\alpha^u=(\id\otimes\Delta^{r,u,\vN})\alpha$. This construction was introduced in \cite[Section 4.3]{DCK24} and plays an important role in \cite{DR25EW}*{Section 4}. In particular, the universal lift $\alpha^u$ always exists in the von Neumann algebraic setting; the notion of universal lift is very useful when there is a need to combine C$^*$ and W$^*$-techniques.

In what follows, we introduce an analogous construction for actions on dual operator spaces. However, unlike in the von Neumann algebraic case, the existence of $\alpha^u$ turns out to be a non-trivial property related to the AP, non-degeneracy, saturation and decomposable approximation property.

\begin{Def}
We say that $(X,\alpha)\in {}_{L^1(\G)} \NMod^{\|\cdot\|}$ (or simply $\alpha$) \emph{admits the universal lift} if there is a completely isometric, normal map $\alpha^u: X\rightarrow X\bar{\otimes}_{\mcF} C_0^u(\G)^{**}$ satisfying $(\alpha\otimes\id)\alpha^u=(\id\otimes \Delta^{r,u,\vN})\alpha$.
\end{Def}

\begin{Rem}
The equation $(\alpha\otimes\id)\alpha^u=(\id\otimes\Delta^{r,u,\vN})\alpha$ shows that if a universal lift $\alpha^u$ exists, it is unique.
\end{Rem}

Recall from Proposition \ref{p:natural} that we can identify ${}_{L^1(\G)}\CB(L^1(\G),X)=\Sat(X, \alpha)$ for all $(X, \alpha)\in {}_{L^1(\G)}\Mod.$ This will be used without further mention in the following result.

\begin{Prop}\label{prop5}
Let $(X,\alpha)\in {}_{L^1(\G)}\NMod^{\|\cdot\|}$. The following conditions are equivalent:
\begin{enumerate}[noitemsep]
\item $\alpha$ admits the universal lift.
\item ${}_{L^1(\G)}\mathcal{DF}(L^1(\G),X)=\alpha(X)$.
\end{enumerate}
\end{Prop}

\begin{proof} We start by noting that the inclusion $\alpha(X)\subseteq {}_{L^1(\G)}\mathcal{DF}(L^1(\G),X)$ is always true. Indeed, given $\alpha(x)\in\alpha(X)$, the associated map is $L^1(\G)\ni \omega\mapsto (\id\otimes\omega)\alpha(x)=\omega\rhd x\in X$. It can be written as the composition of
\[
L^1(\G)\ni \omega\mapsto (\omega,0)\otimes 1\in L^1(\G)_+\hat{\otimes}\C\quad\textnormal{and}\quad L^1(\G)_+\hat{\otimes}\C\ni (\omega,\lambda)\otimes 1\mapsto \omega\rhd x+\lambda x\in X,
\]
hence $\alpha(x)\in {}_{L^1(\G)}\mathcal{DF}(L^1(\G),X)$. 

$(1)\implies (2)$ Assume that $\alpha$ admits the universal lift. Let us prove the inclusion ${}_{L^1(\G)}\mathcal{DF}(L^1(\G),X)\subseteq \alpha(X)$. Arguing as in the proof of Theorem \ref{crossedproductfiniterank}, we see that it is enough to show that if $u: L^1(\G)\rightarrow L^1(\G)_+$ and $v: L^1(\G)_+\rightarrow X$ are completely bounded $L^1(\G)$-module maps and $u$ is completely positive, then $vu: L^1(\G)\rightarrow X$ is in $\alpha(X)$.

First observe that as $L^1(\G)$ is an essential $L^1(\G)$-module, the image of $u$ lies in $L^1(\G)$. Consequently, by \cite[Theorem 5.2]{D12}, there is $\mu\in C_0^u(\G)^* = (C_0^u(\G)^{**})_*$ such that $u(\omega)=\iota(\omega\star \mu)$ for $\omega\in L^1(\G)$. We obtain
\[\begin{split}
&\quad\;
vu(\omega)=v(\iota(\omega\star \mu))=(\omega\star \mu)\rhd v(\eps)=
(\id\otimes (\omega\star \mu))\alpha( v(\eps))=
(\id\otimes (\omega\otimes\mu)\Delta^{r,u,\vN})\alpha( v(\eps))\\
&=
(\id\otimes \omega\otimes\mu)
(\alpha\otimes \id) \alpha^u( v(\eps))=
(\id\otimes \omega)\alpha\bigl(
(\id\otimes\mu)\alpha^u(v(\eps))\bigr),
\end{split}\]
where $\eps\in L^1(\G)_+$ is the unit. This means that $vu$ corresponds to $\alpha\bigl( (\id\otimes\mu)\alpha^u(v(\eps)) \bigr)\in \alpha(X)$.

$(2)\implies (1)$ Assume that ${}_{L^1(\G)}\mathcal{DF}(L^1(\G),X)=\alpha(X)$. Without loss of generality, we may assume that $X\subseteq B(\mcH)$ is a weak$^*$-closed subspace and that there is a unitary representation $\mathrm{U}\in B(\mcH)\bar\otimes L^{\infty}(\G)$ such that $\alpha(x)=\mathrm{U}(x\otimes 1)\mathrm{U}^*$ for all $x\in X$. Consider the half-lifted unitary $\uU\in M(K(\mcH)\otimes C_0^u(\G))\subseteq B(\mcH)\ovot C_0^u(\G)^{**}$. 
Define the normal, completely isometric map
\[
\alpha^u: X\ni x\mapsto \uU(x\otimes 1)\uU^*\in B(\mcH)\bar\otimes_{\mcF} C_0^u(\G)^{**}.
\]
The crux is to show that $\alpha^u(X)\subseteq X\ovot_\mcF C_0^u(\G)^{**}$. To this end, fix $x\in X$ and $\mu\in C_0^u(\G)^* = (C_0^u(\G)^{**})_*$ and put $y:=(\id\otimes \mu)\alpha^u(x)\in B(\mcH)$. Define the completely bounded $L^1(\G)$-module maps
\[
u: L^1(\G)\ni \omega\mapsto (\omega\star \mu,0)\in L^1(\G)_+,\quad 
v: L^1(\G)_+\ni (\omega,z)\mapsto \omega\vartriangleright x+z x\in X.
\]
Since the functional $\mu$ can be written as a linear combination of states, we have $vu\in {}_{L^1(\G)}\mathcal{DF}(L^1(\G),X)$. By assumption, there is $\tilde{x}\in X$ with $vu(\omega)= (\id \otimes \omega)\alpha(\tilde{x})$ for all $\omega \in L^1(\G)$. Given $\omega\in L^ 1(\G)$, we then find
\[\begin{split}
&\quad\;
(\id\otimes \omega)
\bigl(\mathrm{U}
((\id\otimes\mu)\alpha^u(x)\otimes 1)
\mathrm{U}^*\bigr)=
(\id\otimes \omega\otimes  \mu )
(\mathrm{U}_{12} \uU_{13}x_1 \uU_{13}^* \mathrm{U}^*_{12})=
(\id\otimes (\omega\otimes\mu)\Delta^{r,u,\vN} )
(\mathrm{U}(x\otimes 1)\mathrm{U}^*)\\
&=
(\id\otimes(\omega\star \mu))\alpha(x)=
(\omega\star\mu)\vartriangleright x=
vu(\omega)=(\id\otimes\omega)\alpha(\tilde{x})=
(\id\otimes\omega)(\mathrm{U}( \tilde{x}\otimes 1)\mathrm{U}^*),
\end{split}\]
so that $y=(\id\otimes\mu)\alpha^u(x)=\tilde{x}\in X$.
Finally, we have
\[
(\alpha\otimes\id)\alpha^u(x)=
\mathrm{U}_{12} \uU_{13} (x\otimes 1\otimes 1) \uU_{13}^* \mathrm{U}_{12}^*=
(\id\otimes \Delta^{r,u,\vN})
(\mathrm{U}(x\otimes 1)\mathrm{U}^*)=
(\id\otimes \Delta^{r,u,\vN})
\alpha(x)
\]
and consequently $\alpha^u: X\to X\ovot_\mcF C_0^u(\G)^{**}$ is the universal lift of $\alpha$.
\end{proof}

Next we show that the existence of the universal lift is a mild assumption.

\begin{Prop}\label{prop7}
Take $(X,\alpha)\in {}_{L^1(\G)}\NMod^{\|\cdot\|}$ and assume that it is non-degenerate or saturated. Then $\alpha$ admits the universal lift.
\end{Prop}

\begin{proof}
We may assume without loss of generality that $X\subseteq B(\mcH)$ as a weak$^*$-closed subspace and that there is a unitary $\G$-representation $\mathrm{U}\in B(\mcH)\bar\otimes L^{\infty}(\G)$ satisfying $\alpha(x)=\mathrm{U}(x\otimes 1)\mathrm{U}^*$ for all $x\in X$. We then define
\[
\alpha^u: X\ni x\mapsto \uU(x\otimes 1)\uU^*\in B(\mcH)\bar\otimes_{\mcF} C_0^u(\G)^{**}.
\]
It is sufficient to show that $\alpha^u(X)\subseteq X\ovot_\mcF C_0^u(\G)^{**}$. 

Given $\mu \in C_0^u(\G)^*$, we calculate
\begin{align*}
   (\id \otimes \mu) \alpha^u(\omega\rhd x) &= (\id \otimes \mu\otimes \omega)(\uU_{12} \mathrm{U}_{13} x_1 \mathrm{U}_{13} \uU_{12}^*)\\
   &= (\id \otimes \mu\otimes \omega)(\id \otimes \Delta^{u,r,\operatorname{vN}})(\mathrm{U}(x\otimes 1)\mathrm{U}^*) = (\mu\star \omega)\rhd x\in X.
\end{align*}
If $(X, \alpha)$ is non-degenerate, it thus follows that
$(\id \otimes \mu)\alpha^u(X) = (\id \otimes \mu)\alpha^u([L^1(\G)\rhd X]^{w*})\subseteq X.$ In particular, $\alpha^u(X)\subseteq X\ovot_\mcF C_0^u(\G)^{**}$.

Assume now that $(X, \alpha)$ is saturated. We present two arguments to show that the universal lift exists. Consider $(B(\mcH),\beta)\in {}_{L^1(\G)} \NMod^{\|\cdot\|}$ with the action given by $\beta(z)=\mathrm{U}(z\otimes 1)\mathrm{U}^*$. Then $X$ is an $L^1(\G)$-submodule of $B(\mcH)$. As noted after Definition \ref{def1}, since $(X, \alpha)$ is saturated, it is $\G$-complete. For $\omega\in L^1(\G), \mu \in C_0^u(\G)^*$ and $x\in X$, a similar calculation as before gives
\[
\omega\vartriangleright (\id \otimes \mu)\alpha^u(x)=
(\id\otimes\omega)\beta(
(\id\otimes\mu)\alpha^u(x)
)
=(\omega\star \mu)\vartriangleright x\in X,
\]
hence $\G$-completeness of $(X, \alpha)$ implies $(\id \otimes \mu)\alpha^u(x)\in X$, as desired. Alternatively, if $(X, \alpha)$ is saturated, then 
$\alpha(X)\subseteq {}_{L^1(\G)}{\mathcal{DF}}(L^1(\G),X)\subseteq {}_{L^1(\G)}\CB(L^1(\G),X)= \Sat(X, \alpha)= \alpha(X),$
so Proposition \ref{prop5} implies that the universal lift exists.
\end{proof}

\begin{Theorem}\label{LDAP}
    $\G$ has the AP if and only if $L^1(\check{\G})$ has the $L^1(\check{\G})$-DAP. 
\end{Theorem}
\begin{proof}
    Suppose that $\G$ has the AP and take $(X, \alpha)\in {}_{L^1(\check{\G})}\NMod^{\|\cdot\|}$. As shown in the proof of Proposition \ref{prop5}, we automatically have $\alpha(X)\subseteq {}_{L^1(\check{\G})}\mathcal{DF}(L^1(\check{\G}), X)$. Therefore, using Theorem \ref{main'} and Proposition \ref{p:natural}, we find
    $${}_{L^1(\check{\G})}\CB(L^1(\check{\G}), X) = \operatorname{Sat}(X, \alpha) = \alpha(X)\subseteq {}_{L^1(\check{\G})}\mathcal{DF}(L^1(\check{\G}), X),$$
    hence $L^1(\check{\G})$ has the $L^1(\check{\G})$-DAP. 

  Conversely, assume that $L^1(\check{\G})$ has the $L^1(\check{\G})$-DAP. Let $(X, \alpha)\in {}_{L^1(\check{\G})}\NMod^{\|\cdot\|}$ be non-degenerate. The combination of Proposition \ref{prop5} and Proposition \ref{prop7} allows us to conclude that $\alpha(X)= {}_{L^1(\check{\G})}\mathcal{DF}(L^1(\check{\G}),X)$.  Taking weak$^*$-closures, using that $\alpha(X)$ is weak$^*$-closed and using the assumption that $L^1(\check{\G})$ has the $L^1(\check{\G})$-DAP, we see that $(X, \alpha)$ is saturated.
  Condition \eqref{iv} in Theorem \ref{main} then implies that $\G$ has the AP.
\end{proof}

To end this section, we establish the existence of actions which do not admit a universal lift. Consequently (by Proposition \ref{prop7}), they are degenerate and not saturated. In fact, we will prove a stronger statement: $\G$ has the AP if and only if all completely isometric actions of the product quantum group $\hat{\G}\times\hat{\G}$ on dual operator spaces admit the universal lift (Theorem \ref{thm1}). We begin with an auxiliary lemma. Recall the notion of the Fourier-Stieltjes algebra $\B(\G)$ from Subsection \ref{prelim:locallycompactquantumgroups}.

\begin{Lem}\label{lemma5}
Let $\G$ be a locally compact group without the AP. If $0\neq x\in \B(\G)$, then $x\otimes 1\in \B(\G\times \G)$, but $x\otimes 1\notin \overline{\A(\G\times\G)}^{w *}$, where the closure is taken with respect to the weak$^*$-topology of $\M^l_{\cb}(\A(\G\times \G))$.
\end{Lem}

\begin{proof}
We have $x=\lambda^u_{\hat{\G}}(\mu)$ for some $\mu\in C_0^u(\hat{\G})^*$ and $1=\lambda_{\hat{\G}}^u(\hat{\eps}^u)$, where $\hat{\eps}^u\in C_0^u(\hat{\G})^*$ is the counit. Consequently $x\otimes 1=\lambda_{\widehat{\G\times \G}}^u(\mu \otimes \hat{\eps}^u)\in \B(\G\times \G)$ for $\mu \otimes\hat{\eps}^u\in C_0^u(\widehat{\G\times \G})^*=(C_0^u(\hat{\G})\otimes_{\operatorname{max}} C_0^u(\hat{\G}))^*$ \cite[Proposition 13.32]{diss}.

Aiming for a contradiction, let us assume that there is a net $(\rho_i)_{i\in I}$ in $L^1(\widehat{\G\times \G})=L^1(\hat{\G})\wh{\otimes}L^1(\hat{\G})$ such that $\lambda_{\widehat{\G\times \G}}(\rho_i)\xrightarrow[i\in I]{} x\otimes 1$ weak$^*$ in $\M^l_{\cb}(\A(\G\times \G))$. By a standard approximation argument, we can assume that $\rho_i=\sum_{n=1}^{N_i}\rho_{1,n,i}\otimes \rho_{2,n,i}\in L^1(\hat{\G})\odot L^1(\hat{\G})$ for all $i\in I$. Using the description of the predual of $\M^l_{\cb}(\A(\G\times \G))$ obtained in \cite[Proposition 3.9]{DKV24}, we have
\begin{equation}\label{eq7}
\la (\Theta^l(\lambda_{\widehat{\G\times \G}}(\rho_i)) \otimes\id) \hat{Z} - 
(\Theta^l(x\otimes 1) \otimes\id) \hat{Z},\hat{\omega}\ra \xrightarrow[i\in I]{}0
\end{equation}
for all $\hat{Z}\in C_0(\widehat{\G\times \G})\otimes K(\ell^2), \hat{\omega}\in L^1(\widehat{\G\times \G})\wh{\otimes} B(\ell^2)_*$. Since $x\neq 0$, also $\Theta^l(x)\neq 0$ (\cite[Corollary 4.4]{JNR09}) and we can choose $\hat{Z}_1\in C_0(\hat{\G})$ and $\hat{\omega}_1\in L^1(\hat{\G})$ such that $\la \Theta^l(x) \hat{Z}_1,\hat{\omega}_1\ra =\hat{\omega}_1\bigl((\mu\otimes\id)\hat{\Delta}^{u,r}(\hat{Z}_1)\bigr)\neq 0$. Now, for an arbitrary $\hat{Z}_2\in C_0(\hat{\G})\otimes K(\ell^2)$, $\hat{\omega}_2\in L^1(\hat\G)\widehat{\otimes} B(\ell^2)_*$ we deduce from \eqref{eq7}
\[\begin{split}
0&=\lim_{i\in I}\,
\la\sum_{n=1}^{N_i} (\Theta^l(\lambda_{\widehat{\G\times \G}}(\rho_{1,n,i}\otimes \rho_{2,n,i})) \otimes\id) (\hat{Z}_1\otimes \hat{Z}_2) - 
(\Theta^l(x\otimes 1) \otimes\id) (\hat{Z}_1\otimes \hat{Z}_2),\hat{\omega}_1\otimes \hat{\omega}_2\ra\\
&=
\lim_{i\in I}\,
\la\sum_{n=1}^{N_i}
\la \Theta^l(\lambda_{\hat{\G}}(\rho_{1,n,i}))\hat{Z}_1,\hat{\omega}_1\ra \,
 (\Theta^l(\lambda_{\hat{\G}}( \rho_{2,n,i})) \otimes\id) (  \hat{Z}_2) - 
\la \Theta^l(x)\hat{Z}_1,\hat{\omega}_1\ra \,
  \hat{Z}_2 ,  \hat{\omega}_2\ra.
\end{split}\]
This means that the net
\[
\Bigl(\;
\tfrac{1}{\la \Theta^l(x)\hat{Z}_1,\hat{\omega}_1\ra}
\sum_{n=1}^{N_i}
\la \Theta^l(\lambda_{\hat{\G}}(\rho_{1,n,i}))\hat{Z}_1,\hat{\omega}_1\ra \,
 \lambda_{\hat{\G}}( \rho_{2,n,i})
\;\Bigr)_{i\in I}
\]
in $\A(\G)$ converges weak$^*$ to $1\in \M^l_{\cb}(\A(\G))$, which gives a contradiction as $\G$ does not have the AP. 
\end{proof}

\begin{Theorem}\label{thm1}
Let $\G$ be a locally compact quantum group. The following conditions are equivalent:
\begin{enumerate}[noitemsep]
\item $\G$ has the AP.
\item Every $X\in {}_{L^1(\hat\G\times \hat\G)}\NMod^{\|\cdot\|}$ admits the universal lift.
\end{enumerate}
\end{Theorem}

\begin{proof}
$(1)\implies (2)$ If $\G$ has AP, then so does $\G\times \G$ \cite[Proposition 7.22]{DKV24}, hence every $(X,\alpha)\in {}_{L^1( \hat{\G}\times \hat{\G})}\NMod^{\|\cdot\|}$ admits the universal lift by Proposition \ref{prop7} and Theorem \ref{main'}.

$(2)\implies (1)$ If $\G$ is compact, there is nothing to prove. We therefore assume $\G$ is non-compact.
To ease the notation, write $\Hh:=\G\times \G$. Define the set $S:=\{\hat{Z} \in C_0(\hat{\Hh})\otimes K(\ell^2)\mid \|\hat{Z}\|\le 1\}$ and consider the infinite product $\prod_{S} L^{\infty}(\hat{\Hh})\bar\otimes B(\ell^2)$. On each copy of $L^{\infty}(\hat{\Hh})\bar\otimes B(\ell^2)$ we put the amplified (normal, completely isometric) right $\hat{\H}$-action given by $(\Delta_{\hat{\Hh}}\otimes \id)(\cdot)_{132}$. Next, equip $\prod_{S} L^{\infty}(\hat{\Hh})\bar\otimes B(\ell^2)$ with the diagonal $\hat{\H}$-action, 
consider the element $
\Omega=(\hat{Z})_{\hat{Z}\in S}$ and the  weak$^*$-closed subspace
\[
X:=[\Omega,\omega\vartriangleright \Omega \mid \omega\in L^1(\hat{\Hh})]^{w*} \subseteq \prod_{S} L^{\infty}(\hat{\Hh})\bar\otimes B(\ell^2),
\]
together with the restricted action $X\curvearrowleft \hat{\Hh}$. Then $X\in {}_{L^1(\hat{\H})}\NMod^{\|\cdot\|}$.

Choose $ \nu\in L^1(\hat{\G})$ with $0<\|\nu\|<1$ and set $\mu:=(\eps_{\hat{\G}}^u+\nu\Lambda_{\hat{\G}})\otimes \eps_{\hat{\G}}^u\in C_0^u(\hat{\Hh})^*$, where $\Lambda_{\hat{\G}}: C_0^u(\hat{\G})\to C_0(\hat{\G})$ is the reducing map. Since $\G$ is non-compact $\nu\Lambda_{\hat{\G}}\notin \C \eps_{\hat{\G}}^u$. Next, define the element $$\Upsilon=\bigl( (\id\otimes\mu\otimes \id)(\Delta^{r,u}_{\hat{\Hh}}\otimes \id) (\hat{Z})\bigr)_{\hat{Z}\in S}\in \prod_{S} L^{\infty}(\hat{\Hh})\bar\otimes B(\ell^2)$$  and the map
$\phi: L^1(\hat{\Hh})\ni \omega\mapsto 
\omega\vartriangleright \Upsilon=
((\omega\star\mu)\rhd \hat{Z})_{\hat{Z}\in S}
\in X$: note that we use here the equation $( \Delta_{\hat{\Hh}}\otimes \id)\Delta_{\hat{\Hh}}^{r,u}= (\id \otimes \Delta_{\hat{\Hh}}^{r,u})\Delta_{\hat{\Hh}}$. 
Since $L^1(\hat{\Hh})$ is an ideal in $C_0^u(\hat{\Hh})^*$, the map $\phi$ is a well-defined element of ${}_{L^1(\hat{\Hh})}\CB(L^1(\hat{\Hh}),X)$. Similarly as in the proof of Proposition \ref{prop5}, we can write $\phi=vu$, where
\[
u: L^1(\hat{\Hh})\ni \omega\mapsto (\omega\star \mu,0)\in L^1(\hat{\Hh})_+,\quad 
v: L^1(\hat{\Hh})_+\ni (\omega,z)\mapsto \omega\vartriangleright \Omega + z\Omega\in X.
\]
This proves $\phi\in {}_{L^1(\hat{\Hh})}\mathcal{DF}(L^1(\hat{\Hh}),X)$. Since $X\curvearrowleft \hat{\Hh}$ admits the universal lift, Proposition \ref{prop5} allows us to choose $\tilde{x}\in X$ such that $\phi(\omega)=\omega\vartriangleright \tilde{x}$ for $\omega\in L^1(\hat{\Hh})$. Since the action on $\Prod_S L^{\infty}(\hat{\Hh})\bar\otimes B(\ell^2)$ is completely isometric, this implies $\Upsilon=\tilde{x}\in X$ and we conclude that
\[
\Upsilon=\bigl(
(\id\otimes\mu\otimes\id)(\Delta_{\hat{\Hh}}^{r,u}\otimes \id)(\hat{Z})\bigr)_{\hat{Z}\in S}=
{w^* \lim}_{i\in I}\bigl(c_i \hat{Z} + 
\omega_i\vartriangleright \hat{Z}\bigr)_{\hat{Z}\in S}
\]
for some $c_i\in \C, \omega_i\in L^1(\hat{\Hh})$. It follows that $ 
c_i \hat{Z} + 
\omega_i\vartriangleright \hat{Z}\xrightarrow[i\in I]{}(\id\otimes\mu\otimes\id)(\Delta_{\hat{\Hh}}^{r,u}\otimes \id)(\hat{Z})$ weak$^*$ in $L^{\infty}(\hat{\Hh})\bar\otimes B(\ell^2)$ for all $\hat{Z}\in C_0(\hat{\Hh})\otimes K(\ell^2)$. Recall that the dual of $\hat{\Hh}^{\operatorname{op}}$ is $\Hh'$ \cite[Section 4]{KV03}. Consequently, since the Banach $*$-algebras $L^1_{\sharp}(\Hh),L^1_{\sharp}(\Hh')$ are isomorphic via $\theta\mapsto \theta(J_{\vp_\H}\cdot^* J_{\vp_\H})$, we can identify $C_0^u(\hat{\Hh})=C_0^u(\hat{\Hh}^{\operatorname{op}})$. With this observation, we obtain
\[
\bigl\la 
c_i \hat{Z} + 
(\omega_i\otimes\id\otimes\id)
(\Delta_{\hat{\Hh}^{\operatorname{op}}}\otimes\id) (\hat{Z}) -
(\mu\otimes\id\otimes\id)(\Delta^{u,r}_{\hat{\Hh}^{\operatorname{op}}}\otimes \id)(\hat{Z}) , \hat{\omega}\bigr\ra 
\xrightarrow[i\in I]{} 0
\]
for all $\hat{Z}\in C_0(\hat{\Hh})\otimes K(\ell^2)$, $\hat{\omega}\in L^1(\hat{\Hh})\widehat{\otimes}B(\ell^2)_*$. By \cite[Theorem 3.9]{DKV24} it follows that
\[
c_i 1 + \lambda_{\hat{\Hh}^{\operatorname{op}}}(\omega_i)\xrightarrow[i\in I]{\textnormal{weak}^*}\lambda_{\hat{\Hh}^{\operatorname{op}}}^u(\mu)
\]
in $\M^l_{\cb}(\A(\Hh'))$. As $\B(\Hh')\cong C_0^u(\hat{\Hh}^{\operatorname{op}})^*$ is a unital Banach algebra and $\|\lambda^u_{\hat{\Hh}^{\operatorname{op}}}(\nu)\|_{B(\Hh')}<1$, the element $\lambda_{\hat{\Hh}^{\operatorname{op}}}^u(\mu)$ is invertible in $B(\Hh')$. Multiplication of $\M^l_{\cb}(\A(\Hh'))$ is separately weak$^*$-continuous \cite[Proposition 3.3]{DKV24}, hence
\begin{equation}\label{eq8}
c_i \lambda_{\hat{\Hh}^{\operatorname{op}}}^u(\mu)^{-1} + \lambda_{\hat{\Hh}^{\operatorname{op}}}^u(\mu)^{-1} \lambda_{\hat{\Hh}^{\operatorname{op}}}(\omega_i)\xrightarrow[i\in I]{\textnormal{weak}^*} 1.
\end{equation}
Equip $\M^l_{\cb}(\A(\Hh'))$ with the weak$^*$ topology, consider the quotient space $\M^l_{\cb}(\A(\Hh')) \, / \, \overline{\A(\Hh')}^{w*}$ and the canonical quotient map $P: \M^l_{\cb}(\A(\Hh'))\rightarrow \M^l_{\cb}(\A(\Hh')) \, / \, \overline{\A(\Hh')}^{w*}$. Assume by contradiction that $\G$ does not have the AP. Since $\A(\Hh')$ is an ideal in $\B(\Hh')$, \eqref{eq8} implies $c_i P(\lambda_{\hat{\Hh}^{\operatorname{op}}}^u(\mu)^{-1}) \xrightarrow[i\in I]{} P(1)$ with respect to the quotient topology. By \cite[Proposition 4.3, Proposition 7.22]{DKV24}, the quantum group $\Hh'$ does not have the AP, hence $P(1)\neq 0$ and it follows that the net $(c_i)_{i\in I}$ converges in $\C$ to some $c\neq 0$. Using again the convergence \eqref{eq8} we obtain
\[
\lambda_{\hat{\Hh}^{\operatorname{op}}}^u(\mu)^{-1} \lambda_{\hat{\Hh}^{\operatorname{op}}}(\omega_i)\xrightarrow[i\in I]{\textnormal{weak}^*}
 1 - 
c \lambda_{\hat{\Hh}^{\operatorname{op}}}^u(\mu)^{-1} =
\bigl(1- c \lambda_{\hat{\G}^{\operatorname{op}}}^u(\eps_{\hat{\G}}^u+\nu\Lambda_{\hat{\G}})^{-1}\bigr)\otimes 1,
\]
with the weak$^*$-convergence in $\M^l_{\cb}(\A(\Hh'))$. Observe then that $\lambda_{\hat{\Hh}^{\operatorname{op}}}^u(\mu)^{-1} \lambda_{\hat{\Hh}^{\operatorname{op}}}(\omega_i)\in \A(\Hh')$ and $1- c\lambda_{\hat{\G}^{\operatorname{op}}}^u(\eps_{\hat{\G}}^u+\nu\Lambda_{\hat{\G}})^{-1}\in \B(\G')$. Furthermore, as $\nu\neq 0$ and $\G$ is not compact, $1- c\lambda_{\hat{\G}^{\operatorname{op}}}^u(\eps_{\hat{\G}}^u+\nu\Lambda_{\hat{\G}})^{-1}\neq 0$. Consequently Lemma \ref{lemma5} gives us a contradiction. 
\end{proof}

\begin{Rem}
We do not know if we can replace $\hat{\G}\times \hat{\G}$ with $\hat{\G}$ in the formulation of Theorem \ref{thm1}.
\end{Rem}

\section{Exactness and equivariant projectivity}\label{SectionExactness}

Let $\G$ be a locally compact quantum group.
In this section, we investigate homological properties of the crossed product functor
$-\rtimes_\mcF \G$ on appropriate categories of $L^1(\G)$-operator modules. With the observation that $-\rtimes_\mcF \G\cong {}_{L^1(\G)}\CB(B(L^2(\G))_*,-)$ (Proposition \ref{p:natural}) and keeping in mind Proposition \ref{exactness homfunctor}, we are naturally led to study the projectivity of the left $L^1(\G)$-module $B(L^2(\G))_*$.

\subsection{Equivariant projectivity}

We begin with setting up a general context. Consider a Hilbert space $\mcH$ endowed with a left unitary $\G$-representation $\mathrm{U} \in L^\infty(\G)\ovot B(\mcH)$. Then $B(\mcH)$ becomes a right $L^1(\G)$-operator module for
$$x\lhd \omega  := (\omega \otimes \id)(\uu^*(1\otimes x)\uu), \quad x \in B(\mcH), \quad \omega \in L^1(\G),$$
and the predual $B(\mcH)_*$ becomes a left $L^1(\G)$-operator module for 
$$(\omega \rhd \rho)(x):= (\omega \otimes \rho)(\uu^*(1\otimes x)\uu), \quad \omega  \in L^1(\G), \quad \rho\in B(\mcH)_*, \quad  x \in B(\mcH).$$
The case where $\mcH = L^2(\G)$ and $\uu = \ww^{\G}$ is of primary interest.

We start by characterizing relative projectivity of $B(\mcH)_*\in {}_{L^1(\G)}\Mod$. Recall that if $A$ is a completely contractive Banach algebra, then $X\in {}_A\Mod$ is called \emph{relatively $A$-projective} \cite{Cra17}*{Section 2} if there exists a completely contractive $A$-module map $\Phi^+: X\to A_+\hat{\otimes}X$ which is right inverse to the extended module map $A_+\hat{\otimes} X\to X$, and where we regard $A_+\hat{\otimes} X$ as a left $A$-operator module via $a\rhd (s\otimes x)= (a\rhd s)\otimes x$ for $a\in A, s \in A_+$ and $x\in X$. If $X\in {}_A\Mod$ is essential, this is equivalent with the existence of a completely contractive $A$-module map $\Phi: X \to A\hat{\otimes} X$ which is right inverse to the module map $A\hat{\otimes} X\to X$.   

\begin{Prop}\label{relativeprojective}
    Let $\G$ be a locally compact quantum group and let $(\mcH,\uu)$ be a left unitary $\G$-representation. Then $B(\mcH)_*\in {}_{L^1(\G)}\Mod$ is relatively $L^1(\G)$-projective if and only if $\G$ is compact.
\end{Prop}
\begin{proof} Since $B(\mcH)_*$ is essential as an $L^1(\G)$-module, its relative $L^1(\G)$-projectivity is equivalent with the existence of a completely contractive $L^1(\G)$-module map $\Phi: B(\mcH)_*\to L^1(\G)\hat{\otimes}B(\mcH)_* $ which is right inverse to the module map $L^1(\G)\hat{\otimes} B(\mcH)_*\to B(\mcH)_*: \omega\otimes \mu\mapsto (\omega\otimes \mu)\alpha_{\uu}$,  where $\alpha_{\uu}(x)= \uu^*(1\otimes x)\uu$ for $x\in B(\mcH)$. Dualizing, this is equivalent with the  existence of a normal $\G$-equivariant complete contraction $E: (L^\infty(\G)\ovot B(\mcH), \Delta\otimes \id)\to (B(\mcH), \alpha_{\uu})$ satisfying $E\circ \alpha_{\uu}= \id$, so in particular $E$ is ucp. Given a left $\G$-$W^*$-algebra $(M, \alpha)$, the property of admitting a normal, $\G$-equivariant ucp map $E: (L^\infty(\G)\ovot M, \Delta \otimes \id)\to (M, \alpha)$ satisfying $E\circ \alpha= \id_M$ is preserved under $\G$-equivariant Morita equivalence (the same proof strategy as in \cite{DR25}*{Theorem 4.1} applies). Since $(B(\mcH), \alpha_{\uu})$ is $\G$-equivariantly Morita equivalent with $(\C, \tau)$, we conclude that $B(\mcH)_*$ is relatively $L^1(\G)$-projective if and only if there exists a normal $\G$-equivariant state $(L^\infty(\G), \Delta)\to (\C, \tau)$, which is equivalent with the compactness of $\G$ \cite{BT03}*{Proposition 3.10}.
\end{proof}

\begin{Prop}\label{prop2}
Let $\uu\in L^{\infty}(\G)\bar\otimes B(\mcH)$ be a unitary representation of $\G$ on $\mcH$. Assume that there is
\begin{itemize}[noitemsep]
\item a family $(p_i)_{i\in I}$ of non-zero orthogonal projections in $B(\mcH)$ with $\sum_{i\in I} p_i=1$,
\item a net $(\zeta_\lambda)_{\lambda\in \Lambda}\subseteq \mcH$ satisfying $\|\zeta_\lambda\|=1$ and $\zeta_\lambda\xrightarrow[\lambda\in \Lambda]{}0$ weakly.
\end{itemize} 
Assume furthermore that
\begin{equation}\label{eq1}
\|(1\otimes p_i) \uu^* (\xi\otimes \zeta_\lambda)\|\xrightarrow[\lambda\in\Lambda]{}0
\end{equation}
for all $i\in I$ and all $\xi\in L^2(\G)$. Then $B(\mcH)_*$ is not projective in ${}_{L^1(\G)}\Mod$.
\end{Prop}

\begin{proof}
We will prove the equivalent statement that $B(\mcH)\in \NMod_{L^1(\G)}$ is not weak$^*$-injective. The argument that follows is a modification of the proof contained in \cite{ER88}*{Section 3}, which shows that $B(\mathcal{K})\in {}_{\C}\NMod$ is not weak$^*$-injective when $\mcK$ is an infinite-dimensional Hilbert space.
Let $\mcF$ be the directed set of non-empty, finite subsets of $I$, ordered by inclusion (note that the assumption implies that $I$ is infinite). Given $F\in \mcF$, let us write $\mcH_F:= \bigoplus_{i\in F} p_i \mcH\subseteq \mcH$ and $p_F:=\sum_{i\in F}p_i$, which we see as a map $\mcH\rightarrow \mcH_F$. We then define the unital, completely positive map
    $$\Phi: B(\mcH)\to \prod_{F \in\mcF} B(L^2(\G)\otimes \mcH_F):  x \mapsto 
\bigl(
(1\otimes p_F)\uu^*(1\otimes x)\uu (1\otimes p_F^*)\bigr)_{F\in \mcF}.$$
   Endowing the Hilbert space $L^2(\G)\otimes \mcH_F$ with the $\G$-representation $\ww^{}\otimes 1_{\mcH_F}\in L^\infty(\G)\ovot B(L^2(\G)\otimes \mcH_F)$, we can view the von Neumann algebra $B(L^2(\G)\otimes \mcH_F)$ as an object of $\operatorname{NMod}_{L^1(\G)}$. Then $\prod_{F\in \mcF} B(L^2(\G)\otimes \mcH_F) \in \operatorname{NMod}_{L^1(\G)}$ for the diagonal action.
    
The map $\Phi$ is $L^1(\G)$-linear. Indeed, for $\omega\in L^1(\G)$ and $x\in B(\mcH)$ we have
\begin{align*}
\Phi(x)\lhd \omega&=\bigl(
(\omega\otimes\id\otimes\id)
\bigl(
(\ww^{*}\otimes 1_{\mcH_F})
(1\otimes 1\otimes p_F)\uu^*_{23}
x_3 \uu_{23}
(1\otimes 1\otimes p_F^*)
(\ww^{}\otimes 1_{\mcH_F})
\bigr)
\bigr)_{F\in\mcF}\\
&=
\bigl(
(\omega\otimes\id\otimes\id)
\bigl(
(1\otimes 1\otimes p_F)
\uu_{23}^* \uu_{13}^*
x_3 \uu_{13}
\uu_{23}
(1\otimes 1\otimes p_F^*)
\bigr)
\bigr)_{F\in\mcF}\\
&=
\bigl(
(1\otimes p_F)
\uu^* 
(1\otimes x\lhd\omega)
\uu
( 1\otimes p_F^*)
\bigr)_{F\in\mcF}=\Phi(x\lhd \omega).    
\end{align*}

It is easy to see that $\Phi$ is completely isometric and weak$^*$-continuous. Let $X=\Phi(B(\mcH))$ be the image of $\Phi$, which is an $L^1(\G)$-submodule of $\prod_{F\in \mcF}B(L^2(\G)\otimes \mcH_F)$. Since $\Phi$ is a complete isometry, $X$ is weak$^*$-closed and the inverse map $\Phi^{-1}: X\rightarrow B(\mcH)$ is weak$^*$-continuous.

Fix a unit vector $\zeta\in \mcH$ and define rank one maps $x_\lambda=|\zeta\ra\la \zeta_\lambda|\in B(\mcH)$. One easily checks that $(x_\lambda)_{\lambda\in \Lambda}$ converges to $0$ in the $\sigma$-strong topology, but not in the $\sigma$-strong$^*$ topology. Next we claim that $\Phi(x_\lambda)\xrightarrow[\lambda\in \Lambda]{} 0$ in the $\sigma$-strong$^*$-topology, where we regard $\prod_{F\in \mcF}B(L^2(\G)\otimes \mcH_F)$ as represented on $\bigoplus_{F\in \mcF}(L^2(\G)\otimes \mcH_F)$. As this net is uniformly bounded, it is enough to fix a finite set $\emptyset\neq F_0\subseteq I$, a vector $\xi\in L^2(\G)\otimes \mcH_{F_0}\subseteq \bigoplus_{F\in \mcF} (L^2(\G)\otimes \mcH_{F})$ and show that
\begin{equation}\label{eq2}
\|\Phi(x_\lambda)\xi\| \xrightarrow[\lambda\in\Lambda]{}0,\quad 
\|\Phi(x_\lambda^*)\xi\| \xrightarrow[\lambda\in \Lambda]{}0.
\end{equation}
For $\eps>0$, choose $\tilde{\xi}\in L^2(\G)\odot \mcH$ such that $\|\uu\xi-\tilde{\xi}\|\le \eps$. Then we have
\[\begin{split}
\|\Phi(x_\lambda) \xi\|&=
\bigl\|(\delta_{F,F_0} (1\otimes p_{F_0})\uu^* (1\otimes x_{\lambda}) \uu \xi )_{F\in \mcF}\bigr\|=
\bigl\|(1\otimes p_{F_0})\uu^*
\bigl(  (\id\otimes \la \zeta_{\lambda}|)\uu\xi \otimes \zeta\bigr)\bigr\|\\
&\le
\|(\id\otimes\la \zeta_\lambda|)\uu\xi\|\le 
 \eps+\|(\id\otimes \la \zeta_\lambda|) \tilde{\xi} \|\xrightarrow[\lambda\in\Lambda]{}\eps.
\end{split}\]
Since $\eps$ is arbitrary, this proves the first convergence in \eqref{eq2}. Next,
\[\begin{split}
&\quad\;
\|\Phi(x_\lambda^*) \xi\|=
\bigl\|(\delta_{F,F_0} (1\otimes p_{F_0})\uu^*(1\otimes  x_\lambda^*)\uu\xi )_{F\in \mcF}\bigr\|=
\bigl\|
(1\otimes p_{F_0}) \uu^* \bigl(
(\id\otimes \la \zeta|)\uu\xi \otimes \zeta_\lambda\bigr)\bigr\|\\
&=\bigl\|\sum_{i\in F_0} 
(1\otimes p_i )
\uu^* \bigl(
(\id\otimes \la \zeta|)\uu\xi \otimes \zeta_\lambda\bigr)\bigr\|\le
\sum_{i\in F_0}
\bigl\| 
(1\otimes p_i )
\uu^* \bigl(
(\id\otimes \la \zeta|)\uu\xi \otimes \zeta_\lambda\bigr)\bigr\|\xrightarrow[\lambda\in \Lambda]{}0,
\end{split}\]
where the last convergence holds by \eqref{eq1}.

We conclude that $\Phi^{-1}: X \to B(\mcH)$ is not continuous for the (restriction of) the $\sigma$-strong$^*$-topology. By the normal Wittstock-Stinespring decomposition (see e.g.~\cite{Pi20}*{Theorem 1.60}), the weak$^*$-continuous, $L^1(\G)$-linear complete isometry $\Phi^{-1}: X \to B(\mcH)$ does not admit any normal completely bounded extension $\prod_{F\in \mcF } B(L^2(\G)\otimes \mcH_F)\to B(\mcH)$. Indeed, if such an extension exists, then it is $\sigma$-strong$^*$-continuous and then $\Phi^{-1}$ would be $\sigma$-strong$^*$-continuous. Therefore, $B(\mcH)\in \operatorname{NMod}_{L^1(\G)}$ is not weak$^*$-injective.
\end{proof}

The assumption of Proposition \ref{prop2} is satisfied in several situations of interest.

\begin{Prop}\label{prop3}
Let $\G$ be a locally compact quantum group and let $\uu\in L^\infty(\G)\bar\otimes B(\mcH)$ be a unitary representation. Write $\uu=(\id\otimes\pi)\wW^{\G}$ for a non-degenerate representation $\pi: C_0^u(\hat{\G})\rightarrow B(\mcH)$. Assume that either of the following conditions hold:
\begin{enumerate}[noitemsep]
\item $C_{\uu}'\subseteq B(\mcH)$ is infinite-dimensional, where $C_{\uu}:= [(\omega\otimes\id)\uu\mid \omega\in L^1(\G)]$,
\item there is a strictly positive, self-adjoint operator $B$ on $\mcH$ and $c> 0$ such that $\dim (1_{[c,\infty)}(B)\mcH)=\infty$ and $\pi(\hat{\tau}^u_t(x))=B^{it} \pi(x) B^{-it}$ for all $x\in C_0^u(\hat{\G})$ and all $t\in\R$.
\end{enumerate}
Then $B(\mcH)_*$ is not projective in ${}_{L^1(\G)}\Mod$.
\end{Prop}

\begin{proof}
(1) Assume first that $C_{\uu}'\subseteq B(\mcH)$ is infinite-dimensional. Then we can choose a sequence $(p_n)_{n=1}^\infty$ of (non-zero) pairwise orthogonal projections in $C'_{\uu}$. By adding to this sequence $\I-\sum_{n=1}^{\infty} p_n$ if needed, we may assume that $\sum_{n=1}^{\infty} p_n=\I$. For $n\ge 1$, choose a unit vector $\zeta_n$ in $p_n \mcH$. Then the sequence $(\zeta_n)_{n=1}^\infty$ converges weakly to $0$. Take $m\ge 1, \xi\in L^2(\G)$. As
$(1\otimes p_m){\uu}^*(\xi \otimes \zeta_n)=
(1\otimes p_m)\uu^*(\xi \otimes p_m\zeta_n)=0$
for $n>m$, Proposition \ref{prop2} implies that $B(\mcH)_*$ is not $L^1(\G)$-projective.

(2) Let $\{\xi_j\}_{j\in J}$ be an orthonormal basis in $\mcH$, chosen such that\footnote{An orthonormal basis with this property always exists. Indeed, let $E_B$ be the spectral measure of $B$ \cite[Theorem 10.10]{Sol18}. Then $\bigl\{E_B\bigl(\,\bigl]\tfrac{1}{n+1},\tfrac{1}{n}\bigl]\cup \bigl[n,n+1\bigr[\,\bigr)\bigr\}_{n=1}^{\infty}$ is a family of pairwise orthogonal projections which sum up to $\I$; it is enough to choose orthonormal bases in their images and combine them together to an orthonormal basis of $\mcH$.} $\xi_j\in \bigcup_{r>1} 1_{[r^{-1},r]}(B) \mcH$ for each $j\in J$ (note that the assumption implies that $J$ is infinite). Set $p_j$ to be the rank one projection $|\xi_j\ra\la \xi_j|$. Next, let $\mcF$ be the directed set of finite, non-empty subsets of $J$, ordered by inclusion. Choose a sequence $(\zeta_n)_{n\in\N}$ of unit vectors with $1_{[c,\infty)}(B)\zeta_n=\zeta_n$ and $\zeta_n\xrightarrow[n\to\infty]{}0$ weakly. Fix $j\in J, \xi\in L^2(\G)$. It is left to show the convergence \eqref{eq1}.

Take $\eps>0$ and choose a vector $\tilde{\xi}\in L^2(\G)\odot \mcH$ such that $\|\tilde{\xi} - \uu^*(J_{\widehat{\vp}}\xi\otimes B^{-1/2}\xi_j)\|\le \eps$.
Let $S$ and $R$ be the antipode and unitary antipode on $\G$, respectively. Recall that $S=R\tau_{-i/2}$ \cite{VD14}*{Definition 2.23}. Since $(\tau_t\otimes \hat{\tau}^u_t)\wW^{\G}=\wW^{\G}$ for $t\in\R$, we also have $(\tau_t\otimes \oon{Ad}(B^{it}))\uu=\uu$. Next, using Remark 9.1 and Proposition 9.4 in \cite{K01} we have
\[\begin{split}
&\quad\;
(\id\otimes \omega_{\xi_j,\zeta_n})\uu^*=
S( (\id\otimes\omega_{\xi_j,\zeta_n})\uu)=
R \tau_{-i/2} ((\id\otimes\omega_{\xi_j,\zeta_n})\uu)\\
&=
R ((\id\otimes \omega_{B^{-1/2} \xi_j, B^{-1/2} \zeta_n})\uu)=
J_{\widehat{\vp}} (\id\otimes\omega_{B^{-1/2} \zeta_n, B^{-1/2} \xi_j})(\uu^*) J_{\widehat{\vp}},
\end{split}\]
where $\widehat{\vp}$ is the left Haar integral of $\hat{\G}$. With this preparation, we can calculate
\[\begin{split}
&\quad\;
\|(1\otimes p_j)\uu^*(\xi\otimes \zeta_n)\|=
\|(\id\otimes \la \xi_j| ) \uu^* (\xi\otimes \zeta_n)\|=
\|(\id\otimes \omega_{\xi_j,\zeta_n})(\uu^*) \xi \|=
\|
J_{\widehat{\vp}} (\id\otimes\omega_{B^{-1/2} \zeta_n, B^{-1/2} \xi_j})(\uu^*) J_{\widehat{\vp}}
 \xi \|\\
&=
\|(\id\otimes \la B^{-1/2} \zeta_n | )
\uu^* (J_{\widehat{\vp}}\xi \otimes B^{-1/2}\xi_j)\| \le c^{-1/2} \eps + 
\bigl\|(\id\otimes \la B^{-1/2} 1_{[c,\infty)}(B)\zeta_n | ) 
\tilde{\xi}\bigr\|\\
&=
c^{-1/2} \eps + 
\bigl\|(\id\otimes \la \zeta_n | ) 
(1\otimes B^{-1/2} 1_{[c,\infty)}(B) )\tilde{\xi}\bigr\|\xrightarrow[n\to\infty]{} c^{-1/2} \eps.
\end{split}\]
Since $\eps>0$ was arbitrary, Proposition \ref{prop2} ends the proof.
\end{proof}

We immediately obtain two corollaries.

\begin{Cor}
Let $\uu\in L^{\infty}(\G)\bar\otimes B(\mcH)$ be a unitary representation of a locally compact quantum group $\G$. If the scaling group of $\G$ is trivial and $\dim \mcH=\infty$, then $B(\mcH)_*$ is not projective in ${}_{L^1(\G)}\Mod$.
\end{Cor}

\begin{proof}
It is a direct consequence of Proposition \ref{prop3} (2), using that $\G$ has trivial scaling group if and only if $\hat{\G}$ has trivial scaling group \cite{KV03,KS23}.
\end{proof}

\begin{Cor}\label{projectivitytraceclass}
Let $\G$ be a locally compact quantum group and consider $L^2(\G)$ equipped with the left regular representation $\ww$. Then $B(L^2(\G))_*$ is projective in ${}_{L^1(\G)}\Mod$ if and only if $\G$ is finite.
\end{Cor}

\begin{proof}
If $B(L^2(\G))_*$ is projective in ${}_{L^1(\G)}\Mod$, then by Proposition \ref{prop3} (1) the von Neumann algebra
\[
{C_{\ww}'}=
[(\omega\otimes\id)\ww\mid \omega\in L^1(\G)]'=L^{\infty}(\check{\G})
\]
is finite-dimensional, hence $\G$ is finite.

Conversely, assume that $\G$ is finite, then $B(L^2(\G))_*$ is projective in ${}_{\C}\Mod$ \cite{Bl}*{Proposition 3.6, Proposition 3.7}. By Proposition \ref{relativeprojective}, $B(L^2(\G))_*$ is relatively projective in ${}_{L^1(\G)}\Mod$. Then \cite{Cra17}*{Proposition 2.2} shows that $B(L^2(\G))_*$ is projective in ${}_{L^1(\G)}\Mod$.
\end{proof}

\begin{Rem} The notion of injectivity in ${}_{L^1(\G)}\Mod$ behaves very differently than the notion of weak$^*$-injectivity in ${}_{L^1(\G)}\NMod$. Let us point out some striking differences.
\begin{itemize}[noitemsep]
    \item A right $\G$-$W^*$-algebra $(M, \alpha)$ is injective in ${}_{L^1(\G)}\Mod$ if and only if $(M\bar{\rtimes}_\alpha \G, \alpha^\rtimes)$ is injective in ${}_{L^1(\check{\G})}\Mod$ \cite{DR26}*{Theorem 5.1}. However, if the $\G$-$W^*$-algebra $(M, \alpha)\in {}_{L^1(\G)} \NMod$ is weak$^*$-injective, then it need not be true that $(M\bar{\rtimes}_\alpha \G, \alpha^\rtimes)\in {}_{L^1(\check{\G})} \NMod$ is weak$^*$-injective. Indeed, consider an infinite discrete quantum group $\G$ and $M= \ell^\infty(\G)$ with its natural translation action $\alpha= \Delta$. It follows from \cite{Cra17}*{Theorem 5.14} that $(\ell^\infty(\G), \Delta)\in {}_{\ell^1(\G)}\NMod$ is weak$^*$-injective. However, it follows from Corollary \ref{projectivitytraceclass} and from \eqref{crossed} that $(B(\ell^2(\G)), \check{\Delta}_r)\cong_{\check{\G}} (\ell^\infty(\G)\bar{\rtimes}_\Delta \G, \Delta^\rtimes) \in {}_{L^1(\check{\G})}\NMod$ is not weak$^*$-injective. We do not know of any counterexample for the converse implication in $\NMod$.
    \item In \cite{DR25}*{Theorem 4.1}, it was shown that injectivity of $\G$-$W^*$-algebras in ${}_{L^1(\G)}\Mod$ is preserved by equivariant Morita equivalence. This is no longer true for weak$^*$-injectivity. As an easy example, note that $(B(L^2(\G)), \Delta_r)$ is equivariantly Morita equivalent with $(\C, \tau)$. However, $(\C, \tau)\in {}_{L^1(\G)}\operatorname{NMod}$ is weak$^*$-injective if and only if $\G$ is compact, so any infinite compact quantum group $\G$ gives a counterexample by Corollary \ref{projectivitytraceclass}.
\end{itemize}

\end{Rem}

Under the additional assumption of regularity on $\G$, we can fully characterize when $B(\mcH)_*$ is $L^1(\G)$-projective using a completely different approach. For ease of reference, we will now work with \emph{right} $\G$-$W^*$-algebras instead. We need to do some preliminary work first. The following definition is a weak$^*$-continuous completely bounded version of  \cite{DCDR24}*{Definition 6.6}, which introduces the notions of amenability and inner amenability for $\G$-$W^*$-algebras.

\begin{Def}
    Let $(M, \alpha)$ be a right $\G$-$W^*$-algebra. 
    \begin{itemize}
        \item We call $(M, \alpha)$ \emph{W$^*$CB-amenable} if there exists a normal $\G$-equivariant completely bounded projection
    $(M\ovot L^\infty(\G), \id \otimes \Delta)\to (\alpha(M), \id \otimes \Delta)$.
    \item We call $(M, \alpha)$ \emph{W$^*$CB-inner amenable} if there exists a normal $\G$-equivariant completely bounded projection
    $(M\bar{\rtimes}_\alpha \G, \alpha_{\ad})\to (\alpha(M), \id \otimes \Delta).$
    \end{itemize}
\end{Def}

The proof of the following result follows as described at the end of \cite{DCDR24}*{Section 6}. We give the details for the convenience of the reader.
\begin{Lem}\label{duality} Let $(M, \alpha)$ be a right $\G$-$W^*$-algebra.
    \begin{enumerate}[noitemsep]
        \item If $(M, \alpha)$ is W$^*$CB-amenable, then $(M\bar{\rtimes}_\alpha \G, \alpha^\rtimes)$ is W$^*$CB-inner amenable.
        \item If $(M, \alpha)$ is W$^*$CB-inner amenable, then $(M\bar{\rtimes}_\alpha \G, \alpha^\rtimes)$ is W$^*$CB-amenable.
    \end{enumerate}
\end{Lem}
\begin{proof}
    (1) By assumption, there exists a normal $\G$-equivariant completely bounded map
    $$E: (M\ovot L^\infty(\G), \id \otimes \Delta)\to (M, \alpha)$$ such that $E\circ \alpha = \id_M.$
    Taking crossed products and keeping in mind the $\check{\G}$-equivariant isomorphism \eqref{crossed}, we obtain the $\check{\G}$-equivariant normal completely bounded projection
$$(M\ovot B(L^2(\G)), \id \otimes \check{\Delta}_r) \cong ((M \ovot L^\infty(\G))\bar{\rtimes}_{\id \otimes \Delta} \G, (\id \otimes \Delta)^\rtimes) \to (M\bar{\rtimes}_\alpha \G, \alpha^\rtimes).$$
Further composing with the $\check{\G}$-equivariant Takesaki-Takai isomorphism \eqref{TTconcrete}, we obtain the normal $\check{\G}$-equivariant completely bounded map
$((M\bar{\rtimes}_\alpha \G)\bar{\rtimes}_{\alpha^\rtimes} \check{\G}, (\alpha^\rtimes)_{\ad})\to (M\bar{\rtimes}_\alpha \G, \alpha^\rtimes)$ which maps $\alpha^\rtimes(z)\mapsto z$ for $z\in M\bar{\rtimes}_\alpha \G$. Thus, $(M\bar{\rtimes}_\alpha \G, \alpha^\rtimes)$ is W$^*$CB-inner amenable.

(2) By assumption, there is a $\G$-equivariant normal completely bounded map
$E: (M\bar{\rtimes}_\alpha \G, \alpha_{\ad})\to (M, \alpha)$ such that $E \circ \alpha = \id_M.$
Taking crossed products and composing with the $\check{\G}$-equivariant isomorphism \eqref{crossed2}, we obtain the $\check{\G}$-equivariant normal completely bounded map
$$((M\bar{\rtimes}_\alpha \G)\ovot L^\infty(\check{\G}), \id_{M\bar{\rtimes}_\alpha \G}\otimes \check{\Delta})\cong ((M\bar{\rtimes}_\alpha \G)\bar{\rtimes}_{\alpha_{\ad}} \G, (\alpha_{\ad})^\rtimes)\to (M\bar{\rtimes}_\alpha \G, \alpha^\rtimes)$$
which maps $\alpha^\rtimes(z)\mapsto z$ for all $z\in M\bar{\rtimes}_\alpha\G$. Thus, $(M\bar{\rtimes}_\alpha \G, \alpha^\rtimes)$ is W$^*$CB-amenable.
\end{proof}

\begin{Theorem}\label{regularprojectivity}
    Let $\G$ be a regular locally  compact quantum group and let $(\mcH,\uu)$ be a unitary right $\G$-representation. Then $B(\mcH)_*$ is $L^1(\G)$-projective if and only if $\mcH$ is finite-dimensional and $\G$ is compact.
\end{Theorem}
\begin{proof} Assume that $B(\mcH)_*$ is $L^1(\G)$-projective. Consider the induced $\G$-action $B(\mcH)\stackrel{\alpha_\uu}\curvearrowleft\G$ given by $\alpha_\uu(x)= \uu(x\otimes 1)\uu^*$ for $x\in B(\mcH)$, then $(B(\mcH), \alpha_\uu)\in {}_{L^1(\G)}\NMod$ is $L^1(\G)$-injective. From \cite[Proposition 3.10, Proposition 6.2]{DR26} we deduce that $\G$ is amenable, hence $L^\infty(\check{\G})$ is injective (= semidiscrete) \cite[Theorem 3.3]{BT03}.

    Since $(B(\mcH), \alpha_\uu)\in {}_{L^1(\G)}\NMod$ is also weak$^*$-injective, it clearly is W$^*$CB-amenable. Hence, by Lemma \ref{duality} and \eqref{innercrossed}, the $\check{\G}$-$W^*$-algebra $(B(\mcH)\ovot L^\infty(\check{\G}), \id \otimes \check{\Delta})\cong (B(\mcH)\bar{\rtimes}_{\alpha_{\uu}} \G, \alpha_{\uu}^\rtimes)$ is W$^*$CB-inner amenable. Consequently\footnote{If $(M, \alpha)$ is a right $\G$-$W^*$-algebra and $N$ is a von Neumann algebra, then W$^*$CB-inner amenability of $(N\ovot M, \id \otimes \alpha)$ implies W$^*$CB-inner amenability of $(M, \alpha)$. Indeed, take $E: ((N\ovot M)\bar{\rtimes}_{\id \otimes \alpha} \G, (\id \otimes \alpha)_{\ad})\to (N\ovot M, \id \otimes \alpha)$, a normal $\G$-equivariant CB map satisfying $E((\id \otimes \alpha)(z))= z$ for all $z\in N\ovot M$, fix a normal state $\omega\in N_*$ and consider the map $F: (M\bar{\rtimes}_\alpha \G, \alpha_{\ad})\to (M, \alpha): z \mapsto (\omega\otimes \id)(E(1\otimes z))$. It is  a normal $\G$-equivariant CB map satisfying $F(\alpha(m))= m$ for all $m\in M$.}, $(L^\infty(\check{\G}), \check{\Delta})$ is W$^*$CB-inner amenable, so there is a normal $\check{\G}$-equivariant completely bounded projection
    $B(L^2(\G))\cong L^\infty(\check{\G})\bar{\rtimes}_{\check{\Delta}} \check{\G} \to L^\infty(\check{\G}).$
By the same argument, since $(B(\mcH), \alpha_\uu)$ is also W$^*$CB-inner amenable, it follows that $(L^\infty(\check{\G}), \check{\Delta})$ is W$^*$CB-amenable, so there exists a normal $\check{\G}$-equivariant completely bounded map
    $$E: (L^\infty(\check{\G})\ovot L^\infty(\check{\G}), \id \otimes \check{\Delta})\to (L^\infty(\check{\G}), \check{\Delta}), \quad E\circ \check{\Delta}= \id.$$
    We then have all the necessary ingredients to make the argument in the proof of \cite{Cra17}*{Theorem 5.14} work (with $\G$ replaced by $\check{\G}$; it is here that the regularity assumption is used). Hence, we conclude that $\G$ must be compact. By the Peter-Weyl theorem \cite{NT14}*{Theorem 1.5.4}, $\uu$ decomposes as a direct sum $\bigoplus_{i \in I}\uu_i$ of irreducible $\G$-representations. Proposition \ref{prop3} forces $I$ to be finite. Since irreducible representations of a compact quantum group are necessarily finite-dimensional, it follows that $\mcH$ is finite-dimensional.

  Conversely, assume that $\G$ is compact and $\mcH$ is finite-dimensional. By (the right module version of) Proposition \ref{relativeprojective}, $B(\mcH)_*\in \Mod_{L^1(\G)}$ is relatively projective. On the other hand, since $\mcH$ is finite-dimensional, $B(\mcH)_*$ is projective as an operator space \cite{Bl}*{Proposition 3.7}. The $L^1(\G)$-projectivity of $B(\mcH)_*$ then follows from \cite{Cra17}*{Proposition 2.2}.
\end{proof}

To the authors' knowledge, Theorem \ref{regularprojectivity} is new even in the classical group setting.

The authors conjecture that the regularity assumption in Theorem \ref{regularprojectivity} is superfluous. To illustrate this, we verify the conjecture directly for the (non-regular) locally compact quantum group $E_q(2)$, the $q$-deformation of the double cover of the motion group. Its Pontryagin dual will be denoted by $\wh{E}_q(2)$. For a more detailed account on these quantum groups we refer to the articles \cite{W91, B92, J03, KS23} and the book \cite{Timm08}.

We will mostly follow the notation of \cite{J03}. Fix $0<q<1$. Consider the Hilbert space $\ell^2(\Z)$ with the standard orthonormal basis $\{e_k\}_{k\in\Z}$, the unitary operator $s e_k=e_{k+1}$ and the strictly positive, self-adjoint multiplication operator $me_k=q^k e_{k}$ on its natural domain. Next, define
\[
a:=m^{-1/2}\otimes m,\quad 
b:=m^{1/2}\otimes s
\]
as (unbounded) operators on $\ell^2(\Z)\otimes \ell^2(\Z)$. The operators $a,a^{-1},b$ are affiliated with $C_0(\widehat{E}_q(2))$, we have
\[
\Delta_{\wh{E}_q(2)}(a)=a\otimes a,\quad 
\Delta_{\wh{E}_q(2)}(b)=
a\otimes b \dot{+}b\otimes a^{-1}
\]
and one can describe $C_0(\widehat{E}_q(2))$ in terms of $a$ and $b$ (see \cite[Proposition 2.5.34, Theorem 2.5.21]{J03}, see also \cite[Theorem 1.2]{W91} for the meaning of a morphism applied to an affiliated element). 

We recall also that both $E_q(2), \widehat{E}_q(2)$ are strongly amenable, second countable and non-regular \cite[Corollary 2.8.25, Theorem 3.2.7, Theorem 3.2.29]{J03}.

\begin{Prop}\label{nonregularexample}
Let $0<q<1$ and let $\uu\in L^{\infty} ( E_q(2))\bar\otimes B(\mcH)$ be a unitary representation of $E_q(2)$. Then $B(\mcH)_*$ is not projective in ${}_{L^1(E_q(2))}\Mod$.
\end{Prop}

\begin{proof}
Assume first that $\dim \mcH=+\infty$. As explained in Subsection \ref{prelim:locallycompactquantumgroups}, the unitary representation $\uu$ of $E_q(2)$ corresponds to a non-degenerate representation $\pi: C_0(\widehat{E}_q(2))\rightarrow B(\mcH)$.
By \cite[Proposition 2.6.50]{J03}, \cite[Section 5.1]{KS23} we have
\[
\hat{\tau}_t(a)=a,\quad \hat{\tau}_t(b)=q^{-2it}b=a^{-2it} b a^{2it},
\]
hence $\hat{\tau}_t=\operatorname{Ad}(a^{-2it})$ for $t\in \R$ and $a^{-2it}\in M(C_0(\widehat{E}_q(2) ))$. We conclude that 
\begin{equation}\label{eq3}
\pi (\hat{\tau}_t(x))=
\pi( a^{-2it}) \pi(x) \pi(a^{2it})
\end{equation}
for $t\in \R$. Since $\pi$ is non-degenerate, the one-parameter group $(\pi(a^{-2it}))_{t\in\R}$ is strongly continuous, hence Stone's theorem gives us a strictly positive, self-adjoint operator $B$ on $\mcH$ such that $\pi(a^{-2it})=B^{it}$ (in fact $B=\pi(a^{-2})$). We will show that $\dim (1_{\left[1,\infty\right)}(B)\mcH)=\infty$.

Since $C_0(\widehat{E}_q(2))$ is separable, the representation $\pi$ decomposes into a direct sum $\pi\cong \bigoplus_{i\in I} \pi_i$, where each $\pi_i$ is a non-degenerate representation of $C_0(\widehat{E}_q(2))$ on a separable Hilbert space $\mcH_i$. The irreducible representations of $C_0(\widehat{E}_q(2))$ are completely classified \cite[Proposition 4.3.2]{J03}: every irreducible representation of $C_0(\widehat{E}_q(2))$ is unitarily equivalent to exactly one of the following:
\begin{itemize}[noitemsep]
\item $(k\in \Z)$ character $\hat{\omega}_k: C_0(\widehat{E}_q(2))\rightarrow \C$ given by $\hat{\omega}_k(a)=q^{k/2}, \hat{\omega}_k(b)=0$,
\item $(k\in \Z)$ irreducible representation $\hat{\rho}_k: C_0(\widehat{E}_q(2))\rightarrow B(\ell^2(\Z))$ given by $\hat{\rho}_k(a)=m,\, \hat{\rho}_k(b)=q^k s$,
\item $(k\in \Z)$ irreducible representation $\hat{\rho}'_k: C_0(\widehat{E}_q(2))\rightarrow B(\ell^2(\Z))$ given by $\hat{\rho}'_k(a)=q^{1/2}m,\, \hat{\rho}'_k(b)=q^k s$.
\end{itemize}
In particular, the image of every irreducible representation contains a non-zero compact operator. Indeed, it is clear for $\hat{\omega}_k$, while for $\hat{\rho}_k,\hat{\rho}_{k}'$ we see that the image of $f(a)\in C_0(\widehat{E}_q(2))$ is compact, where $f: \R_{\ge 0}\rightarrow\left[0,1\right]$ is a continuous function satisfying $f(1)=1=f(q^{1/2})$ and $f(t)=0$ when $|1-t|\ge  1-q$. Consequently $C_0(\widehat{E}_q(2))$ is type I \cite[Theorem 4.6.4]{Sakai98}.

The spectrum of $C_0(\widehat{E}_q(2))$ is countable, hence the disintegration theorem \cite[Theorem 8.6.6]{DixmierC} simplifies and for each $i\in I$ we can find multiplicities $M^1_{i,k},M^2_{i,k},M^3_{i,k}\in \{0,1,2,\dots,\infty\}$ together with a unitary equivalence
  \[
\pi_i\cong\bigoplus_{k\in\Z} M^1_{i,k} \cdot \hat{\omega}_k\oplus
 \bigoplus_{k\in\Z} M^2_{i,k} \cdot \hat{\rho}_k\oplus
 \bigoplus_{k\in\Z} M^3_{i,k} \cdot \hat{\rho}'_k.
 \]
Since $\dim \mcH=\infty$, either $\pi$ has an infinite-dimensional subrepresentation build from characters $\hat{\omega}_k$ or one of the representations $\hat\rho_k, \hat{\rho}'_k$ as a direct summand. In each of these cases, we see that the operator (unitarily equivalent with) $B$ has an infinite-dimensional spectral subspace corresponding to $\left[1,\infty\right)$. Proposition \ref{prop3} ends the proof in the infinite-dimensional case.\\

Consider now the case $\dim \mcH<+\infty$. The decomposition described above implies that $\uu$ admits a one-dimensional subrepresentation. More concretely, there is a one-dimensional subspace $\mcK\subseteq \mcH$ such that the orthogonal projection $P\in B(\mcH)$ onto $\mcK$ is an intertwiner of $\uu$. Define the normal state $\theta_2: B(\mcH)\ni T\mapsto PTP\in \C$, where we identify $PB(\mcH)P=B(\mcK)=\C$. It is $L^1(E_q(2))$-equivariant: for $T\in B(\mcH)$ and $\omega\in L^1(E_q(2))$,
\[
\theta_2\bigl( (\omega\otimes\id)( U^*(1\otimes T)U)\bigr)=
(\omega\otimes \id)( U^*(\I\otimes PTP)U)=(\omega\otimes \id)( \I\otimes PTP  )=\omega(\I)\theta_2(T).
\]
Assume by contradiction that $B(\mcH)_*$ is $L^1(E_q(2))$-projective; then $B(\mcH)$ is weak$^*$-injective in $\NMod_{L^1(E_q(2))}$ and the inclusion $\C\hookrightarrow B(\mcH)$ gives rise to a normal, unital, completely bounded $L^1(E_q(2))$-equivariant map $\theta_1\colon L^{\infty}(E_q(2))\rightarrow B(\mcH)$. The composition $m=\theta_2\theta_1: L^{\infty}(E_q(2))\rightarrow \C$ is a normal, unital, invariant functional. Since $m(\I)=1$, there is $x\in C_0(E_q(2))$ such that $m(x)\neq 0$. We obtain 
\[
m(x)1=(\id\otimes m)\Delta(x)\in C_0(E_q(2)),
\]
which is a contradiction.
\end{proof}

\subsection{Exactness of the crossed product functor}

We will need the following elementary lemma, generalizing \cite{Cra17}*{Proposition 3.1}, together with a general homological observation that leverages Proposition \ref{flatness}.

\begin{Lem}\label{selfinduced}
   Let $(X, \alpha)\in \NMod_{L^1(\G)}^{\|\cdot\|}$ be saturated. Then $X_*\in {}_{L^1(\G)}\Mod$ is induced, i.e.\ the canonical map $m: L^1(\G)\hat{\otimes}_{L^1(\G)}X_* \to X_*$ is a completely isometric surjection. 
\end{Lem}
\begin{proof} Let $q: L^1(\G)\pten X_*\to L^1(\G) \pten_{L^1(\G)} X_*$ be the canonical complete quotient map. Dualizing, we obtain the complete isometry $q^*: (L^1(\G) \pten_{L^1(\G)} X_*)^* \to (L^1(\G)\pten X_*)^*$. Its image consists exactly of those bounded functionals $\omega$ on $L^1(\G)\pten X_*$ satisfying the condition $\omega((f\star g)\otimes h) = \omega(f \otimes ((g\otimes h)\circ \alpha))$ for all $f,g \in L^1(\G)$ and $h\in X_*$. Under the canonical completely isometric identification $(L^1(\G)\hat{\otimes} X_*)^*\cong L^\infty(\G)\ovot_\mcF X$, this image therefore corresponds exactly to $\{z\in L^\infty(\G)\ovot_\mcF X\mid (\Delta \otimes \id)(z) = (\id \otimes \alpha)(z)\}= \alpha(X)$. Consequently, we see that $m^* = \alpha: X \cong (X_*)^*\to (L^1(\G)\hat{\otimes}_{L^1(\G)} X_*)^*\cong \alpha(X)$ is a completely isometric isomorphism. 
\end{proof}
In statement of the next result, ${}_A \NMod^{\|\cdot\|}$ refers to those $Y\in {}_A\NMod$ for which the canonical map
$$j_Y:Y\ni y\mapsto (a\mapsto a\rhd y)\in\mathcal{CB}(A,Y)$$
is a complete isometry.
\begin{Prop}\label{genflat}
     Let $A$ be a completely contractive Banach algebra with a complete quotient multiplication $m_A:A\pten A\twoheadrightarrow A$ such that $A$ is flat in ${}_A\Mod$,  and let $M \in {}_A \Mod$. If the functor ${}_A \CB(M,-)$ is $1$-exact on ${}_A \NMod^{\|\cdot\|}$ and $M$ is induced, then $M$ is flat in ${}_A \Mod$.
\end{Prop}

\begin{proof} By Proposition \ref{flatness}, it suffices to show that ${}_A \CB(M,-)$ is $1$-exact on ${}_A \NMod$. To that end, let  
$$\begin{tikzcd}
0 \arrow[r] & X \arrow[r, "\iota"] & Y \arrow[r, "\phi"] & Z \arrow[r] & 0
\end{tikzcd}$$
be a 1-exact sequence in ${}_A \NMod$. By flatness of $A$, the sequence
$$\begin{tikzcd}
0 \arrow[r] & Z_*\pten_A A \arrow[r, "\phi_*\ten \id_A"] & Y_*\pten_A A \arrow[r, "\iota_*\ten \id_A"] & X_*\pten_A A \arrow[r] & 0
\end{tikzcd}$$
is 1-exact in $\Mod_A$ (as it is 1-exact in $\Mod_{\mathbb{C}}$ and the arrows are morphisms in $\Mod_A$). Hence, so too is its dual sequence
$$\begin{tikzcd}
0 \arrow[r] & (X_*\pten_A A)^*\arrow[r] & (Y_*\pten_A A)^* \arrow[r] & (Z_*\pten_A A)^*\arrow[r] & 0
\end{tikzcd}$$
Under the completely isometric identification $\mathcal{CB}(A,(X_*\pten_A A)^*)=((X_*\pten_A A)\pten A)^*$ and associativity of module tensor products \eqref{associativity}, one sees that $j_{(X_*\pten_A A)^*}=(\id_{X_*}\ten_A m_A)^*$. Therefore, as $A$-module tensor products of complete quotient morphisms are complete quotient maps, $(X_*\pten_A A)^*\in {}_A \NMod^{\|\cdot\|}$, and similarly, $(Y_*\pten_A A)^*,(Z_*\pten_A A)^*\in {}_A \NMod^{\|\cdot\|}$. Applying 1-exactness of ${}_A \CB(M,-)$ on ${}_A \NMod^{\|\cdot\|}$, and taking the predual of the resulting sequence yields the 1-exact sequence
$$\begin{tikzcd}
0 \arrow[r] & (Z_*\pten_A A)\pten_A M \arrow[r] & (Y_*\pten_A A)\pten_A M \arrow[r] & (X_*\pten_A A)\pten_A M\arrow[r] & 0
\end{tikzcd}.$$
Appealing to associativity of module tensor products \eqref{associativity} and the fact that $M$ is induced, the sequence
$$\begin{tikzcd}
0 \arrow[r] & Z_*\pten_A M \arrow[r, "\phi_*\ten \id_M"] & Y_*\pten_A M \arrow[r, "\iota_*\ten \id_M"] & X_*\pten_A M \arrow[r] & 0
\end{tikzcd}$$
is 1-exact. Equivalently, 
$$\begin{tikzcd}
0 \arrow[r] & {}_A \CB(M,X) \arrow[r, "\iota_*"] & {}_A \CB(M,Y) \arrow[r, "\phi_*"] & {}_A \CB(M,Z) \arrow[r] & 0
\end{tikzcd}$$
is 1-exact.
\end{proof}

\begin{Theorem}\label{exactnness} Let $\G$ be a locally compact quantum group.
    \begin{enumerate}[noitemsep]
        \item $-\rtimes^\mcF \G$ is $1$-exact on ${}_{L^1(\G)}\Mod$ if and only if $\G$ is finite.
        \item $-\rtimes^\mcF \G$ is $1$-exact on ${}_{L^1(\G)}\NMod$ if and only if $\G$ is amenable.
        \item If $\G$ is amenable, then $-\rtimes^\mcF \G$ is $1$-exact on ${}_{L^1(\G)}\NMod^{\|\cdot\|}$. The converse holds if, in addition, $\hat{\G}$ is amenable.
    \end{enumerate}
\end{Theorem}

\begin{proof} (1) Since $-\rtimes^\mcF \G\cong {}_{L^1(\G)}\CB(B(L^2(\G))_*,-)$ (Proposition \ref{p:natural}), this follows by the combination of Proposition \ref{exactness homfunctor} and Corollary \ref{projectivitytraceclass}.

(2) This follows by Proposition \ref{flatness} and the fact that $B(L^2(\G))_*$ is $L^1(\G)$-flat if and only if $B(L^2(\G))$ is injective in ${}_{L^1(\G)}\Mod$, which is equivalent with amenability of $\G$ (see \cite[Theorem 5.5]{CN16} or \cite{DR26}*{Proposition 6.2}, c.f. Corollary \ref{cor2}).

    (3) The first implication follows immediately from the previous point. Suppose that $\hat{\G}$ is amenable and  that $-\rtimes^\mcF \G\cong {}_{L^1(\G)} \CB(B(L^2(\G))_*, -)$ is 1-exact on ${}_{L^1(\G)}\NMod^{\|\cdot\|}$. By Lemma \ref{selfinduced}, $B(L^2(\G))_*$ is induced as an $L^1(\G)$-operator module. By \cite[Theorem 5.1]{Cra17}, $L^1(\G)$ is $L^1(\G)$-flat, and since it carries a complete quotient multiplication, Proposition \ref{genflat} implies the $L^1(\G)$-flatness of $B(L^2(\G))_*$, and hence amenability of $\G$.
\end{proof}

\textbf{Acknowledgements}.  A portion of the work of JC was completed while in residence at Institut Mittag-Leffler in Djursholm, Sweden, during the Operator Algebras and Quantum Information Program 2026. He is grateful for the hospitality and support from the Swedish Research Council under grant no. 2021-06594. JC was partially supported by NSERC Discovery Grant RGPIN-2023-05133. The work in this research was initiated when JDR was supported by Fonds voor Wetenschappelijk Onderzoek (Flanders),
via an FWO Aspirant fellowship, grant 1162524N. The work of JK was partially supported by FWO grant 1246624N. 
\bibliographystyle{plain}
\bibliography{bibliography}

\end{document}